\documentclass[11pt,a4paper]{article}
\usepackage{amsthm,amsfonts,latexsym,amsmath,amssymb}
\usepackage[mathscr]{eucal}
\usepackage{graphicx,color}
\usepackage{comment}
\usepackage{hyperref}
\hypersetup{colorlinks=true,linkcolor=magenta,citecolor=red,filecolor=magenta,urlcolor=blue}
\usepackage{epstopdf}
\usepackage[capitalise]{cleveref}
\usepackage[cm]{fullpage}
\usepackage{enumitem}

\usepackage{tikz}
\usepackage{tikz-3dplot}

\theoremstyle{plain} 
\newtheorem{proposition}{Proposition}[section] 
\newtheorem{theorem}[proposition]{Theorem}

\newtheorem{lemma}[proposition]{Lemma}
\newtheorem{corollary}[proposition]{Corollary}
\newtheorem{remark}[proposition]{Remark}
\theoremstyle{definition}
\newtheorem{definition}[proposition]{Definition}

\newtheorem{ques}[proposition]{Question}

\newcommand{\mc}{\mathcal{}}
\newcommand{\tl}{\tilde}

\newcommand{\z}{\mathbb{Z}}

\newcommand{\al}{\alpha}

\newcommand{\ens}[1]{ \left\{#1\right\} }

\newcommand{\ip}[2]{\left<{#1},{#2}\right>}
\newcommand{\norm}[1]{\left\|#1\right\|}

\newcommand{\cplx}{\mathbb{C}}

\newcommand{\rea}{\mathbb{R}}

\newcommand{\pa} {\partial}
\newcommand{\N}{\mathbb{N}}

\newcommand{\rd}{{\rm d}}
\newcommand{\Aalb}{\boldsymbol{\mathcal{A}}_\alpha}
\newcommand{\AalAb}{\boldsymbol{\mathcal{A}}_{\alpha,A}}

\numberwithin{equation}{section}

\newcommand{\dv}{\textnormal{div}}

\catcode`\@=11
\def\downparenfill{$\m@th\braceld\leaders\vrule\hfill\bracerd$}
\def\overparen#1{\mathop{\vbox{\ialign{##\crcr\crcr
				\noalign{\kern0.4ex}
				\downparenfill\crcr\noalign{\kern0.4ex\nointerlineskip}
				$\hfil\displaystyle{#1}\hfil$\crcr}}}\limits}
\catcode`\@=12

\begin{document}

\title{\LARGE \bf%
	Boundary null controllability of a class of 2-$d$ degenerate parabolic PDEs\thanks{This work has received support from UNAM-DGAPA-PAPIIT grant IN117525 (Mexico).}
}

\author{
Víctor Hernández-Santamaría\thanks{V. Hern\'andez-Santamar\'ia is supported by the program ``Estancias Posdoctorales por México para la Formación y Consolidación de las y los Investigadores por México'' of CONAHCYT (Mexico). He also received support from Project CBF2023-2024-116 of CONAHCYT and by UNAM-DGAPA-PAPIIT grants IA100324 and IN102925 (Mexico). } 
\and 
Subrata Majumdar\thanks{Subrata Majumdar is supported by the UNAM Postdoctoral Program (POSDOC).}
\and 
Luz de Teresa}

\maketitle
	\begin{abstract}
		This article deals with the boundary null controllability of some degenerate parabolic equations posed on a square domain, presenting the first study of boundary controllability for such equations in multidimensional settings. The proof combines two classical techniques: the method of moments and the Lebeau-Robbiano strategy. A key novelty of this work lies in the analysis of boundary control localized on a subset of the boundary where degeneracy occurs.
		Furthermore, we establish a Kalman rank condition as a full characterization of boundary controllability for coupled degenerate systems. The results are extended to $N$-dimensional domains, and some other potential extensions, along with open problems, are discussed to motivate further research in this area.
		\end{abstract}

		
	\tableofcontents

\section{Introduction and main results}
\subsection{System under study}\label{system}
Let us consider the following degenerate parabolic equation in a rectangular domain $\Omega=(0,1)\times(0,1)$
\begin{equation}\label{DCP_sc}
	\begin{cases}
		\partial_t u=\text{div}\left(D\nabla u\right) &  \text{ in } (0,T)\times \Omega,\\
	u(0)=u_0, & \text{ in } \Omega,
	\end{cases}
\end{equation}
with the \text{ boundary conditions } \eqref{bd3} defined below. 
 The matrix function $D: \overline\Omega\mapsto \mc M_{2\times 2}(\mathbb{R})$ is given by
\begin{equation*}
	D(x,y)=\begin{pmatrix}
		x^{\alpha_1}& 0\\
		0& y^{\alpha_2}
	\end{pmatrix},
\end{equation*}
where $\alpha=(\alpha_1, \alpha_2) \in [0,2]\times[0,2]$, and $u_0$ is the initial data that lies in the functional space $H^{-1}_{\alpha}(\Omega)$ (see \Cref{sec:wp} below for details).

Let us denote $\Gamma_1=\{0\}\times[0,1],  \Gamma_2=[0,1]\times\{0\}, \Gamma_3=\{1\}\times[0,1],   \Gamma_4=[0,1]\times\{1\}$, so that $\partial \Omega=\Gamma_1\cup\Gamma_2\cup\Gamma_3\cup\Gamma_4.$ We also denote ${\Sigma}_i=(0,T)\times\Gamma_i$, $i=1,2,3,4,$ $\Sigma=\cup_{i=1}^{4}{\Sigma}_i$, $\Sigma_{ij}={\Sigma}_i\cup {\Sigma}_j$, ${\Sigma}_{ijk}={\Sigma}_i\cup{\Sigma}_j\cup {\Sigma}_k.$
The boundary conditions depending on the degenerate parameters $\alpha_1, \alpha_2$ are the following
\begin{equation}\label{bd3}
	\begin{cases}
	u(t)=\mathbf{1}_{\gamma} q(t) \text{ on } \Sigma \quad \text{if } \alpha_1, \alpha_2 \in [0,1),\\
 (D\nabla u(t))\nu=\mathbf{1}_{\gamma} q(t) \text{ on }  \Sigma_1, \text{ and }	u(t)=0 \text{ on } \Sigma_{34} \text{ and }(D\nabla u(t))\nu=0 \text{ on } \Sigma_2 \quad \text{if } \alpha_1,\alpha_2\in [1,2],\\
	u(t)=\mathbf{1}_{\gamma} q(t) \text{ on } \Sigma_1, u(t)=0 \text{ on } \Sigma_{34} \text{ and } (D\nabla u(t))\nu=0 \text{ on } \Sigma_2 \quad \text{if } (\alpha_1,\alpha_2) \in [0,1)\times [1,2],\\
(D\nabla u(t))\nu=\mathbf{1}_{\gamma} q(t) \text{ on }  \Sigma_1, \text{ and } 	u(t)=0 \text{ on } \Sigma_{234} \quad \text{if } (\alpha_1,\alpha_2)\in [1,2]\times [0,1),
	\end{cases}
\end{equation}
 where $\gamma=\{0\}\times \omega,$ $\omega\subset (0,1)$ is a nonempty open set, $q\in  L^2(0,T;L^2(\pa\Omega))$ is a control function (to be determined) which acts on the system through the boundary supported at $\gamma,$ $\nu=\nu(x)$ is the outward unit normal to $\Omega$ at the points $x\in \pa \Omega$ (see \Cref{fig:region}). 
 

In this paper, we analyze the problem of controllability for \eqref{DCP_sc}--\eqref{bd3}, the goal being to steer the state to a null final target by a suitable choice of the control function $q$. More precisely,
\begin{definition}
System \eqref{DCP_sc}--\eqref{bd3} is said to be null controllable at time $T$ if for any $u_0\in H^{-1}_{\alpha}(\Omega)$, there is a control $q\in  L^2(0,T;L^2(\pa\Omega))$ such that the associated state $u$ satisfies $u(T,\cdot)=0$ in $\Omega$. 
\end{definition}

There is extensive literature on the boundary null controllability of parabolic equations, initiated by the seminal works of Fattorini and Russell employing the moment method (see \cite{FR1,FR2}). Over the years, these foundational contributions have been significantly extended, leading to numerous important results for a wide variety of problems and approaches; see, for instance, \cite{FCGBdeT10,JMPA11,AKBGBdT16,BFM20}. It is important to highlight that these developments have predominantly focused on the one-dimensional setting, as the original formulation of the moment method in \cite{FR1} was restricted to this case. However, by combining the moment method with the classical approach introduced by Lebeau and Robbiano \cite{LR}, this framework has been successfully extended in \cite{AB2014} to address certain classes of multidimensional problems.

In parallel, a considerable amount of work has also been devoted to the boundary controllability of degenerate parabolic equations in the one-dimensional setting (see \Cref{sec:soa} for a comprehensive list). However, the question of boundary controllability in higher dimensions remains largely unexplored. The aim of this paper is to contribute toward addressing this challenge.

 \begin{figure}[htbp!]
    \centering
    \begin{tikzpicture}[scale=3] 
        \draw[thin, -] (-0.2,0) -- (0,0); 
        \draw[thin, ->] (1,0) -- (1.2,0) node[below] {$x$}; 

        \draw[thin, -] (0,-0.2) -- (0,0); 
        \draw[thin, ->] (0,1) -- (0,1.2) node[left] {$y$}; 

        \draw[densely dashed, very thick] (0,0) -- (1,0) node[midway, below] {$\Gamma_2$};
        \draw[thick] (1,0) -- (1,1) node[midway, right] {$\Gamma_3$};
        \draw[thick] (1,1) -- (0,1) node[midway, above] {$\Gamma_4$};
        \draw[densely dashed, very thick] (0,1) -- (0,0) node[midway, left] {$\Gamma_1$};

        \draw[ultra thick, red] (0,0.4) -- (0,0.75); 
        \node[left] at (0.2,0.575) {$\omega$}; 

        \node[below left] at (0,0) {0};
        \node[below] at (1,0) {1};
        \node[left] at (0,1) {1};
        \node at (0.6, 0.7) {\Large$\Omega$}; 
    \end{tikzpicture}
\caption{The domain $\Omega$ for equation \ref{DCP_sc}, with the operator degenerating along the dashed lines $\Gamma_1$ and $\Gamma_2$. The red region, denoted by $\omega$, represents the control set, which is active at the boundary where the system degenerates.}
    \label{fig:region}
\end{figure}
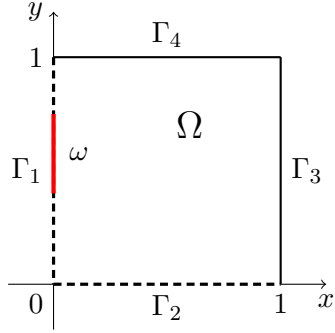

\subsection{State of the art}\label{sec:soa}
The study of the controllability of degenerate parabolic problems has received significant attention throughout recent years, resulting in a vast and rich body of literature. 
 We provide an overview of the most relevant areas and highlight key works related to these topics.
\begin{itemize}
\item Boundary control of scalar 1-$d$ degenerate problems: \cite{gueye2014exact,MRR16,moyano2016flatness,cannarsa2017cost,CMV20,BLR23,galo2024boundary,galo2023boundary}.
\item Internal control, Carleman estimates and related problems: \cite{cannarsa2008carleman,cannarsa10null,cannarsa2012unique,ABHELFM19,FdeT20,araujo2022boundary}.
\item Higher dimensions (internal control): \cite{cannarsa2016global,FA19}.
\item Higher-order degenerate problems: \cite{Gal24}. 
\item Systems with physical applications: \cite{MRV03,Fra18,Flo20,FY21}.
\item Bilinear controls for degenerate problems: \cite{CF11,Flo14,CMU23}.
\item Degenerate hyperbolic models: \cite{GL16,CF24}.
\end{itemize}
Among these, we focus on the works most closely related to our study, as they provide essential context and motivation for our contributions. We begin our bibliographic review by discussing works that examine the boundary controllability properties of one-dimensional degenerate parabolic equations. Our initial focus will be on the weakly degenerate case, that is,
\begin{equation}\label{intro_oned}
	\begin{cases}
		\pa_t u =\pa_{x}(x^{\alpha_1}\pa_{x}u) & \text{ in }   (0,T) \times (0,1),\\
		u(t,0)=h_1(t), \quad u(t,1)=0 &\text{ in }   (0,T),\\
		u(0,x)=u_0(x)  & \text{ in }   (0,1),
	\end{cases}
\end{equation}
where $0\leq \alpha_1<1$. In the paper \cite{gueye2014exact}, using the transmutation technique, the author investigated the null controllability of equation \eqref{intro_oned} with a control $h_1\in L^2(0,T)$ exerted on the left end of the Dirichlet boundary, where degeneracy occurs. Later, the authors in \cite{cannarsa2017cost} considered the same equation \eqref{intro_oned} and using the moment method they explored an explicit control cost with respect to $\alpha_1$ and $T$ when the control $h_1\in H^1(0,T)$. Recently, the work \cite{galo2023boundary} extended the aforementioned results in the following sense: the authors considered a more general degenerate/singular parabolic equation with a drift term and studied the boundary null controllability, providing an estimate on the control cost using an \( L^2(0,T) \)-control acting at either endpoint of the spatial domain.
 Boundary approximate controllability of \eqref{intro_oned} has been established using Carleman estimates in \cite{cannarsa2012unique}.

Next, let us consider the following strongly degenerate parabolic equation
\begin{equation}\label{intro_oned1}
	\begin{cases}
		\pa_t u =\pa_{x}(x^{\alpha_1}\pa_{x}u) & \text{ in }   (0,T) \times (0,1),\\
		(x^{\alpha_1}\pa_x u)(t,0)=0, \quad 
		u(t,1)=h(t) &\text{ in }   (0,T),\\
		u(0,x)=u_0(x)  & \text{ in }   (0,1),
	\end{cases}
\end{equation}
where $1\leq\alpha_1<2.$ The work \cite{moyano2016flatness} addressed the boundary null controllability of \eqref{intro_oned1} using the flatness strategy. Later, \cite{CMV20} extended this result by analyzing the explicit control cost using a control $h$ in the space \( H^1(0,T) \) acting at the right endpoint. More recently, \cite{galo2024boundary} established that the null controllability for \eqref{intro_oned1} can be achieved using an \( L^2(0,T) \)-control when \( 1 < \alpha_1 < 2 \). They also considered the control problem
\begin{equation*}
	\begin{cases}
		\pa_t u =\pa_{x}(x^{\alpha_1}\pa_{x}u) & \text{ in }   (0,T) \times (0,1),\\
		(x^{\alpha_1}\pa_x u)(t,0)=h(t), \quad 
		u(t,1)=0 &\text{ in }   (0,T),\\
		u(0,x)=u_0(x)  & \text{ in }   (0,1),
	\end{cases}
\end{equation*}
and proved the boundary null controllability issue with an \( L^2(0,T) \)-control acting on the degenerate point.

For the full range $0 \leq \alpha_1 < 2$, \cite{cannarsa2008carleman,cannarsa10null} investigated the null controllability of a degenerate parabolic equations with distributed controls and boundary controls acting at the Dirichlet boundary point $\{x=1\}$. Similarly, \cite{araujo2022boundary} addressed this question and established the null controllability of a degenerate heat equation via boundary control at $\{x=1\}$, based on an asymptotic result regarding the interior control problem for both weakly and strongly degenerate cases.

Moving on to two-dimensional models, we now focus on the relevant works that address internal control for degenerate parabolic systems, as boundary controllability in higher dimensions has not, to our knowledge, been studied, with our work being the first to tackle this problem. In \cite{FA19}, the authors investigated the null controllability of the scalar degenerate parabolic equation 
\begin{equation*}
	\begin{cases}
		\partial_t u-\text{div}\left(D\nabla u\right)+bu=g\chi_{\omega}, &  \text{ in } (0,T)\times \Omega,\\
		u(0)=u_0, & \text{ in } \Omega,
	\end{cases}
\end{equation*}
with the homogeneous boundary conditions \eqref{bd3} (with $q=0$) by means of a localized interior control $g$. In this context, we refer to the book \cite{cannarsa2016global} which provides a more general framework regarding the geometry of the domains and the class of degenerate operators, offering many useful tools and results. 

Due to its geometry, the boundary conditions, and the functional framework, the work \cite{FA19} is the closest to ours, as we are addressing the same problem but in a setting with boundary controls.

\subsection{Statement of the main results}\label{sec_main_result}
The main results of this paper are the following.
\begin{theorem}\label{main theorem_scalar}
	Let $T>0$ and $\alpha_i \in [0,2)\setminus\{1\}$ for $i = 1, 2$. For any $u_0\in H^{-1}_{\alpha}(\Omega)$, there exists a control $q\in L^2(0,T;L^2(\partial \Omega))$ such that system \eqref{DCP_sc}--\eqref{bd3} satisfies $u(T)=0$.
\end{theorem}
\begin{remark}\label{remark_extensions}
	The following remarks are in order.
	\begin{itemize}
		 \item  We establish the controllability result in \Cref{main theorem_scalar} using a boundary control applied at the degenerate side, $\gamma = \{0\} \times \omega$, where $\omega \subset (0,1)$. A similar controllability result can be demonstrated when the control is exerted on the other sides of the square. In \Cref{non_deg}, we briefly discuss the case where the control is applied on the opposite side, that is, $\gamma = \{1\} \times \omega$, with $\omega \subset (0,1)$.
		
		\item \Cref{main theorem_scalar} has been stated in a unit square. However, our approach can be generalized to the domain 
		$
		\Omega = \underbrace{(0,1)\times (0,1)\times \ldots \times (0,1)}_{N \text{ times}} \subset \mathbb{R}^N.
		$
		We provide a detailed discussion in \Cref{higher_d}.	
		\item Note that the case $\alpha_1 = 1$ is excluded from the statement of \Cref{main theorem_scalar}. This exclusion arises because the functional framework for well-posedness and control requires slight modifications. Nevertheless, the result remains valid in this case, as detailed in \Cref{alphaeq1}.
			\end{itemize}
\end{remark}
The result in \Cref{main theorem_scalar} can be extended to a system of coupled degenerate parabolic equations. To do so,
let  $m,n\in\mathbb N$, $A\in \mathcal{L}(\cplx^n)$ and $B\in\mathcal{L}(\cplx^m;\cplx^n)$ be given matrices. We denote $u:= \left(\begin{array}{cccc}u_1 & u_2 & \cdots & u_n\end{array}\right)^\top$ as the state and define the spatial operator 
 \begin{equation*}\mathbf{D}u:=\left(\begin{array}{ccc}\text{div}\left(D\nabla u_1\right) & \cdots & \text{div}\left(D\nabla u_n\right)\end{array}\right)^\top.
 	\end{equation*} Let us consider the degenerate coupled system
\begin{equation}\label{DCP}
	\begin{cases}
		\partial_t u={\mathcal{I}_n \mathbf{D} u+A u} &  \text{ in } (0,T)\times \Omega,\\ 
		u(0)=u_0, & \text{ in } \Omega,
	\end{cases}
\end{equation}
with the following boundary conditions
\begin{equation}\label{bd}
	\begin{cases}
		u(t)=\mathbf{1}_{\gamma} Bq(t) \text{ on } \Sigma,\,\, \text{ if } \alpha_1, \alpha_2 \in [0,1)\\
		 \mathbf{P} u(t)=\mathbf{1}_{\gamma} Bq(t) \text{ on }  \Sigma_1, \text{ and }	u(t)=0 \text{ on } \Sigma_{34} \text{ and } \mathbf{P} u(t)=0 \text{ on } \Sigma_2,\,\, \text{ if } \alpha_1,\alpha_2\in [1,2]\\
		u(t)=\mathbf{1}_{\gamma} Bq(t) \text{ on } \Sigma_1, u(t)=0 \text{ on } \Sigma_{34} \text{ and }  \mathbf{P} u(t)=0 \text{ on } \Sigma_2,\,\, \text{ if } (\alpha_1,\alpha_2) \in [0,1)\times [1,2]\\
		 \mathbf{P} u(t)=\mathbf{1}_{\gamma} Bq(t) \text{ on }  \Sigma_1, \text{ and } 	u(t)=0 \text{ on } \Sigma_{234}, \,\, \text{ if } (\alpha_1,\alpha_2)\in [1,2]\times [0,1),
	\end{cases}
\end{equation}
where $\mathbf{P}(u):=\left(\begin{array}{ccc}D\nabla u_1\cdot \nu & \cdots & D\nabla u_n\cdot \nu\end{array}\right)^\top$, $D$, $\nu$ are defined in \Cref{system}, $q\in L^2(0,T;L^2(\partial \Omega)^m)$ is the control, and $\mathcal I_n$ is the $n\times n$ identity matrix. 
\begin{remark}
	To ensure clarity in exposition, we have used the same notation for both the vector and scalar forms of the state 
	$u$, initial data $u_0$, and control $q$. Their meaning can be understood from the context or position in the formulation.
\end{remark}
The problem amounts to apply $m$ controls to a system of $n$ variables exerted on the boundary of the spatial domain of the equations. The most interesting case occurs when $m \ll n$, as it represents controlling many equations with a minimal number of controls.
 
To state the controllability result for \eqref{DCP}--\eqref{bd}, let us fix some notations. Let $(\lambda_{\alpha_1,k}, \phi_{\alpha_1,k})_{k\in \N}$ be the eigenelement of the following eigenvalue problem
\begin{equation}\label{eigen eqn1}
	\begin{cases}
		-(x^{\alpha_1}\phi')'(x)=\lambda \phi(x) \qquad  x\in (0,1),\\
		\begin{cases}
			\phi(0)=0 &\text{ if } 0\leq \alpha_1<1\\
			(x^{\alpha_1}\phi')(0)=0 &\text{ if } 1\leq\alpha_1<2
		\end{cases},\\
		\phi(1)=0.
	\end{cases}
\end{equation}
For $k\geq 1$, we denote the following matrices
\begin{equation}\label{matrix}
	\mathbf{B}_k:=\begin{pmatrix}
		B\\.\\.\\.\\B
	\end{pmatrix}\in \mathcal{L}(\cplx^m;\cplx^{nk}), \,\, \mathbf{L}_k:=\begin{pmatrix}
		L_1& 0&\cdot \cdot \cdot&0\\0& L_2&\cdot \cdot \cdot&0\\
		\vdots&\cdot \cdot \cdot&\ddots&\vdots\\
		0&\cdot \cdot \cdot&0&L_k
	\end{pmatrix}\in \mathcal{L}(\cplx^{nk}),
\end{equation}
where $	L_k:=-\lambda_{\alpha_1,k}\mathcal I_n+A\in \mathcal{L}(\cplx^n).$
We introduce the Kalman matrix associated with the pair $(\mathbf{L}_k, \mathbf{B}_k)$
\begin{equation*}
	\boldsymbol{{\mathcal{{K}}}}_k:=[\mathbf{L}_k | \mathbf{B}_k]=\left[\mathbf{B}_k,\, \mathbf{L}_k\mathbf{B}_k,\,  \mathbf{L}^{2}_k \mathbf{B}_k,...,\mathbf{L}^{nk-2}_k\mathbf{B}_k,\, \mathbf{L}^{nk-1}_k \mathbf{B}_k\right] \in\mathcal{L}(\cplx^{mnk};\cplx^{nk}).
\end{equation*}
Let us state the following characterization for the controllability of the coupled degenerate system \eqref{DCP}--\eqref{bd} with $\alpha_i \in [0,2)\setminus\{1\}$ for $i = 1, 2$.
\begin{theorem}\label{main theorem}
	Let $T>0$ and $u_0\in H^{-1}_{\alpha}(\Omega)^n$. For any given $A\in \mathcal{L}(\cplx^n)$ and $B\in\mathcal{L}(\cplx^m;\cplx^n)$, there exists a control $q\in L^2(0,T;L^2(\partial\Omega)^m)$ such that the system \eqref{DCP}--\eqref{bd} satisfies $u(T)=0$ if and only if 
	\begin{equation}\label{kalman}
		\text{rank } 	\boldsymbol{\mathbf{\mathcal{K}}}_k=\text{rank } [\mathbf{L}_k | \mathbf{B}_k]=nk, \, \forall k\geq 1.
	\end{equation}
\end{theorem}
We provide a detailed proof of \Cref{main theorem} in \Cref{lr sc}, from which \Cref{main theorem_scalar} directly follows.

\subsection{Strategy of the proof of \Cref{main theorem}}\label{sec:strat}
The approach transforms the 2-$d$ problem into an infinite set of 1-$d$ problems by decomposing the solution of the associated adjoint system into eigenfunction expansions along one direction. The main idea is to combine the boundary control result in 1-$d$ with the Lebeau-Robbiano strategy \cite{LR} to establish the null controllability of the 2-$d$ problem. This strategy was developed for the first time in \cite{AB2014}.

Below, we outline the structure of the proof.
\begin{itemize}
	\item We begin by employing the moment method to establish a boundary null controllability result of the one-dimensional coupled degenerate parabolic system
	\begin{equation}\label{oned}
		\begin{cases}
			\pa_t w =\pa_{x}(x^{\alpha_1}\pa_{x}w)+A w & \text{ in }   (0,T) \times (0,1),\\
			\begin{cases}w(t,0)=Bh(t) & \text{ if } 0\leq \alpha_1<1\\
				(x^{\alpha_1}\pa_x w)(t,0)=Bh(t) & \text{ if } 1 < \alpha_1<2
			\end{cases} &\text{ in }   (0,T),\\
			w(t,1)=0 &\text{ in }   (0,T)\\
			w(0,x)=w_0(x)  & \text{ in }   (0,1),
		\end{cases}
	\end{equation}
in the space $H^{-1 }_{ \alpha_1}(0,1)^n$ (see \Cref{sec_fn} for details)	with a control satisfying the cost $C_T=Ce^{C/T},$ where $C>0$ is a constant independent of $T.$
	\begin{theorem}\label{null control 1}
		Let $T>0,$ $A\in \mathcal{L}(\cplx^n)$ and $B\in\mathcal{L}(\cplx^m;\cplx^n)$ be given. Then for every $w_0\in H^{-1 }_{ \alpha_1}(0,1)^n$,  there exists a control $h\in L^2(0,T)^m$ such that the system \eqref{oned} satisfies $w(T)=0$ if and only if the Kalman rank condition \eqref{kalman} holds. Moreover, the control satisfies the following estimate
		\begin{equation}\label{cost_cyl}
			\norm{h}_{L^2(0,T)^m}\leq Ce^{\frac{C}{T}}\norm{w_0}_{H^{-1 }_{ \alpha_1}(0,1)^n},
		\end{equation} 
		for some positive constant $C$ which is independent of $T.$
	\end{theorem}
\begin{remark}\label{rm}	In what concerns \Cref{null control 1}, the following remarks are worth mentioning:
	\begin{enumerate}
		\item	For 1-$d$ coupled degenerate parabolic systems, we refer to \cite{allal2022boundary}, where boundary controllability was studied for a system of two weakly degenerate equations (\(0 \leq \alpha_1 < 1\)) with control acting at the non-degenerate point. It is important to note that no controllability result exists for coupled degenerate systems with control at the point of degeneracy, as in the scalar case studied in \cite{galo2024boundary}. This is an additional novelty of our work.
 \item  \label{rm1}Without loss of generality, we assume throughout the paper that $A$ is stable (i.e., all eigenvalues of $A$ have negative real part). In fact, we can replace the matrix $A$ by $A-\mu \mathcal I_n$, for some $\mu>0$ to make the coupling matrix a stable one. Importantly, this modification does not affect the controllability result, as it simply involves replacing $w(t)$ by $e^{-\mu t} w(t)$ in \eqref{oned}. 
\end{enumerate}
\end{remark}

\item Next, we explore a partial observability (see \Cref{par obs}) for the corresponding adjoint system with data in a subset of $H^{1}_{\alpha,0}(\Omega)^n$ (defined in \Cref{sec:wp}). To construct the control in two dimensions, we use the controllability result from the one-dimensional case (see \Cref{null control 1}) along with the following spectral inequality
\begin{theorem}[Theorem 1.1, \cite{RB24}]\label{thm:spec_ineq_degen}
	Let $(\lambda_j,\Phi_j)$ be the solution of the eigenvalue problem \eqref{eigen eqn1}. Let $\omega$ be an open and nonempty subset of $(0, 1)$. There exists a constant $C > 0$ such that
	\begin{equation}\label{lr}
		\sum_{\lambda_j \leq \mu} |a_j|^2 \leq C e^{C_1\frac{1}{(2-\alpha_1)^2}  \sqrt{\mu}} 
		\int_{\omega} \left| \sum_{\lambda_j \leq \mu} a_j \Phi_j \right|^2 \, \rd{x}
	\end{equation}
	for any $\alpha_1 \in [0, 2)$, $\{a_j\} \in \mathbb{R}$, and any $\mu > 0$.
	\end{theorem}

\item We utilize the partial observability mentioned in the previous step and combining it with the usual Lebeau-Robbiano approach, we design a control strategy driving $u$ to zero at time $T$ (see \Cref{control_2d_suf}). 
  This is achieved in two stages:
 \begin{enumerate}
 	\item First, apply the control derived in the previous step to suppress low-frequency components.
 	\item  Then, allow the system to evolve naturally without the action of control, using its inherent dissipative properties to further decrease the norm of the state.
 \end{enumerate}
\end{itemize}

\subsection{Paper organization}
The rest of the paper is organized as follows. \cref{sec:wp} is devoted to the functional framework and well-posedness of the 2-$d$ degenerate system. In \Cref{1d}, we establish that under the Kalman rank condition, the one-dimensional coupled degenerate (weak or strong) parabolic system is null controllable with an optimal estimate of the cost of the control. \Cref{lr sc} contains the proof of the main result \Cref{main theorem} which uses the previous 1-$d$ controllability result and the Lebeau-Robbiano technique. \Cref{higher_d} deals with the control problem in higher dimensions.
In \Cref{sec:open pr}, we investigate some further extensions of our main result, raise some open problems and comments. 


\section{Functional framework and well-posedness}\label{sec:wp}
Let us denote the space $H^1_{\alpha}(\Omega)=\{u\in L^2(\Omega)\mid \nabla u\cdot D\nabla u\in L^1(\Omega)\},$ endowed with the norm 
\begin{equation*}
\norm{u}^2_{H^1_{\alpha}(\Omega)}=\int_{\Omega}|u|^2+\int_{\Omega}|D^{1/2}\nabla u|^2=\int_{\Omega}|u|^2+\int_{\Omega}\left(x^{\alpha_1}|\pa_xu|^2+y^{\alpha_2}|\pa_yu|^2\right).
\end{equation*} 
It is easy to note that $H^1(\Omega)\subset H^1_{\alpha}(\Omega).$ We also denote 
\begin{equation}\label{def:H_alpha2}
H^2_{\alpha}(\Omega) := \left\{ u \in H^1_{\alpha}(\Omega) \mid \operatorname{div} (D \nabla u) \in L^2(\Omega) \right\}.
\end{equation}
%
We will consider the following norms
	\begin{equation*}
\|u\|_{H^1_{\alpha}(\Omega)} := \left( \int_{\Omega} |u|^2 + \int_{\Omega}|D^{1/2} \nabla u |^2 \right)^{1/2}, \quad u \in H^1_{\alpha}(\Omega),
	\end{equation*}
	\begin{equation*}
\|u\|_{H^{2}_{\alpha}(\Omega)} := \left( \|u\|^2_{H^1_{\alpha}(\Omega)} +\int_{\Omega} |\operatorname{div} (D \nabla u)|^2 \right)^{1/2}, \quad u \in H^2_{\alpha}(\Omega).
	\end{equation*}
Note that, for these (natural) norms, \( H^1_{\alpha}(\Omega) \), \( H^2_{\alpha}(\Omega) \), 
are Hilbert spaces, and the following continuous embedding are satisfied
	\begin{equation*}
H^1(\Omega) \hookrightarrow H^1_{\alpha}(\Omega) \hookrightarrow L^2(\Omega); \quad H^2_{\alpha}(\Omega) \hookrightarrow H^1_{\alpha}(\Omega). 
	\end{equation*}
Furthermore, we have that $ H^1_{\alpha}(\Omega) \subset H^1_{\operatorname{loc}}(\Omega) $ and $ H^2_{\alpha}(\Omega) \subset H^2_{\operatorname{loc}}(\Omega) $.

Next we define the space $H^1_{\alpha,0}(\Omega):=\overline{\mathcal D_{\alpha}}^{H^1_{\alpha}(\Omega)},$ where the space $\mathcal D_{\alpha}$ 
depends on \( \alpha = (\alpha_1, \alpha_2) \) are given by
\[
\mathcal D_{\alpha} := 
\begin{cases} 
	\{ v \in C^{\infty}(\overline\Omega) : \operatorname{supp}(v) \subset\subset \Omega \} & \text{if } \alpha_1, \alpha_2 \in [0, 1), \\
	\{ v \in C^{\infty}(\overline\Omega) : \exists \, \epsilon > 0, \; \operatorname{supp}(v) \subset  (0, 1 - \epsilon) \times (0, 1 - \epsilon) \} & \text{if } \alpha_1, \alpha_2 \in [1, 2], \\
	\{ v \in C^{\infty}(\overline\Omega) : \exists \, \epsilon > 0, \; \operatorname{supp}(v) \subset  (\epsilon, 1 - \epsilon) \times (0, 1 - \epsilon) \} & \text{if } (\alpha_1, \alpha_2) \in [0, 1) \times [1, 2], \\
	\{ v \in C^{\infty}(\overline\Omega) : \exists \, \epsilon > 0, \; \operatorname{supp}(v) \subset  (0, 1 - \epsilon) \times (\epsilon, 1 - \epsilon) \} & \text{if } (\alpha_1, \alpha_2) \in [1, 2] \times [0, 1).
\end{cases}
\]
Thanks to \cite[Lemmas 4 and 13]{FA19}, we can state the following Hardy-Poincar\'e inequality.
\begin{lemma}\label{lm hardy}
Assume that $\alpha_i \in [0,2]\setminus\{1\}$ for $i = 1, 2$. Then, there exists a constant $C = C(\alpha_1, \alpha_2)$ such that 
	\begin{equation}\label{hardy}
		\int_{\Omega} x_i^{\alpha_i - 2} |u|^2 \, \rd x_1 \, \rd x_2 \leq C 
		\int_{\Omega} x_i^{\alpha_i} \left| \frac{\partial u}{\partial x_i} \right|^2 \, \rd x_1 \, \rd x_2, \quad \forall \, u \in H^1_{\alpha, 0}(\Omega).
	\end{equation}
\end{lemma}
Thus, we define the equivalent norm on the Hilbert space $H^1_{\alpha, 0}(\Omega)$, when $\alpha_i \in [0,2]\setminus\{1\}$ for $i = 1, 2$ as 
\begin{align*}
	\norm{u}^2_{H^1_{\alpha, 0}(\Omega)}:=\int_{\Omega}|D^{1/2}\nabla u|^2=\int_{\Omega}\left(x^{\alpha_1}|\pa_xu|^2+y^{\alpha_2}|\pa_yu|^2\right).
\end{align*}
We denote $H^{-1}_{\alpha}(\Omega)$ as the dual of the space $H^1_{\alpha,0}(\Omega)$ with respect to the pivot space $L^2(\Omega)$ with the natural norm $\norm{u}_{H^{-1}_{\alpha}(\Omega)}=\sup\limits_{\norm{v}_{H^1_{\alpha,0}(\Omega)}=1}{\ip{u}{v}_{H^{-1}_{\alpha}(\Omega),H^1_{\alpha,0}(\Omega)}}
.$

Let us denote $D^{-1}(x,y):=\text{diag}(x^{-\alpha_1}, y^{-\alpha_2})$
and introduce the following spaces, $L^2_{\alpha^{-1}}(\Omega):=\{w\in [L^2(\Omega)]^2\mid w\cdot D^{-1}w\in L^1(\Omega)\}$ and $H_\alpha^{\dv}(\Omega):=\{w\in L^2_{\alpha^{-1}}(\Omega):\dv(w)\in L^2(\Omega)\}$. We state the following useful integration-by-parts formula.
\begin{lemma}\label{lem:int_parts}
Let $\alpha_i\in[0,2)$, $i=1,2$. For any $w\in H_{\alpha}^{\dv}(\Omega)$
\begin{equation*}
\int_{\Omega}w\cdot \nabla u=-\int_{\Omega}\dv(w)u \quad \text{for all } u\in H_{\alpha,0}^1(\Omega).
\end{equation*}
\end{lemma}

\begin{proof}
The proof of this result follows directly from Lemmas 6, 7 and 11 in \cite{FA19}.
\end{proof}
Next, note that the inner product $\ip{\cdot}{\cdot}_{H^1_{\alpha, 0}(\Omega)}$ induces an isomorphism $-\Aalb:H^1_{\alpha, 0}(\Omega) \mapsto H^{-1}_{\alpha}(\Omega)$ given by 
	$\ip{-\Aalb u}{v}_{H^{-1}_{\alpha}(\Omega), H^1_{\alpha, 0}(\Omega)}:=\ip{u}{v}_{H^1_{\alpha, 0}(\Omega)}=\int_{\Omega}\left(x^{\alpha_1}\pa_xu\,\pa_x v + y^{\alpha_2}\pa_yu\,\pa_y v\right),$
see \Cref{iso app} for more details.

Let us define  $\mathcal D(\Aalb):=\left(-\Aalb\right)^{-1}(L^2(\Omega))=\{u\in H^1_{\alpha, 0}(\Omega)\mid -\Aalb u\in L^2(\Omega)\}=\{u\in H^1_{\alpha, 0}(\Omega)\mid \text{ there exists } g\in L^2(\Omega) \text{ such that } \ip{u}{v}_{H^1_{\alpha, 0}(\Omega)}=\ip{g}{v}_{L^2(\Omega)}  \text{ for all }  v \in H^1_{\alpha,0}(\Omega)\}.$
\begin{proposition}\label{well p op}
	Let $\alpha_i\in[0,2)\setminus{1}$, $i=1,2$. Then, \(\mathcal D(\Aalb)=H^2_{\alpha}(\Omega) \cap H^1_{\alpha,0}(\Omega)=\{u\in H^1_{\alpha,0}(\Omega)\mid \dv (D\nabla u)\in H^2_{\alpha}(\Omega)  \} \) and $\Aalb u= \dv(D\nabla u)$, $u\in \mathcal D(\Aalb)=H^2_{\alpha}(\Omega) \cap H^1_{\alpha,0}(\Omega).$ 
\end{proposition}
\begin{proof}
We begin by proving that \(\mathcal D(\Aalb)=H\) where $H:=H^2_{\alpha}(\Omega) \cap H^1_{\alpha,0}(\Omega)$. Let us take $u\in \mathcal D(\Aalb)$. Then there is $g\in L^2(\Omega)$ such that 
\begin{equation*}
\ip{u}{v}_{H^1_{\alpha, 0}(\Omega)}=\int_{\Omega}\left(x^{\alpha_1}\pa_xu\,\pa_x v + y^{\alpha_2}\pa_yu\,\pa_y v\right)=\int_{\Omega} g v, \quad \forall \, v\in H_{\alpha,0}^1(\Omega).
\end{equation*}
By density and the fact that $\mathcal C_c^\infty(\Omega)\subset \mathcal D_{\alpha}$, we have
\begin{equation*}
\int_{\Omega}D\nabla u\cdot \nabla v =\int_{\Omega} g v, \quad \forall\, v\in \mathcal C_c^\infty(\Omega).
\end{equation*}
Hence $-\dv(D\nabla u)=g$ in $(\mathcal C_c^\infty(\Omega))^\prime$, $-\dv(D\nabla u)\in L^2(\Omega)$ and therefore $u\in H$ (cf. eq. \eqref{def:H_alpha2}). 

On the other hand, let us take $u\in H$. Note that $(D \nabla u)\in H_{\alpha}^{\dv}(\Omega).$  Applying \Cref{lem:int_parts} we obtain,
\begin{equation}\label{dis wellp}
\ip{u}{v}_{H^1_{\alpha, 0}(\Omega)}=\int_{\Omega}D\nabla u\cdot \nabla v=-\int_{\Omega}\dv(D\nabla u) v = \langle g,v \rangle_{L^2(\Omega)},\quad  \forall \, v\in H_{\alpha,0}^1(\Omega).
\end{equation}
where $g:=-\dv(D\nabla u)\in L^2(\Omega)$ by assumption. Therefore $u\in \mathcal D(\Aalb)$ as required. 

Next considering $u\in \mathcal D(\Aalb), v\in H^1_{\alpha, 0}(\Omega),$ the following identity \begin{align*}\ip{-\Aalb u}{v}_{H^{-1}_{\alpha}(\Omega), H^1_{\alpha, 0}(\Omega)}=-\int_{\Omega}\dv(D\nabla u) v =\ip{-\dv(D\nabla u)}{v}_{H^{-1}_{\alpha}(\Omega), H^1_{\alpha, 0}(\Omega)}
\end{align*}
completes the proof.
%
%
\end{proof}
From now on wards we consider $\alpha_i\in[0,2)\setminus{1}$, $i=1,2$. 

The previous result justifies the following. 
\begin{proposition}\label{semig}
	The operator \( \Aalb : \mathcal D(\Aalb)\subset L^2(\Omega) \to L^2(\Omega) \) is an infinitesimal generator of a strongly continuous semigroup of operators on \( L^2(\Omega) \).
\end{proposition}
\begin{proof}
Taking $u=v$ in \eqref{dis wellp} where $u\in\mathcal{D}(\Aalb)$, we get $\ip{\Aalb u}{u}_{L^2(\Omega)}\leq 0.$ Hence $\Aalb$ is dissipative. 

Next, we show that $\operatorname{Ran}(\boldsymbol{I}-\Aalb) = L^2(\Omega)$. Let $g \in L^2(\Omega)$ be given. Thanks to \Cref{lm hardy}, the inner product $\langle \cdot, \cdot \rangle_{L^2(\Omega)} + \langle \cdot, \cdot \rangle_{H^1_{\alpha}(\Omega)}$ is equivalent to $\langle \cdot, \cdot \rangle_{ H_{\alpha,0}^1(\Omega)}$ in $ H_{\alpha,0}^1(\Omega)$. As $g \in H^{-1}_{\alpha}(\Omega)$, using the Riesz representation theorem, we obtain that there exists a unique $u \in  H_{\alpha,0}^1(\Omega)$ such that
\begin{equation*}
	\int_{\Omega}uv+	\int_{\Omega}D\nabla u\cdot \nabla v =\int_{\Omega} g v, \quad \forall v\in H_{\alpha,0}^1(\Omega) .
\end{equation*}
This implies that $u-\dv(D\nabla u)=g$ in $(\mathcal C_c^\infty(\Omega))^\prime$. Therefore  $u \in\mathcal D(\Aalb), \text{ and }  (u-\Aalb u) =g\in L^2(\Omega).$ Thus, we have proved that $\Aalb$ is maximal dissipative and hence forms an infinitesimal generator of a strongly continuous semigroup of operators on \( L^2(\Omega) \). This ends the proof.
\end{proof}
Next, we introduce the unbounded operator corresponding to the coupled system \eqref{DCP}
\begin{align*}
	\AalAb u= {\mathcal{I}_n \mathbf{D} u+A u}
	\quad\textnormal{with}\quad  \mathcal D(\AalAb)=(H^2_{\alpha}(\Omega) \cap H^1_{\alpha,0}(\Omega))^n. 
\end{align*}
As $A$ is a constant matrix, thanks to \Cref{semig}, the full operator $ \AalAb: \mathcal D(\AalAb)\subset L^2(\Omega)^n \to L^2(\Omega)^n$, is an infinitesimal generator of a strongly continuous semigroup of operators on $ L^2(\Omega)^n$.

Let us write the adjoint system of the homogeneous version of \eqref{DCP}--\eqref{bd} as follows
\begin{equation}\label{adjintro}
	\begin{cases}
		\partial_t \sigma+\mathcal{I}_n \mathbf{D} \sigma+A^* \sigma 
		=0 &  \text{ in } (0,T)\times \Omega,\\
		\sigma(T)=\sigma_T & \text{ in } \Omega,
	\end{cases}
\end{equation}
where $A^*$ denotes the adjoint matrix of $A$ and with the following boundary conditions
{\begin{equation}\label{bd2}
		\begin{cases}
			\sigma(t)=0 \text{ on } \Sigma \,\, \text{ if } \alpha_1, \alpha_2 \in [0,1),\\
			\sigma(t)=0 \text{ on } \Sigma_{34} \text{ and }  \mathbf{P} \sigma(t)=0 \text{ on }  \Sigma_1, \,  \mathbf{P} \sigma(t)=0 \text{ on } \Sigma_2 \,\, \textnormal{ if }\alpha_1,\alpha_2\in [1,2],\\
			\sigma(t)=0\text{ on } \Sigma_1, \sigma(t)=0 \text{ on } \Sigma_{34} \text{ and }  \mathbf{P} \sigma(t)=0 \text{ on } \Sigma_2 \,\, \textnormal{ if } \alpha_1 \in [0,1), \,\alpha_2\in [1,2],\\
			\sigma(t)=0 \text{ on } \Sigma_{234} \text{ and } \mathbf{P} \sigma(t)=0\text{ on }  \Sigma_1 \,\, \textnormal{ if }  \alpha_1\in [1,2], \, \alpha_2\in [0,1).
		\end{cases}
\end{equation}}
We have the following well-posedness result 
\begin{proposition}
	Let $T>0$ and assume that $ \sigma_T \in H^1_{\alpha,0}(\Omega)^n $. 
	 Then, system \eqref{adjintro}--\eqref{bd2} has a unique strong solution
		\begin{equation*}
	\sigma \in L^2(0,T; (H^2_{\alpha}(\Omega)\cap H^1_{\alpha,0}(\Omega))^n) \cap \mathcal C^0([0,T]; H^1_{\alpha,0}(\Omega)^n).
		\end{equation*}
	In addition, there exists a positive constant \( C \) such that
	\begin{equation*}
	\|\sigma\|_{L^2(0,T;(H^2_{\alpha}(\Omega)\cap H^1_{\alpha,0}(\Omega))^n)} + \|\sigma\|_{\mathcal C^0([0,T]; H^1_{\alpha,0}(\Omega)^n)} \leq C  \|\sigma_T\|_{H^1_{\alpha,0}(\Omega)^n}.
	\end{equation*}
\end{proposition}
The proof of this result is standard and we omit it. This motivates to define a weaker notion of solution for \eqref{DCP}--\eqref{bd}. 
\begin{definition}[Solution by transposition]
	 Let \( u_0 \in H^{-1}_{\alpha}(\Omega)^n \), \( q \in L^2(0,T; L^2(\pa \Omega)^m) \) be given. A function \( u \in \mathcal C^0([0,T]; H^{-1}_{\alpha}(\Omega)^n) \) is said to be a solution by transposition to \eqref{DCP}--\eqref{bd} if, for each \( \sigma_{\tau} \in H^1_{\alpha,0}(\Omega)^n) \) and for all $\tau\in (0,T]$ one has
	\begin{align*}
		\begin{cases}
				 \ip{u(\tau)}{\sigma_{\tau}}_{{H^{-1}_{\alpha}(\Omega)^n, H^1_{\alpha, 0}(\Omega)^n}}  = \langle u_0, \sigma( 0,\cdot) \rangle_{H^{-1}_{\alpha}(\Omega)^n, H^1_{\alpha, 0}(\Omega)^n}\vspace{0.1cm} \\
				\hspace{4.1cm}- \displaystyle \int_0^\tau\int_{\pa \Omega} \left(\mathbf{1}_{\gamma} B^*(  \mathbf{P} \sigma(t)), q(t) \right)_{\cplx^m} \,\rd S \,  \rd t \hspace{0.5cm} \text{ if } 0\leq \alpha_1<1,\vspace{0.2cm}\\
				\ip{u(\tau)}{\sigma_{\tau}}_{{H^{-1}_{\alpha}(\Omega)^n, H^1_{\alpha, 0}(\Omega)^n}}  = \langle u_0, \sigma( 0,\cdot) \rangle_{H^{-1}_{\alpha}(\Omega)^n, H^1_{\alpha, 0}(\Omega)^n}\vspace{0.1cm}\\ \hspace{4.1cm} \displaystyle -\int_0^\tau \int_{\pa\Omega} \left(\mathbf{1}_{\gamma} B^* \sigma(t), q(t) \right)_{\cplx^m} \, \rd S \, \rd t \hspace{1cm}  \text{ if } 1<\alpha_1<2,
		\end{cases}
	\end{align*}
where \( \sigma \in L^2(0,\tau; (H^2_{\alpha}(\Omega)\cap H^1_{\alpha,0}(\Omega))^n) \cap\mathcal C^0([0,\tau]; H^1_{\alpha,0}(\Omega)^n) \) is    the solution of \eqref{adjintro}--\eqref{bd2} with $\sigma(\tau,\cdot)=\sigma_\tau(\cdot).$
\end{definition}
Using the admissibility condition of the control operator (see \Cref{ad} in \Cref{sec:ad}) and following classical arguments (see eg. \cite[Section 2.3]{Cor08}), we can prove the existence and uniqueness for the solutions to \eqref{DCP}--\eqref{bd} in the above sense. 
\begin{proposition}
	Let $T>0$. for every $u_0\in H^{-1}_{\alpha}(\Omega)^n$, $q\in L^2(0,T;L^2(\pa \Omega)^m),$ equation \eqref{DCP}--\eqref{bd} possesses a unique solution $u\in \mathcal C^0([0,T]; H^{-1}_{\alpha}(\Omega)^n)$ and satisfies the following continuity estimate
	\begin{equation}\label{con est}
		\norm{u}_{\mathcal C^0([0,T]; H^{-1}_{\alpha}(\Omega)^n)}\leq Ce^{CT}\left(\norm{u_0}_{H^{-1}_{\alpha}(\Omega)^n}+\norm{q}_{L^2(0,T;L^2(\pa \Omega)^m)}\right),
	\end{equation}
	for some $C>0.$
\end{proposition}

\section{A one-dimensional boundary controllability problem for degenerate systems}\label{1d}
In this section, our primary objective is to establish the null controllability of the one-dimensional degenerate coupled parabolic system \eqref{oned}, as stated in \Cref{null control 1}. 

\subsection{Functional setting for the one dimensional case}\label{sec_fn}
We recall some foundational aspects of the well-known functional framework associated with the corresponding one-dimensional degenerate spatial operator  which can be found in \cite{gueye2014exact} and \cite{CMV20} for the weak and strong degenerate cases, respectively.
\subsubsection{Weakly degenerate case ($0\leq\alpha_1<1$)}
For $0\leq \alpha_1 <1$, let us define 
\begin{equation*}
	H^1_{\alpha_1}(0,1):=\{ u\in L^2(0,1) \, \mid u \text{ is absolutely continuous on } [0,1], \, x^{ \frac{\alpha_1}{2}} \pa_x u \in L^2(0,1)\},\end{equation*}
endowed with the inner product 
\begin{equation}	\label{ip1}
\ip{u}{v}_{H^1_{\alpha_1}(0,1)}:=\int_0^1 \left( uv+ x^{\alpha_1} \pa_x u \, \pa_x v\right) \, \textrm{\textrm{d}}x, \qquad u, v\in H^1_{\alpha_1}(0,1).
\end{equation}
We denote $||\cdot||_{H^1_{\alpha_1}(0,1)}$ the corresponding norm and note that $(H^1_{\alpha_1}(0,1),(\cdot,\cdot)_{H^1_{\alpha_1}(0,1)})$ is a Hilbert space.
Let us also define 
\begin{equation*}
H^1_{ \alpha_1,0}(0,1):=\{ u\in H^1_{ \alpha_1}(0,1) \, | \, u(0)=u(1)=0\},
\end{equation*}
and
\begin{equation*}
	H^2_{\alpha_1}(0, 1) := \left\{ u \in H^1_{\alpha_1}(0, 1) \mid x^{\alpha_1} \pa_x u \in H^1(0, 1) \right\}.
\end{equation*}
We recall the following Hardy-Poincar\'e inequality:
\begin{equation*}
\int_0^1 x^{ \alpha_1-2} \left| u \right|^2 \, \rd x\leq \dfrac{4}{(1- \alpha_1)^2}	\int_0^1 x^{\alpha_1} \left| \pa_x u \right|^2 \, \rd x , \qquad \text{ for all } u\in  C^{\infty}_{ c}(0,1).
\end{equation*}
Using the fact that $x^{\alpha_1-2}\geqslant 1$ for $x\in (0,1)$, we obtain (since $H^1_{\alpha_1,0}(0,1)$ is the closure of $C^{\infty}_c(0,1)$ endowed with the $H^1_{\alpha_1}(0,1)$ norm) 
\begin{equation}\label{HPIi}
	\int_0^1 \left| u \right|^2 \, \textrm{\textrm{d}}x \leq \dfrac{4}{(1- \alpha_1)^2} \int_0^1 x^{\alpha_1} |\pa_x u|^2 \, \textrm{\textrm{d}}x, \qquad \text{ for all } u \in H^1_{ \alpha_1,0}(0,1).
\end{equation}
Finally, thanks to \eqref{HPIi}, we define the following norm on $H^1_{ \alpha_1,0}(0,1)$, which is equivalent to the restriction of the norm $||\cdot||_{H^1_{\alpha_1}}$ on $H^1_{ \alpha_1,0}(0,1)$ 
\begin{equation}\label{eq_1d}
\| u \|^2_{H^1_{ \alpha_1,0}}:=\left( \int_0^1 x^{\alpha_1} | \pa_x u |^2 \, \rd x \right).
\end{equation}
Let $ H^{-1}_{\alpha}(0, 1)$ be the dual space of $ H^1_{\alpha, 0}(0, 1) $ with respect to the pivot space $ L^2(0, 1) $, endowed with the natural norm:
\begin{equation}\label{h-1}
	\|f\|_{H^{-1}_{\alpha_1}(0,1)} := \sup_{\|g\|_{H^1_{\alpha_1, 0}(0,1)} = 1} \langle f, g \rangle_{H^{-1}_{\alpha}(0,1), H^1_{\alpha, 0}(0,1)}.
\end{equation}
We define the unbounded operator $\mathcal A_{\alpha_1} :\mathcal D(\mathcal A_{\alpha_1}) \subset L^2(0,1) \rightarrow L^2(0,1)$ by
\begin{equation*}
\begin{cases}
	\mathcal A_{\alpha_1}u:= \pa_x(x^{\alpha_1} \pa_x u), \text{ for all } u \in \mathcal D(\mathcal A_{\alpha_1}),\\ 
\mathcal	D(\mathcal A_{\alpha_1}):=\{ u\in H^1_{\alpha_1,0}(0,1) \, | \, x^{\alpha_1} \pa_x u \in H^1(0,1) \}= H^2_{\alpha_1}(0, 1) \cap H^1_{\alpha_1,0}(0, 1).
\end{cases}
\end{equation*}
One can prove that the operator $\mathcal{A}_{\alpha_1}$ is
the infinitesimal generator of a strongly continuous semigroup of contractions on $L^2(0, 1)$.
\subsubsection{Strongly degenerate case ($1<\alpha_1< 2$)}\label{sec_str}
 Let us introduce the weighted Sobolev space
\begin{eqnarray}\label{h1_a}
	H_{\alpha_1}^1(0,1) := \left\{ u \in L^2(0,1) \mid \, u \textrm{ is locally absolutely continuous on } (0,1], 
	  x^{\frac{\alpha_1}{2}} \pa_x u \in L^2(0,1) \right\}, 
\end{eqnarray} endowed with the norm
\begin{equation*}
	\|u\|_{H^1_{\alpha_1}(0,1)}^2 := \|u\|_{L^2(0,1)}^2 + \|x^{\frac{\alpha_1}{2}} \pa_x u\|_{L^2(0,1)}^2, \quad \forall u \in H^1_{\alpha_1}(0,1).
\end{equation*} 
We remark that $H^1_{\alpha_1}(0,1)$ is a Hilbert space with the scalar product \eqref{ip1}. Let us define the following Sobolev spaces
\begin{equation*}
H^1_{\alpha_1,0}(0, 1) := \left\{ u \in H^1_{\alpha_1}(0, 1) \mid u(1) = 0 \right\},
\end{equation*}
and
\begin{equation*}
H^2_{\alpha_1}(0, 1) := \left\{ u \in H^1_{\alpha_1}(0, 1) \mid x^{\alpha_1} \pa_x u \in H^1(0, 1) \right\}.
\end{equation*}
We recall the following Hardy-Poincar\'e inequality (see \cite[Proposition 3.3]{galo2024boundary}) for the case $1<\alpha_1<2$, 
\begin{equation}
\label{weightespoincare1}
\int_0^1 |u|^2 \, \rd x \leq \int_0^1 \frac{|u|^2}{x^{2 - \alpha_1}} \, \rd x \leq \frac{4}{(\alpha_1-1)^2} \int_0^1 x^{\alpha_1 } |\pa_x u|^2 \, \rd x, \quad \text{ for all } u \in H^1_{ \alpha_1,0}(0,1).
\end{equation}
We can define the equivalent norm on $H^1_{ \alpha_1,0}(0,1)$ as \eqref{eq_1d} and the space $H^{-1}_{\alpha_1}(0,1)$ as in \eqref{h-1}.

Next, let us define the operator $\mathcal A_{\alpha_1} : \mathcal D(\mathcal A_{\alpha_1}) \subset L^2(0, 1) \to L^2(0, 1) $ by
\begin{equation*}
\begin{cases}
	 \mathcal A_{\alpha_1} u := \pa_x(x^{\alpha_1} \pa_x u), \text{ for all } u \in \mathcal D(\mathcal A_{\alpha_1}), \\
	\mathcal D(\mathcal A_{\alpha_1}) := \left\{ u \in H^1_{\alpha_1,0}(0, 1) \mid x^{\alpha_1} \pa_x u \in H^1(0, 1), (x^{\alpha_1}u)\in H^1_0(0,1) \right\} = H^2_{\alpha_1}(0, 1) \cap H^1_{\alpha_1,0}(0, 1).
\end{cases}
\end{equation*}
Note that, if $ u \in D(\mathcal A_{\alpha_1}) $, then $ u $ satisfies the Neumann boundary condition $ (x^{\alpha_1} \pa_x u)(0) = 0 $ at $ x = 0$ and the Dirichlet boundary condition $ u(1) = 0 $ at $ x = 1 $. 
\begin{remark}
	Note that, for simplicity of the notation, we have not distinguished between spaces such as $H^1_{ \alpha_1}(0,1)$, $H^1_{ \alpha_1,0}(0,1),$ $H^{-1}_{\alpha_1}(0,1)$ etc and the corresponding operators $(\mathcal A_{\alpha_1}, \mathcal D(\mathcal A_{\alpha_1})) $ for the weak and strong degenerate cases. We maintain this convention in the subsequent sections.
\end{remark}

\subsection{Spectral analysis of the 1-$d$ degenerate parabolic operator}
This section is devoted to recall some information regarding the eigenvalues and eigenfunctions of the scalar degenerate parabolic operator under the suitable boundary conditions (see \cite{gueye2014exact}, \cite{CMV20}). We begin by stating the following well-known result. 
\begin{proposition}\label{prop:hilbert_basis}
The weak and strong degenerate operators $-\mathcal A_{\alpha_1}$ defined above are self-adjoint, positive definite, and have a compact resolvent. Therefore, for each case, there exists an orthonormal basis  $\{\phi_{\alpha_1,k}\}_{k \in \N^*}$, of $L^2(0,1)$ and an increasing sequence $(\lambda_{\alpha_1,k})_{k \in \N^*}$,  of real numbers such that $\lambda_{\alpha_1,k}>0$ and $\lambda_{\alpha_1,k} \rightarrow +\infty$.
\end{proposition}
In order to find explicit expressions of the spectrum, we solve the following Sturm-Liouville problem for any suitable $\lambda\in \rea:$
\begin{equation}\label{eigen eqn}
	\begin{cases}
		-(x^{\alpha_1}\phi')'(x)=\lambda \phi(x) \qquad x\in (0,1),\\
		\begin{cases}
			\phi(0)=0 &\text{ if } 0\leq \alpha_1<1\\
			(x^{\alpha_1}\phi')(0)=0 &\text{ if } 1\leq\alpha_1<2
		\end{cases},\\
		\phi(1)=0.
	\end{cases}
\end{equation}
Let us introduce some certain standard definitions, notations, and properties of Bessel's functions to study the above eigenvalue problem \eqref{eigen eqn}. The Bessel's functions denoted by $J_{\nu}$
of order $\nu$ are the solutions of the following differential equation
\begin{equation*}
	x^2y''(x)+xy'(x)+(x^2-\nu^2)y(x)=0, \quad  x\in (0,\infty).
\end{equation*}
One can define Bessel functions of the first kind by the following formula
\begin{equation}\label{exp_of_bessel}
	J_\nu(x)=\sum_{k\geq 0} \frac{(-1)^k}{k!\,\Gamma(k+\nu+1)}\left(\frac{x}{2}\right)^{2k+\nu},\: x \geq 0,
\end{equation}
where $\Gamma(\cdot)$ is the gamma function\footnote{We recall that the gamma function is defined as follows: $\Gamma:\: x\in (0,\infty)\mapsto \int_0^{+\infty} t^{x-1}e^{-t}\rd t.$}. It can be proved that the function $J_\nu$ has an infinite number of real zeros. These zeros are all simple, with the possible exception of $x=0$, depending on the value of $\nu$. In particular, for $\nu \geq 0$ and $0 < x \leq \sqrt{\nu + 1}$, from \eqref{exp_of_bessel} one can obtain (see \cite[9.1.7, p. 360]{AM1992})
\begin{equation*}
	J_\nu(x) \sim \frac{1}{\Gamma(\nu + 1)} \left(\frac{x}{2}\right)^\nu \quad \text{as } x \to 0^+.
\end{equation*}
From now on, we concentrate on the following choice of parameter $\nu$
\begin{align*}
\nu(\alpha_1):=\frac{2}{2-\alpha_1}\sqrt{\left(\frac{1-\alpha_1}{2}\right)^2}=\frac{|1-\alpha_1|}{2-\alpha_1} &\text{ for }  0\leq \alpha_1<2.
\end{align*}
We also introduce the parameter $\kappa_{\alpha_1}$ given by 
\begin{equation*}
	\kappa_{\alpha_1}:= \dfrac{2-\alpha_{1}}{2}.
\end{equation*}
Let us denote $\{j_{\nu(\alpha_1),k}\}_{k\geq 1}$ the strictly increasing sequence of positive zeros of the Bessel function $J_{\nu(\alpha_1)}$
\begin{equation*}
	0<j_{\nu(\alpha_1),1}<j_{\nu(\alpha_1),2}<...<j_{\nu(\alpha_1),k}<...
\end{equation*}
and also we recall that $j_{\nu(\alpha_1),k}\to \infty$ as $k\to \infty$.  We have the following bounds on the zeros (see \cite{LM08})
\begin{align}
	\label{first in}	& \text{ For } \, \nu(\alpha_1) \in \left[0,\frac{1}{2}\right],\,  \forall \, k\geq 1,\, \,   \left(k+\frac{\nu(\alpha_1)}{2}-\frac{1}{4}\right)\pi\leq j_{\nu(\alpha_1),k}\leq \left(k+\frac{\nu(\alpha_1)}{4}-\frac{1}{8}\right)\pi,\\
	\label{second in}	& \text{ For } \, \nu(\alpha_1)\leq \frac{1}{2}, \, \forall \, k \,\geq 1,\, \,   \left(k+\frac{\nu(\alpha_1)}{4}-\frac{1}{8}\right)\pi\leq j_{\nu(\alpha_1),k}\leq \left(k+\frac{\nu(\alpha_1)}{2}-\frac{1}{4}\right)\pi.
\end{align}
Moreover we have the following classical result from \cite[Proposition 7.8]{KL05}.
\begin{itemize}
	\item For $\nu(\alpha_1) \in [0,\frac{1}{2}]$, the sequence $\{j_{\nu(\alpha_1),{k+1}}-j_{\nu(\alpha_1),{k}}\}_{k\geq 1}$ is nondecreasing and converges to $\pi$ as $k\to \infty$.
	\item For  $\nu(\alpha_1)\geq \frac{1}{2}$, the sequence $\{j_{\nu(\alpha_1),{k+1}}-j_{\nu(\alpha_1),{k}}\}_{k\geq 1}$ is nonincreasing and converges to $\pi$ as $k\to \infty$.
\end{itemize}
For any $\nu(\alpha_1)\geq 0$, Bessel functions satisfy the following orthogonality property
\begin{equation*}
	\int_{0}^{1}x^{1-\alpha_1}J_{\nu(\alpha_1)}(j_{\nu(\alpha_1),k}x^{\kappa_{\alpha_1}})J_{\nu(\alpha_1)}(j_{\nu,m}x^{\kappa_{\alpha_1}})\rd x=\begin{cases}\frac{1}{2\kappa_{\alpha_1}}[J_{\nu(\alpha_1)}'(j_{\nu(\alpha_1),k})]^2, \text{ if } k=m,\\
		0 , \text{ if } k\neq m.
	\end{cases}
\end{equation*}
We are ready to state the result regarding the spectrum of the degenerate parabolic operators
\begin{proposition}\label{eigenelement}
	Assume $0\leq \alpha_1<2$.
	Then, the admissible eigenvalues $\lambda$ for problem \eqref{eigen eqn}
	are given by 
	\begin{equation}\label{eigenvalue}
		\lambda_{\alpha_1,k}=\kappa_{\alpha_1}^2j_{\nu(\alpha_1), k}^2, \,\, k\geq 1
	\end{equation}
	and the corresponding normalized eigenfunctions are 
	\begin{equation}\label{eigenfn}
		\phi_{\alpha_1,k}(x)=\frac{\sqrt{2-\alpha_1}}{|J'_{\nu({\alpha_1})}(j_{\nu(\alpha_1),k})|}\sqrt{x^{1-\alpha_1}}\,\, J_{\nu({\alpha
		_1})}(j_{\nu(\alpha_1),k}x^{\kappa_{\alpha_1}}), \,\, x\in (0,1), k\geq 1.
	\end{equation}
	Moreover the family $\{\phi_{\alpha_1,  k}\}_{k\geq 1}$ forms an orthonormal basis in $L^2(0,1).$
\end{proposition}
We now write the following gap conditions of the eigenvalues given in \eqref{eigenvalue}.
\begin{lemma}\label{lmm_gap}
	The sequence of eigenvalues $\{\lambda_{\alpha_1, k}\}_{k\geq 1}$ satisfies the following gap condition: there exist positive constants $\rho_1, \rho_2$
	\begin{equation}\label{gap con}
		\rho_1|k^2-m^2|\leq |\lambda_{ \alpha_1, k}-\lambda_{ \alpha_1, m}|\leq \rho_2|k^2-m^2|, \,\, \forall k,m\geq 1.
	\end{equation}
\end{lemma}
\begin{proof} First assume $0\leq \alpha_1<1.$ Note that, in this case $\nu(\alpha_1) \in \left[0,\frac{1}{2}\right]$. Let us first prove the right hand side inequality. Let us assume that $k>m.$
	  Thanks to \eqref{first in}, we have 
		\begin{align*}
			\lambda_{ \alpha_1, k}-\lambda_{ \alpha_1, m}&\leq \kappa_{\alpha_1}^2\bigg[ \left(k+\frac{\nu(\alpha_1)}{4}-\frac{1}{8}\right)^2\pi^2-\left(m+\frac{\nu(\alpha_1)}{2}-\frac{1}{4}\right)^2\pi^2\bigg]\\
			&\leq \pi^2\kappa_{\alpha_1}^2\bigg[\left(k+m+\frac{3\nu(\alpha_1)}{4}-\frac{3}{8}\right)\left(k-m-\frac{\nu(\alpha_1)}{4}+\frac{1}{8}\right)\bigg]\\
			&\leq 2\pi^2\kappa_{\alpha_1}^2\left(k^2-m^2\right), \left[\text{using} \left(k-m-\frac{\nu(\alpha_1)}{4}+\frac{1}{8}\right)\leq 2(k-m)\right].
		\end{align*}
	Therefore there exists $\rho_2>0$ such that \begin{equation*}\lambda_{ \alpha_1, k}-\lambda_{ \alpha_1, m}\leq \rho_2(k^2-m^2).
	\end{equation*}
	Interchanging the role of $k,m$ in the above two cases, we have
	\begin{equation*}\lambda_{ \alpha_1, m}-\lambda_{ \alpha_1, k}\leq \rho_2(m^2-k^2).
	\end{equation*} 
	Hence the right hand inequality of \eqref{gap con} follows.
	
	\noindent
		Again using \eqref{first in}, we have 
		\begin{align*}
			\lambda_{ \alpha_1, k}-\lambda_{ \alpha_1, m}&\geq \kappa_{\alpha_1}^2\bigg[\left(k+\frac{\nu(\alpha_1)}{2}-\frac{1}{4}\right)^2\pi^2-\left(m+\frac{\nu(\alpha_1)}{4}-\frac{1}{8}\right)^2\pi^2\bigg]\\
			&\geq \pi^2\kappa_{\alpha_1}^2\left(k+m+\frac{3\nu(\alpha_1)}{4}-\frac{3}{8}\right)\left(k-m+\frac{\nu(\alpha_1)}{4}-\frac{1}{8}\right)\\
			&\geq \pi^2\kappa_{\alpha_1}^2\frac{(k+m)}{2}\frac{(k-m)}{2}\\
			&\geq \frac{\pi^2}{4}\kappa_{\alpha_1}^2\left(k^2-m^2\right). 
		\end{align*}
	Again interchanging the role of $k,m$ in the above two cases, we can conclude the desired inequality.
	
\noindent
For the case of $\alpha_1\in [1,2),$ note that $\nu(\alpha_1)\leq \frac{1}{2}, \text{ when } 1\leq \alpha_1\leq \frac{4}{3}$ and $\nu(\alpha_1)\geq \frac{1}{2},$ otherwise. Thus in this cases, using \eqref{first in} and \eqref{second in}, one can prove the required result.
\end{proof}

\subsection{Well-posedness}\label{sub_wp}
Let us first consider the adjoint system
\begin{equation}\label{adj system}
	\begin{cases}
		\pa_t v+\pa_{x}(x^{\alpha_1}\pa_{x}v)+A^*v=0 & \text{ in }   (0,T) \times (0,1),\\
		\begin{cases}v(t,0)=0 & \text{ if } 0\leq \alpha_1<1\\
			(x^{\alpha_1}\pa_x v)(t,0)=0 & \text{ if } 1 < \alpha_1<2
		\end{cases} &\text{ in }   (0,T),\\
		v(t,1)=0 &\text{ in }   (0,T),\\
		v(T,x)=v_T(x)  & \text{ in }   (0,1).
	\end{cases}
\end{equation}
The following well-posedness result is now classical (see  \cite[Proposition 2]{El2013}).
\begin{proposition}
	Assume that \( v_T \in H^1_{\alpha_1,0}(0,1)^n \). Then, system \eqref{adj system} has a unique strong solution
		\begin{equation*}
	v \in L^2(0,T; (H^2_{\alpha_1}(0,1)\cap H^1_{\alpha_1,0}(0,1))^n) \cap\mathcal C^0([0,T]; H^1_{\alpha_1,0}(0,1)^n).
		\end{equation*}
	In addition, there exists a positive constant \( C \) such that
	\begin{equation*}
		\|v\|_{L^2(0,T;(H^2_{\alpha_1}(0,1)\cap H^1_{\alpha_1,0}(0,1))^n)} + \|v\|_{\mathcal C^0([0,T]; H^1_{\alpha_1,0}(0,1)^n)} \leq C  \|v_T\|_{H^1_{\alpha_1,0}(0,1)^n}. 
	\end{equation*}
\end{proposition}
As in the two dimensional case, we can state the following weaker notion of solution and its corresponding existence/uniqueness result. For more details regarding scalar parabolic equations, we refer to \cite[Proposition 4.3]{gueye2014exact} 
 for the weakly degenerate case, and \cite[Proposition 3.10]{galo2024boundary} for the strongly degenerate one.
	\begin{definition}[Solution by transposition]\label{trans_oned}
		Let \( w_0 \in H^{-1}_{\alpha_1}(0,1)^n \), \( h \in L^2(0,T)^m \) be given. We will say that \( w \in\mathcal C^0([0,T]; H^{-1}_{\alpha_1}(0,1)^n) \) is a solution by transposition to \eqref{oned} if, for each $ v_\tau \in H^1_{\alpha_1, 0}(0,1)^n $, and for all $\tau \in (0,T]$, one has
		\begin{equation*}
			\begin{cases}
			\ip{w(\tau)}{v_{\tau}}_{{H^{-1}_{\alpha_1}(0,1)^n, H^1_{\alpha_1, 0}(0,1)^n}} &= \langle w_0, v( 0,\cdot) \rangle_{{H^{-1}_{\alpha_1}(0,1)^n, H^1_{\alpha_1, 0}(0,1)^n}}  \\
			 &\displaystyle \hspace{0.2cm} -\int_0^\tau \left(  h(t), B^*(x^{\alpha_1} \pa_x v)(t,0) \right)_{\cplx^m} \, \rd t  \qquad \text{ if } 0\leq \alpha_1<1,\\
				\ip{w(\tau)}{v_{\tau}}_{{H^{-1}_{\alpha_1}(0,1)^n, H^1_{\alpha_1, 0}(0,1)^n}} &= \langle w_0, v( 0,\cdot) \rangle_{{H^{-1}_{\alpha_1}(0,1)^n, H^1_{\alpha_1, 0}(0,1)^n}}  \\
				&\hspace{0.2cm}- \displaystyle \int_0^\tau \left(  h(t), B^*( v)(t,0) \right)_{\cplx^m} \, \rd t  \hspace{1.74 cm} \text{ if } 1<\alpha_1<2,
			\end{cases}
		\end{equation*}
		where \( v \in L^2(0,\tau;(H^1_{\alpha_1,0}(0,1)\cap H^2_{\alpha_1}(0,1))^n) \cap\mathcal C^0([0,\tau]; H^1_{\alpha_1,0}(0,1)^n) \) is the solution to \eqref{adj system} with $v(\tau,\cdot)=v_{\tau}(\cdot).$
	\end{definition}

	\begin{proposition}
		Let $T>0$. for every $w_0\in H^{-1}_{\alpha_1}(0,1)^n, h\in L^2(0,T)^m,$ equation \eqref{oned} possesses a unique solution $w\in C^0([0,T]; H^{-1}_{\alpha_1}(0,1)^n)$ and satisfies the following continuity estimate
		\begin{equation*}
			\norm{w}_{\mathcal C^0([0,T]; H^{-1}_{\alpha_1}(0,1)^n)}\leq Ce^{CT}\left(\norm{w_0}_{H^{-1}_{\alpha_1}(0,1)^n}+\norm{h}_{L^2(0,T)^m}\right),
		\end{equation*}
		for some $C>0.$
	\end{proposition}
	
\subsection{Analysis of the coupled degenerate operator}
It is well-known that the following duality result provides an equivalent criterion for the null controllability of \eqref{oned}.
\begin{lemma}
	Let us choose $w_0\in H^{-1 }_{ \alpha_1}(0,1)^n.$ We know that the null controllability of \eqref{oned} is equivalent to finding a control $h\in L^2(0,T)^m$
	such that for all $v_T \in H^{1 }_{ \alpha_1,0}(0,1)^n$, the following holds
	\begin{align}
\label{moment}	&	\ip{w_0}{v(0,\cdot)}_{{H^{-1}_{\alpha_1}(0,1)^n, H^1_{\alpha_1, 0}(0,1)^n}}=\int_{0}^{T}\left(h(t), B^* (x^{ \alpha_1}\pa_x v)(t,0)\right)_{\cplx^m} \rd t & \text{ if } 0\leq \alpha_1<1,\\
		\label{moment2}
				&	\ip{w_0}{v(0,\cdot)}_{{H^{-1}_{\alpha_1}(0,1)^n, H^1_{\alpha_1, 0}(0,1)^n}}=\int_{0}^{T}\left(h(t), B^* v(t,0)\right)_{\cplx^m} \rd t &  \text{ if } 1<\alpha_1<2,
		\end{align}
	where $v$ is the solution of the adjoint system \eqref{adj system} with $v(T,\cdot)=v_T$.
\end{lemma}
\begin{proof}
	Thanks to \Cref{trans_oned}, the proof follows directly.
\end{proof}
By well-known arguments (see e.g. \cite[Section V.4.2]{Boy23}), using \Cref{prop:hilbert_basis} the above result can be reformulated in an equivalent moment problem in terms of the eigenvalues of the underlying system (see eq. \eqref{moment1} below). In turn, this moment problem will be solved by verifying the hypothesis of the following theorem. 
\begin{theorem}[Theorem 1.5, \cite{AB2014}]\label{d1}
	Let $\Lambda =\left \{ \Lambda_{ k} \right \}_{k \geq 1}\subset \cplx$ be a complex sequence satisfying the following properties:
	\begin{enumerate}[label=\roman*)]
		\item \label{item1} $\Lambda_{k}\neq \Lambda_m$ for all $m,k\in \mathbb{N} $ with $m\neq k$; 
		\item \label{item2}$\Re(\Lambda_k)>0$ for every $k\geq1$;
		\item \label{item3} for some $\beta>0$, $$\vert \Im(\Lambda_k) \vert \leq \beta \sqrt{\Re(\Lambda_k)},\,\, \text{for any} \,\,  k\geq1;$$
		\item \label{item4}  $\{ \Lambda_k\}_{k\geq1}$ is non decreasing in modulus, i.e., $\vert \Lambda_k\vert \leq \vert \Lambda_{k+1} \vert$, for any $ k\geq1$;
		\item \label{item5}$\left \{ \Lambda_k \right \}_{k \geq 1}$ satisfies the following gap condition: for some $\rho, q>0$
		\begin{align*}
			\rho \left| k^2 -m^2 \right| \leq &\left\vert \Lambda_{k}-\Lambda_{m} \right\vert   \text{ for any }  m,k\geq 1: \vert k-m \vert \geq q; \\
			&	\inf_{k\neq m: |k-m|<q}\left\vert \Lambda_{k}-\Lambda_{m} \right\vert>0
		\end{align*}

		\item \label{item7} There exist $p_0$, $p_1$, $p_2$ with $p_1,  p_2 \geq p_0>0 $ such that one has,
		\begin{equation*}
			-\varpi+p_1\sqrt{r}  \leq\mathcal{N}(r)  \leq \varpi+ p_2 \sqrt{r}, \quad \forall r>0,
		\end{equation*}
		where $\mathcal{N}$ is the counting function associated with the sequence $\Lambda $, defined by 
		\begin{equation}\label{counting}
			\mathcal{N}(r)= \# \ens{ k :\, \vert \Lambda_{k}\vert \leq r}, \quad \forall r>0.
		\end{equation}
	\end{enumerate}
	Then, there exists $T_0>0$ such that, for every $\eta\geq 1$ and $0<T<T_0$, we can find a family of complex valued functions $\{\Psi_{k,j}\}_{k\geq1, 0\leq j\leq \eta-1}\in L^2\left(-\frac{T}{2},\frac{T}{2}\right)$ biorthogonal to $\{e_{k,j}\}_{k\geq1, 0\leq j\leq \eta-1},$ where for every $t\in (-\frac{T}{2},\frac{T}{2}),$  $e_{k,j}=t^je^{-\Lambda_k t}$ with in addition,
	\begin{equation}\label{cost}
		\norm{\Psi_{k,j}}_{L^2(-\frac{T}{2},\frac{T}{2})}\leq C e^{C\sqrt{\Re(\Lambda_k)}+\frac{C}{T}}.
	\end{equation}
\end{theorem}

So, our task will be reduced to arrange the eigenvalues of the one-dimensional degenerate operator $\mathcal{A}_{\alpha_1,A^*}=\pa_{x}(x^{\alpha_1}\pa_{x})+A^*$ in a suitable way and verify that the collection of the eigenvalues satisfies the conditions \ref{item1} to \ref{item7} of \Cref{d1}, in the same spirit as in \cite[Section 3.1]{AB2014}.

\begin{lemma}\label{lem:rea}
Let $\{\mu_l\}_{1\leq l\leq p}\subset \cplx$ be the set of distinct eigenvalues of the matrix $A^*$ and $\{\Lambda_k\}_{k\geq 1}$ be the eigenvalues of the the operator $-\mathcal{A}_{\alpha_1,A^*}$. There are natural numbers $\tilde{p}$ and $K_0$ such that the family $\{\Lambda_k\}_{k\geq 1}$ can be rearranged as 
\begin{equation}\label{ar eigen}
	\begin{cases}
		\Lambda_{\ell}=-\gamma_{\ell}\;\; \text{ for } 1\leq \ell\leq \tilde{p},\\
		\Lambda_{\tilde{p}+i}=\lambda_{ \alpha_1, K_0+j}-\mu_l \text{ with } j=\lfloor \frac{i-1}{p}\rfloor+1 \text{ and } l=i-\lfloor\frac{i-1}{p}\rfloor p , \;\;  i\geq 1,
	\end{cases}
\end{equation}
where $\{\gamma_{\ell}\}_{1\leq \ell\leq \tilde{p}}=\{-\lambda_{ \alpha_1, k}+\mu_l\}_{1\leq k \leq K_0, 1\leq l\leq {p}}$.
\end{lemma}

\begin{proof}

Let $\{\mu_l\}_{1\leq l\leq p}\subset \cplx$ be the set of distinct eigenvalues of the matrix $A^*.$ We assume that the geometric multiplicity of the eigenvalue $\mu_l$ is $n_l,$ for $1\leq l\leq p$ and we assume that the size of the Jordan chains are $\tau_{l,j}, 1\leq j\leq n_l$. By using \cite[case 2 p. 583]{JMPA11}, one may consider $\tau_{l,j}=\tau_{l},$ independent of $j.$ We also set $\hat{n}=\max_{1\leq n\leq p}n_l.$

We next assume the eigenvalues $\mu_l$ of $A^*$ in the following way: for $1\leq l\leq p-1$,
\begin{equation}\label{arr}
	\begin{cases}
		\Re({\mu_l})\geq \Re({\mu_{l+1}}),\\
		|\mu_l|\geq |\mu_{l+1}|, \text{ if } \Re({\mu_l})=\Re({\mu_{l+1}}). 
	\end{cases}
\end{equation}
The eigenvalues of the operator $\mathcal{A}_{\alpha_1,A^*}=\pa_{x}(x^{\alpha_1}\pa_{x})+A^*$,  
 are given by $-\lambda_{ \alpha_1, k}+\mu_l$, $k\geq 1$, $1\leq l \leq p$. Using the gap condition \eqref{gap con} of \Cref{lmm_gap}, and \cite[Proposition 3.2]{JMPA11}, we can find $k_0\in \N$ such that \begin{equation}\label{distinct}-\lambda_{ \alpha_1, k}+\mu_i\neq -\lambda_{\alpha_1,l}+\mu_j, \text{ for every } k\geq k_0, l\geq 1, l\neq 1 \text{ and }  1\leq i,j\leq p, i\neq j.
 \end{equation}
Thanks to the arrangements \eqref{arr} of the eigenvalues of $A^*$, we get the existence of $k_1\in \N$ (as large as needed) such that  for $1\leq l\leq p-1$
\begin{equation*}
	2\lambda_{ \alpha_1, k_1}\left(\Re(\mu_l)-\Re(\mu_{l+1})\right)+|\mu_{l+1}|^2-|\mu_l|^2\geq 0.
\end{equation*}
A simple computation allows us to deduce for $k\geq k_1, 1\leq l\leq p-1,$
\begin{equation}\label{con1}
	|\lambda_{ \alpha_1, k}-\mu_l|\leq |\lambda_{ \alpha_1, k}-\mu_{l+1}|.
\end{equation}
Also we have $k_2\in \N$ such that for all $k\geq k_2, 1\leq i\leq p, 1\leq j\leq p$ with $i\neq j$
\begin{equation}\label{con2}
	|\lambda_{ \alpha_1, k}-\mu_i|\leq |\lambda_{\alpha_1,k+1}-\mu_j|.
\end{equation}

Let us set $K_0=\max\{k_0,k_1,k_2\}.$ With this $K_0$, let $\tilde{p}$ be the number of distinct eigenvalues of the matrix $\mathbf{L}^*_{K_0}$ (see \eqref{matrix}) and denote them as $\{\gamma_{\ell}\}_{1\leq \ell\leq \tilde{p}}\subset \cplx$. We arrange these distinct eigenvalues in such a way that $|\gamma_{\ell}|\leq |\gamma_{\ell+1}|$ for every $1\leq \ell\leq \tilde{p}$. Keeping the consistent notation used for the matrix $A^*$, here we assume that for $1\leq \ell\leq \tilde p$, $N_{\ell}$ is the number of geometric multiplicity of the eigenvalue $\gamma_{\ell}$ and the size of the Jordan chains is $\tilde{\tau}_{\ell,j}, 1\leq j\leq N_{\ell}$. As we have assumed that the $\tau_{l,j}=\tau_l$ for the case of the matrix $A^*$, here also we assume that $\tilde{\tau}_{\ell,j}=\tilde{\tau}_{\ell}$ is independent of $j$. We set $\hat{N}=\max_{1\leq \ell\leq \tilde{p}}N_{\ell}$. Now, the structure of the eigenvalue $\gamma_{\ell}$ for the matrix operator $\mathbf{L}^*_{K_0}$ is $\{\gamma_{\ell}\}_{1\leq \ell\leq \tilde{p}}=\{-\lambda_{ \alpha_1, k}+\mu_l\}_{1\leq k\leq K_0, 1\leq l\leq {p}}$. From here, we can relabel the full sequence $\{\Lambda_k\}_{k\geq 1}$ and deduce \eqref{ar eigen}. This ends the proof.
\end{proof}

\begin{proposition}\label{prop:verify}
Let $\{\Lambda_k\}_{k\geq 1}$ be the eigenvalues of $-\mathcal{A}_{\alpha_1,A^*}$ rearranged as in \Cref{lem:rea}. The family $\{\Lambda_k\}_{k\geq 1}$ verifies conditions \ref{item1} to \ref{item7} of \Cref{d1}. 
\end{proposition}

\begin{proof}
We verify the six conditions of \Cref{d1}.
\begin{itemize}
	\item Condition \ref{item1} is immediate by the arrangements \eqref{ar eigen} in \Cref{lem:rea} of the eigenvalues.
	\item Condition \ref{item2} is obvious for higher frequencies. For the lower frequencies it does also hold by \Cref{rm1} of \Cref{rm}. 
	\item By definition, $|\Im(\Lambda_k)|=|\Im(\mu_l)|\leq \max_{1\leq l\leq p}|\Im(\mu_l)|$ and also $\Re(\Lambda_k)\geq \lambda_1- \max_{1\leq l\leq p}|\Re(\mu_l)|,$ a positive number as $A$ is stable. Therefore condition \ref{item3} follows.
	\item \eqref{con1} and \eqref{con2} imply condition \ref{item4}.
	\item Thanks to definition \eqref{ar eigen} in \Cref{lem:rea}, we have $\displaystyle	\inf_{k\neq n: |k-n|<q}\left\vert \Lambda_{k}-\Lambda_{n} \right\vert>0$, for any $q\in \N.$
	Now we show the first part of condition \ref{item5}. Let us first choose, $k=\tilde p+i_k, n=\tilde p+i_n,$ (for the case $k\leq \tilde p$ or $n\leq \tilde p$ the condition \ref{item5} easily follows). Using \eqref{ar eigen}, we write $\Lambda_k=\lambda_{ \alpha_1,   K_0+j_k}-\mu_{l_k}, \text{ and } \Lambda_n=\lambda_{ \alpha_1,   K_0+j_n}-\mu_{l_n} $ where 
	\begin{equation}\label{ik}
		\begin{cases}
			j_k=\lfloor \frac{i_k-1}{p}\rfloor+1 \text{ and } l_k=i_k-\lfloor\frac{i_k-1}{p}\rfloor p,\\
			j_n=\lfloor \frac{i_n-1}{p}\rfloor+1 \text{ and } l_n=i_n-\lfloor\frac{i_n-1}{p}\rfloor p. 
		\end{cases}
	\end{equation}
	Let us now compute the gap of the eigenvalues
	\begin{align*}
		|\Lambda_n-\Lambda_{k}|&\geq ||\lambda_{ \alpha_1,  K_0+j_n}-\lambda_{ \alpha_1,  K_0+j_k}|-|\mu_{l_n} -\mu_{l_k} ||^2\\
		&\geq|\lambda_{ \alpha_1,  K_0+j_n}-\lambda_{ \alpha_1,  K_0+j_k}|^2-2|\lambda_{ \alpha_1,  K_0+j_n}-\lambda_{ \alpha_1,  K_0+j_k}||\mu_{l_n} -\mu_{l_k}|+|\mu_{l_n} -\mu_{l_k}|^2.
	\end{align*}
	Let us denote $m=\min_{1\leq l,l'\leq p; \\l\neq l'} |\mu_l-\mu_{l'}|$, $M=\max_{1\leq l,l'\leq p; \\l\neq l'} |\mu_l-\mu_{l'}|, d=|j_k-j_n|, s=j_k+j_n$
	and $b=d(s+2K_0).$
	Using the gap condition \eqref{gap con} in  \Cref{lmm_gap} of the eigenvalues of the operator $-\mathcal{A}_{\alpha_1}$, we have \begin{equation}\label{gap_1}|\Lambda_n-\Lambda_{k}|^2\geq \rho_1 b^2-2Mb\rho_2+m^2.
	\end{equation}
	By the expression \eqref{ik} of $j_k$, we have $j_k\geq \frac{i_k-1}{p} \text{ and } j_k\leq \frac{i_k-1}{p}+1$. 
	This implies that $i_k\leq pj_k+1$ and $i_k\geq pj_k-p+1.$ Thus using these conditions we have
	\begin{equation*}
		|k^2-n^2|^2=|2\tilde p+i_k+i_n|^2|i_k-i_n|^2\leq |ps+2+2\tilde p|^2|pd+p|^2.
	\end{equation*}
	For large enough $d, s$ we have a constant $C>0$ such that
	\begin{equation*}
		|k^2-n^2|^2\leq Cd^2(s+2K_0)^2=C b^2.
	\end{equation*}
	As $b$ is large enough, using \eqref{gap_1}, we have the existence of a constant $\rho>0$ such that condition \ref{item5} holds.
	\item Finally, we will verify the assumption \ref{item7} regarding the counting function. The definition of the counting function \eqref{counting} and condition \ref{item4} imply that $\mathcal{N}(r)=n$ if and only if $|\Lambda_n|\leq r$ and $|\Lambda_{n+1}|>r$ and which further implies ${|\Lambda_{\mathcal{N}(r)}|}\leq r <{|\Lambda_{\mathcal{N}(r)+1}|}.$
	
	We choose $r$ in such a way that $\mathcal{N}(r)>\tilde p$ with $\tilde p$ as in \Cref{lem:rea}. Set $\mathcal{N}(r)=\tilde p+i.$ Then $\Lambda_{\mathcal{N}(r)}=\lambda_{ \alpha_1, K_0+j}-\mu_l,$ where $j=\lfloor \frac{i-1}{p}\rfloor+1 \text{ and } l=i-\lfloor\frac{i-1}{p}\rfloor p , i\geq 1$. Further we have $|\Lambda_{\mathcal{N}(r)}|\leq |\lambda_{ \alpha_1, K_0+j}|+\hat{M}$, where $\hat{M}=\max_{1\leq l\leq p}\mu_l$. Thanks to \eqref{first in}, we have for $\nu(\alpha_1) \in [0,\frac{1}{2}]$
	\begin{align}
	\nonumber	\kappa_{\alpha_1}^2\left(K_0+j+\frac{\nu(\alpha_1)}{2}-\frac{1}{4}\right)^2\pi^2-\hat{M}&\leq |\Lambda_{\mathcal{N}(r)}|
		\leq \kappa_{\alpha_1}^2\left(K_0+j+\frac{\nu(\alpha_1)}{4}-\frac{1}{8}\right)^2\pi^2+\hat{M}.
	\end{align}
	By the expression of $j$ we have $\frac{i-1}{p}\leq j\leq \frac{i-1}{p}+1 $. Setting  $c_1=-\frac{1+\tilde p}{p}+\frac{\nu(\alpha_1)}{2}-\frac{1}{4}+K_0,$ and 
$c_2=-\frac{1+\tilde p}{p}+1+\frac{\nu(\alpha_1)}{4}-\frac{1}{8}+K_0,$ one can obtain
	\begin{equation}\label{eq1}
		\kappa_{\alpha_1}^2\left(\frac{\mathcal{N}(r)}{p}+c_1\right)^2\pi^2-\hat{M}\leq |\Lambda_{\mathcal{N}(r)}|\leq  \kappa_{\alpha_1}^2\left(\frac{\mathcal{N}(r)}{p}+c_2\right)^2\pi^2+\hat{M}.
	\end{equation}
	Using the left inequality of \eqref{eq1} and the fact $\mathcal{N}(r)\leq r $, we have 
	\begin{align*}
		 \mathcal{N}(r)\leq \frac{1}{\kappa_{\alpha_1}}\frac{p}{\pi}\sqrt{r}+\left(\frac{p}{\kappa_{\alpha_1}\pi}\sqrt{\hat{M}}-pc_1\right).
	\end{align*}
	Using the right inequality of \eqref{eq1}, we have 
	\begin{align*}
		r<	|\Lambda_{\mathcal{N}(r)+1}|&\leq  \kappa_{\alpha_1}^2\left(\frac{\mathcal{N}(r)+1}{p}+c_2\right)^2\pi^2+\hat{M}\\
		&\leq \left(\kappa_{\alpha_1}\left(\frac{\mathcal{N}(r)+1}{p}+c_2\right)\pi+\sqrt{\hat{M}}\right)^2.
	\end{align*}
It can be easily checked that the condition \ref{item7} is verified. Similarly we can show this when $\nu(\alpha_1) \geq \frac{1}{2}.$ Therefore we complete the counting function argument \ref{item7} and the proof is finished.
\end{itemize}
\end{proof}
We are in position to prove the main result of this section, that is, \Cref{null control 1}. We give a full proof for weakly degenerate case and just mention the changes for strong one.
\subsection{Proof of the main controllability result in 1-$d$}
Here, we prove \Cref{null control 1}. For readability, we separate the proof in three parts. First, we present the proof of sufficiency of the Kalman rank condition (see eq. \eqref{kalman}) for weakly degenerate case. Then, we present the modifications needed for strongly degenerate one. We conclude by presenting the necessity of \eqref{kalman}.

\begin{proof}[Proof of \Cref{null control 1} for weakly degenerate case $(0\leq\alpha_1<1)$]
Let us choose $v_T\in H^{1 }_{ \alpha_1,0}(0,1)^n.$ The solution of the adjoint problem \eqref{adj system} is given by
\begin{equation*}
	v(t,x)=\sum_{k=1}^{\infty}  e^{(-\lambda_{ \alpha_1, k}I_d+A^*)(T-t)}\phi_{ \alpha_1, k}(x) v_{T,k}, \text{ where } v_{T,k}=\int_{0}^{1}v_T(x) \phi_{ \alpha_1, k}(x) \rd x \in \cplx^n.
\end{equation*}
This implies the observation term is
\begin{equation}\label{observation}
	B^*(x^{ \alpha_1}\pa_x v)(t,0)=\sum_{k=1}^{\infty}  B^* e^{(-\lambda_{ \alpha_1, k}I_d+A^*)(T-t)} (x^{ \alpha_1}\pa_x \phi_{ \alpha_1, k})|_{x=0} \, 
	 v_{T,k}.
\end{equation}
Thanks to \Cref{obs weak}, we have the expression $$(x^{ \alpha_1}\pa_x \phi_{ \alpha_1, k})|_{x=0}=\frac{(1-\alpha_1)\sqrt{(2- \alpha_1)}(j_{\nu({ \alpha_1 }),k})^{\nu(\alpha_1)} }{2^{\nu(\alpha_1)} \Gamma(\nu(\alpha_1)+1)\left|J'_{\nu({ \alpha_1})}(j_{\nu({ \alpha_1 }),k})\right|}.$$

\noindent
Let us set $V_T=\left(\frac{(1-\alpha_1)\sqrt{(2- \alpha_1)}(j_{\nu({ \alpha_1 }),k})^{\nu(\alpha_1)} }{2^{\nu(\alpha_1)} \Gamma(\nu(\alpha_1)+1)\left|J'_{\nu({ \alpha_1})}(j_{\nu({ \alpha_1 }),k})\right|} v_{T,k}\right)_{1\leq k\leq K_0} \in \cplx^{nK_0}$. Then the expression of \eqref{observation} can be written as:
{\small\begin{equation}\label{obs fin}
	B^*(x^{ \alpha_1}\pa_x v)(t,0)=\mathbf{B}^*_{K_0}e^{\mathbf{L}^*_{K_0}(T-t)}V_T+\sum_{k=K_0}^{\infty}  B^* e^{(-\lambda_{ \alpha_1, k}I_d+A^*)(T-t)} \frac{(1-\alpha_1)\sqrt{(2- \alpha_1)}(j_{\nu({ \alpha_1 }),k})^{\nu(\alpha_1)} }{2^{\nu(\alpha_1)} \Gamma(\nu(\alpha_1)+1)\left|J'_{\nu({ \alpha_1})}(j_{\nu({ \alpha_1 }),k})\right|} v_{T,k}.
\end{equation}}
Taking the duality product $\ip{\cdot}{\cdot}_{{H^{-1}_{\alpha}(\Omega)^n, H^1_{\alpha, 0}(\Omega)^n}}$ between $w_0$ and $v(0,\cdot)$, we obtain:
\begin{equation*}
	\ip{w_0}{v(0,\cdot)}=\sum_{k=1}^{\infty} \left(w_{0,k}, e^{(-\lambda_{ \alpha_1, k}I_d+A^*)T} v_{T,k}\right)_{\cplx^n}, \quad \text{where } w_{0,k}=\ip{w_0}{\phi_{\alpha_1,k}}\in \cplx^n. 
\end{equation*}
Let us choose $U_0=\left( \frac{2^{\nu(\alpha_1)} \Gamma(\nu(\alpha_1)+1)\left|J'_{\nu({ \alpha_1})}(j_{\nu({ \alpha_1 }),k})\right| }{ (1-\alpha_1)\sqrt{(2- \alpha_1)}(j_{\nu({ \alpha_1 }),k})^{\nu(\alpha_1)}} u_{0,k}\right)_{1\leq k\leq K_0} \in \cplx^{nK_0}$. Then as before, we represent the above duality product in the following form
\begin{equation}\label{inner pr}
	\ip{w_0}{v(0,\cdot)}=\left(U_0, e^{\mathbf{L}^*_{K_0}T}V_T\right)_{\cplx^{nK_0}}+\sum_{k=K_0}^{\infty} \left(w_{0,k}, e^{(-\lambda_{ \alpha_1, k}I_d+A^*)T} v_{T,k}\right)_{\cplx^n}. 
\end{equation}
Next, we first take $v_T$ arbitrarily from the space $\text{span}\{\phi_{ \alpha_1, k}\}_{1\leq k\leq K_0},$ and then taking $v_T=a\phi_{ \alpha_1, k}$ for $k>K_0, a\in \cplx^n$ as $\{\phi_{ \alpha_1, k}\}_{k\geq 1}$ forms an orthonormal basis in $L^2(0,1),$ using \eqref{obs fin} and \eqref{inner pr}, the moment problem \eqref{moment} reduces to the following:
\begin{equation}\label{moment h}
	\begin{cases}
		\displaystyle \int_{0}^{T}\left( h(T-t,)\mathbf{B}^*_{K_0}e^{\mathbf{L}^*_{K_0}(T-t)}V_T\right)_{\cplx^m} \rd t=F(U_0, V_T), \quad \forall\, V_T\,\in \cplx^{nK_0},\\
		\displaystyle \int_{0}^{T}\left( h(T-t),   B^* e^{(-\lambda_{ \alpha_1, k}I_d+A^*)t} a \right)_{\cplx^m} \rd t= G_k(w_0,a),
	\end{cases}
\end{equation}
where the function $F:\cplx^{nK_0}\times \cplx^{nK_0}\to \cplx$ and $G_k:H^{-1 }_{\alpha_1}(0,1)^n\times \cplx^n\to \cplx$ are defined by
\begin{equation*}
	\begin{cases}
		F(U_0, V_T)=-\left(U_0, e^{\mathbf{L}^*_{K_0}T}V_T\right)_{\cplx^{nK_0}},\\
		G_k(w_0,a)=-\frac{2^{\nu(\alpha_1)} \Gamma(\nu(\alpha_1)+1)\left|J'_{\nu({ \alpha_1})}(j_{\nu({ \alpha_1 }),k})\right| }{ (1-\alpha_1)\sqrt{(2- \alpha_1)}(j_{\nu({ \alpha_1 }),k})^{\nu(\alpha_1)}} \left(w_{0,k}, e^{(-\lambda_{ \alpha_1, k}I_d+A^*)T} a\right)_{\cplx^n}.
	\end{cases}
\end{equation*}
With this setting, using \cite[Proposition 5.1]{JMPA11} we now write down the simplified moment problem in a similar way as done in \cite[Section 3.2]{AB2014}.
Assuming the Kalman rank condition \eqref{kalman}, the control system \eqref{oned} is null controllable in time $T>0$ if for every $1\leq q\leq \hat{N},$ there exists a solution $h_q\in L^2(0,T)$ to the following moment problems
\begin{equation}\label{moment1}
	\begin{cases}
		\displaystyle \int_{0}^{T}\frac{t^\delta}{\delta!}e^{\overline{\gamma_{\ell}}t}h_q(t) \, \rd t=c_{\ell,\delta,q}(w_0;T),  \quad 1\leq \ell\leq \tilde{p}, 0\leq \delta\leq \tilde{\tau}_{\ell}-1,\\
		\displaystyle \int_{0}^{T}\frac{t^\sigma}{\sigma!}e^{({-\lambda_{ \alpha_1, k}+\overline{\mu_l}})t}h_q(t) \, \rd t=d^k_{l,\sigma,q}(w_0;T), \quad \forall k>K_0, 1\leq l\leq {p}, 0\leq \sigma\leq {\tau}_l-1,
	\end{cases}
\end{equation}
where $c_{l,\delta,q}$ and $d^k_{l,\sigma,q}$ have the following estimates
\begin{align*}
	&|c_{\ell,\delta,q}|\leq C \norm{e^{\mathbf{L}^*_{K_0}T}}_{\mc{M}_{nK_0}(\cplx)} \norm{w_0}_{H^{-1}_{ \alpha_1}(0,1)^n}\leq Ce^{CT}\norm{w_0}_{H^{-1 }_{\alpha_1}(0,1)^n}\\
	&|d^k_{l,\sigma,q}|\leq \frac{C}{(j_{\nu(\alpha_1),k})^{\nu(\alpha_1)+1/2}}\norm{e^{(-\lambda_{ \alpha_1, k}I_d+A^*)T}}_{\mc{M}_{n}(\cplx)}\left|\ip{w_0}{\phi_{ \alpha_1, k}}\right|\leq C ({j_{\nu( \alpha_1),k}})^{\frac{1}{2}-\nu(\alpha_1)}{e^{-\lambda_{ \alpha_1, k} T}} \norm{w_0}_{H^{-1}_{ \alpha_1}(0,1)^n},
\end{align*}
here we have used that $\norm{\phi_{ \alpha_1, k}}_{H^1_{\alpha_1,0}(0,1)}\leq C {j_{\nu( \alpha_1),k}},$ for some $C>0$ and the estimate \eqref{est obs weak} in \Cref{obs weak}. 
In this case, control $h(t)$ in \eqref{moment h} is given by a linear combination of $h_q(T-t), 1\leq q\leq \hat{N}$ and we can write
\begin{align}\label{main con}
	\norm{h}_{L^2(0,T)^m}\leq C \max_{1\leq q\leq \hat{N}}||h_q||_{L^2(0,T)}.
\end{align}
Next, we use a change of variable $s=t-\frac{T}{2}$ in the moment problem \eqref{moment1} so that we can use the biorthogonal result \Cref{d1} for the time interval $(-\frac{T}{2},\frac{T}{2})$. Using the binomial formula\\ $t^J=(s+\frac{T}{2})^J=\sum_{j=0}^{J}\begin{pmatrix}
	J\\j
\end{pmatrix}s^{J-j}(\frac{T}{2})^j$, we obtain 
\begin{equation*}
	\begin{cases}
		\displaystyle \sum_{j=0}^{\delta}\begin{pmatrix}
			\nu \\ j
		\end{pmatrix}(\tfrac{T}{2})^j\int_{-\frac{T}{2}}^{\frac{T}{2}}s^{\delta-j}e^{\overline{\gamma_{\ell}}s}h_q\left(s+\frac{T}{2}\right) \rd s=c'_{\ell,\delta,q}(w_0;T), \;\; 1\leq \ell\leq \tilde{p}, 0\leq \delta\leq \tilde{\tau}_{\ell}-1,\\
		\displaystyle \sum_{j=0}^{\sigma}\begin{pmatrix}
			\sigma \\ j
		\end{pmatrix}(\tfrac{T}{2})^j\int_{-\frac{T}{2}}^{\frac{T}{2}}s^{\sigma-j}e^{(-\lambda_{ \alpha_1, k}+\overline{\mu_l})s}h_q\left(s+\frac{T}{2}\right) \rd s={d^k}'_{l,\sigma,q}(w_0;T),\;\; \forall k>K_0,   1\leq l\leq {p}, 0\leq \sigma\leq {\tau}_l-1,
	\end{cases}
\end{equation*}
where $c'_{\ell,\delta,q}(w_0;T)=\delta! e^{-\frac{T}{2}\overline{\gamma_{\ell}}}c_{\ell,\delta,q}(w_0;T),{d^k}'_{l,\sigma,q}(w_0;T)=\sigma! e^{-(-\lambda_{ \alpha_1, k}+\overline{\mu_l})\frac{T}{2}}{d^k}_{l,\sigma,q}(w_0;T).$

Let us now choose $\eta=\max\{\tau_l,\tilde{\tau}_{\ell}, 1\leq l\leq p, 1\leq \ell\leq \tilde{p}\}$. Thanks to \Cref{prop:verify} and \Cref{d1}, there exists $T_0>0$ such that, for this $\eta\geq1$ and $0<T<T_0$, we can find a family of complex valued functions $\{\Psi_{k,j}\}_{k\geq1, 0\leq j\leq \eta-1}\in L^2\left(-\frac{T}{2},\frac{T}{2}\right)$ biorthogonal to 
$e_{k,j}=t^je^{-\Lambda_k t}.$ Thus noting the fact $-\lambda_{ \alpha_1, k}+\mu_l=-\Lambda_{\tilde p+(k-K_0-1)p+l}, \text{ for } k>K_0,$ we observe that if we consider the control function as
\begin{align}
	\nonumber	h_q(t)=\sum_{\ell=1}^{
		\tilde p}\sum_{\delta=0}^{\tilde{\tau}_{\ell}-1}c''_{\ell,\delta,q}(w_0;T) &\Psi_{\ell,\delta}\left(t-\frac{T}{2}\right)\\
	\label{control} &+\sum_{k=K_0}^{\infty}\sum_{\l=1}^{
		p}\sum_{\sigma=0}^{{\tau}_l-1}{d^k}''_{l,\sigma,q}(w_0;T) \Psi_{\tilde p+(k-K_0-1)p+l,\sigma}\left(t-\frac{T}{2}\right)
\end{align}
then $h_q$ satisfy the moment problem \eqref{moment1}, provided $h_q\in L^2(0,T)$ and $c''_{\ell,\delta,q}(w_0;T)$ and ${d^k}''_{l,\sigma,q}(w_0;T)$ solve the following system 
\begin{align*}
	P(T)\begin{pmatrix}
		c''_{\ell,0,q}(w_0;T)\\\vdots\\c''_{\ell,\tilde{\tau}_\ell-1,q}(w_0;T)
	\end{pmatrix}=\begin{pmatrix}
		c'_{\ell,0,q}(w_0;T)\\\vdots\\c'_{\ell,\tilde{\tau}_\ell-1,q}(w_0;T)
	\end{pmatrix}, \text{ and } Q(T)\begin{pmatrix}
		{d^k}''_{l,0,q}(w_0;T)\\\vdots\\{d^k}''_{l,{\tau}_l-1,q}(w_0;T)
	\end{pmatrix}=\begin{pmatrix}
		{d^k}'_{l,0,q}(w_0;T)\\\vdots\\{d^k}'_{l,{\tau}_l-1,q}(w_0;T)
	\end{pmatrix}
\end{align*}
where the coefficients of the matrix $P(T)$ and $Q(T)$. are respectively, given for $i\geq j,$ $p_{ij}(T)=\begin{pmatrix}
	i-1\\j-1
\end{pmatrix}(\frac{T}{2})^{i-j},$ $q_{ij}(T)=\begin{pmatrix}
	i-1\\j-1
\end{pmatrix}(\frac{T}{2})^{i-j}, $ and $p_{ij}(T)=q_{ij}(T)=0$ otherwise. We also get 
\begin{align*}
	\norm{P(T)^{-1}}_{\mc{M}_{\tilde{\tau}_{\ell}-1}(\cplx)}\leq C T^{\tilde{\tau}_{\ell}-1},\, \norm{Q(T)^{-1}}_{\mc{M}_{{\tau}_l-1}(\cplx)}\leq C T^{{\tau}_l-1}. 
\end{align*}
Now we recover the estimates for the scalar $c''_{\ell,\delta,q}(w_0;T)$ and ${d^k}''_{l,\sigma,q}(w_0;T).$
\begin{align}
	\label{c}	&|c''_{\ell,\delta,q}(w_0;T)|\leq C T^{\tilde{\tau}_{\ell}-1} |e^{-\frac{T}{2}\overline{\gamma_{\ell}}}|e^{CT}\norm{w_0}_{H^{-1 }_{ \alpha_1}(0,1)^n}\leq  C e^{CT}\norm{w_0}_{H^{-1 }_{ \alpha_1}(0,1)^n}\\
	\nonumber	&|{d^k}''_{l,{\tau}_l-1,q}(w_0;T)|\leq C T^{{\tau}_l-1}({j_{\nu( \alpha_1),k}})^{\frac{1}{2}-\nu(\alpha_1)} |e^{-(-\lambda_{ \alpha_1, k}+\overline{\mu_l})\frac{T}{2}}|e^{CT}e^{-\lambda_{ \alpha_1, k} T}\norm{w_0}_{H^{-1 }_{ \alpha_1}(0,1)^n}\\
	\label{d}	&\hspace{2.84cm}\leq Ce^{CT} ({j_{\nu( \alpha_1),k}})^{\frac{1}{2}-\nu(\alpha_1)}e^{-\lambda_{ \alpha_1, k}\frac{T}{2}}\norm{w_0}_{H^{-1 }_{ \alpha_1}(0,1)^n}.
\end{align}
Next, we show that $h_q$ indeed is in $L^2(0,T).$ Thanks to the estimate \eqref{cost}, \eqref{c} and \eqref{d}, we have from the definition \eqref{control},
\begin{align*}
	\norm{h_q}_{L^2(0,T)}&\leq C e^{CT}\bigg(\sum_{\ell=1}^{\tilde p}e^{C\sqrt{-\Re(\gamma_{\ell})}+\frac{C}{T}}
	+\sum_{k=K_0}^{\infty}({j_{\nu( \alpha_1),k}})^{\frac{1}{2}-\nu(\alpha_1)}e^{-\lambda_{ \alpha_1, k}\frac{T}{2}}\sum_{l=1}^{p}e^{C\sqrt{\lambda_{ \alpha_1, k}-\Re(\mu_l)}+\frac{C}{T}}\bigg)\norm{w_0}_{H^{-1 }_{ \alpha_1}(0,1)^n}\\
	&\leq Ce^{CT+\frac{C}{T}}\bigg(1+\sum_{k=K_0}^{\infty}({j_{\nu( \alpha_1),k}})^{\frac{1}{2}-\nu(\alpha_1)}e^{-\lambda_{ \alpha_1, k}\frac{T}{2}}e^{C\sqrt{\lambda_{ \alpha_1, k}}}\bigg)\norm{w_0}_{H^{-1 }_{ \alpha_1}(0,1)^n}.
\end{align*}
Using Young's inequality, for every $k\geq 1, T>0$ we have $C\sqrt{\lambda_{ \alpha_1, k}}\leq \lambda_{ \alpha_1, k} \frac{T}{4}+\frac{C^2}{T}$ and using the expression of the eigenvalues $\lambda_{ \alpha_1, k}=\kappa_{\alpha_1}^2 j^2_{\nu( \alpha_1 ),k}$ and its bound \eqref{first in}--\eqref{second in} we finally obtain,
\begin{align*}
	\norm{h_q}_{L^2(0,T)}&\leq Ce^{CT+\frac{C}{T}} \sum_{k=1}^{\infty}e^{-\lambda_{ \alpha_1, k} \frac{ C T}{4}} \norm{w_0}_{H^{-1 }_{ \alpha_1}(0,1)^n}\leq Ce^{CT+\frac{C}{T}} \sum_{k=1}^{\infty}e^{-(k+c)^2 \frac{ C T}{4}} \norm{w_0}_{H^{-1 }_{ \alpha_1}(0,1)^n},
\end{align*}
for some real numbers $ C, c>0$. Next using Gauss integral, we have 
\begin{align*}
	\norm{h_q}_{L^2(0,T)}&\leq Ce^{CT+\frac{C}{T}}\sqrt{\frac{C}{T}}\norm{w_0}_{H^{-1 }_{ \alpha_1}(0,1)^n} \leq Ce^{CT+\frac{C}{T}}\norm{w_0}_{H^{-1 }_{ \alpha_1}(0,1)^n}.
\end{align*}
As we consider $T<T_0,$ using \eqref{main con} we get the existence of the control with the cost $$\norm{h}_{L^2(0,T)^m}\leq Ce^{\frac{C}{T}}\norm{w_0}_{H^{-1 }_{ \alpha_1}(0,1)^n}.$$
The case $T\geq T_0$ can be reduced to the previous one. In fact, any continuation by zero of a control on $(0,\frac{T_0}{2})$ is a control on $(0,T)$ and the estimate follows from the decrease of the cost with respect to time.

\noindent
This ends the proof for weakly degenerate case. 
\end{proof}
In the case of strong degeneracy, we only indicate the main changes in the previous proof and everything else will work. 

\begin{proof}[Proof of \Cref{null control 1} in the strong degenerate case $(1<\alpha_1<2)$]
Note that, in this case the observation term is
\begin{equation*}
	B^*(v)(t,0)=\sum_{k=1}^{\infty}  B^* e^{(-\lambda_{ \alpha_1, k}I_d+A^*)(T-t)} ( \varphi_{ \alpha_1, k})|_{x=0} \, 
	v_{T,k}.
\end{equation*}
Thanks to \Cref{obs strong}, we have the expression $$( \phi_{ \alpha_1, k})|_{x=0}=\frac{\sqrt{2\kappa_{\alpha_1}}(j_{\nu({ \alpha_1 }),k})^{\nu(\alpha_1)} }{2^{\nu(\alpha_1)} \Gamma(\nu(\alpha_1)+1)\left|J'_{\nu({ \alpha_1})}(j_{\nu({ \alpha_1 }),k})\right|}.$$
Our next job is to take $V_T, U_0$ suitably 
and the whole proof works in the same way as weak degenerate case. In particular we choose 
\begin{align*}V_T&=\left(\frac{{\sqrt{2\kappa_{\alpha_1}}}(j_{\nu({ \alpha_1 }),k})^{\nu(\alpha_1)} }{2^{\nu(\alpha_1)} \Gamma(\nu(\alpha_1)+1)\left|J'_{\nu({ \alpha_1})}(j_{\nu({ \alpha_1 }),k})\right|} v_{T,k}\right)_{1\leq k\leq K_0}, \\
U_0&=\left( \frac{2^{\nu(\alpha_1)} \Gamma(\nu(\alpha_1)+1)\left|J'_{\nu({ \alpha_1})}(j_{\nu({ \alpha_1 }),k})\right| }{ \sqrt{2\kappa_{\alpha_1}}(j_{\nu({ \alpha_1 }),k})^{\nu(\alpha_1)}} u_{0,k}\right)_{1\leq k\leq K_0}
.\end{align*}
\end{proof}
\begin{proof}[Necessity of the Kalman rank condition \eqref{kalman}]  
	Let us recall that there exists $ k_0\in \N$ such that \eqref{distinct} holds. If possible, assume that
	\begin{align*}
			\text{rank } 	\boldsymbol{\mathbf{\mathcal{K}}}_{k_0}= 	\text{rank } [\mathbf{L}_{k_0} | \mathbf{B}_{k_0}]=	\text{rank } \left[\mathbf{B}_{k_0},\, \mathbf{L}_{k_0}\mathbf{B}_{k_0},\,  \mathbf{L}^{2}_{k_0} \mathbf{B}_{k_0},...,\mathbf{L}^{nk-2}_{k_0}\mathbf{B}_{k_0},\, \mathbf{L}^{nk-1}_{k_0} \mathbf{B}_{k_0}\right]<nk_0.
	\end{align*}
This condition implies that the ODE represented by the pair $(\mathbf{L}_{k_0},\mathbf{B}_{k_0})$ is not controllable. Duality argument gives that the pair $(\mathbf{L}^*_{k_0},\mathbf{B}^*_{k_0})$ does not satisfy the observability inequality. In other words, there exists some nontrivial $V_T\in \cplx^{nk_0}$ such that the solution of the adjoint system 
\begin{equation*}
	\begin{cases}
		V'(t)+\mathbf{L}^*_{k_0} V(t)=0 & \text{ in } (0,T),\\
		V(T)=V_T,
	\end{cases}
\end{equation*}
satisfies $\mathbf{B}^*_{k_0}V(t)=0$ for all $t\in (0,T).$ If we take,
 $V_T=\left(\frac{(1-\alpha_1)\sqrt{(2- \alpha_1)}(j_{\nu({ \alpha_1 }),k})^{\nu(\alpha_1)} }{2^{\nu(\alpha_1)} \Gamma(\nu(\alpha_1)+1)\left|J'_{\nu({ \alpha_1})}(j_{\nu({ \alpha_1 }),k})\right|} v_{T,k}\right)_{1\leq k\leq k_0} \in \cplx^{n k_0}$, where $v_{T,k}\in \cplx^n $ for $1\leq k\leq k_0$ and $v_T=\sum_{k=1}^{k_0}v_{T,k} \phi_{\alpha_1,k}$, we have $v_T\in H^1_{\alpha_1}(0,1)^n.$ Let us now consider the 1-$d$ adjoint system \eqref{adj system} corresponding to this $v_T.$ The solution of the adjoint problem \eqref{adj system} is given by
 \begin{equation*}
 	v(t,x)=\sum_{k=1}^{k_0}  e^{(-\lambda_{ \alpha_1, k}I_d+A^*)(T-t)}\phi_{ \alpha_1, k}(x) v_{T,k}.
 \end{equation*}
 This implies the observation term for the weakly degenerate case is
\begin{align*}
	B^*(x^{ \alpha_1}\pa_x v)(t,0)=\sum_{k=1}^{k_0}  B^* e^{(-\lambda_{ \alpha_1, k}I_d+A^*)(T-t)} (x^{ \alpha_1}\pa_x \phi_{ \alpha_1, k})|_{x=0} 
	v_{T,k}.
\end{align*}
Using the expression of the observation term (see \Cref{obs weak}) 
 \begin{align*}
	B^*(x^{ \alpha_1}\pa_x v)(t,0)= \mathbf{B}^*_{k_0} e^{(-\mathbf{L}^*_{k_0})(T-t)} V_T=\mathbf{B}^*_{k_0} V(t)=0.
\end{align*}
 Hence the system \eqref{adj system} is not observable. Thus we proved that a necessary condition for the null controllability of 1-$d$ degenerate parabolic equation \eqref{oned} is $\text{rank } 	\boldsymbol{\mathbf{\mathcal{K}}}_{k_0}= nk_0.$ Noting the fact that this implies the Kalman rank condition \eqref{kalman} (thanks to \cite[Corollary 3.3]{JMPA11}), we conclude our proof.
	\end{proof}

\section{Boundary controllability of the degenerate system in 2-$d$}\label{lr sc}

In this section, we will prove our main 2-$d$ result, that is, \Cref{main theorem}. As we have mentioned in \Cref{sec:strat}, we will exploit the geometrical configuration of our problem and employ the 1-$d$ controllability result proved in \Cref{null control 1} in combination with the Lebeau-Robbiano technique. Here we follow the strategy developed in \cite{AB2014}.

Let us consider  $(	\lambda_{\alpha_2,k}, 	\phi_{\alpha_2,k})_{k\in \N}$ the corresponding eigenvalues and eigenfunctions of the problem \begin{equation}\label{eigen eqn2}
	\begin{cases}
		-(y^{\alpha_2}\phi')'(y)=\lambda \phi(y) \quad y\in (0,1),\\
		\begin{cases}
			\phi(0)=0 &\text{ if } 0\leq \alpha_2<1\\
			(y^{\alpha_2}\phi')(0)=0 &\text{ if } 1< \alpha_2<2
		\end{cases},\\
		\phi(1)=0,
	\end{cases}
\end{equation}
and introduce the following closed subspace of $ H^{1 }_{ \alpha,0}(\Omega)^n$ 
\begin{equation}\label{E_aJ}
	E_{\alpha,J}=\bigg\{\sum\limits_{j=1}^{J}\ip{u}{\phi_{\alpha_2,j}}_{L^2(0,1)}\phi_{\alpha_2,j} \mid  u\in  H^{1 }_{ \alpha,0}(\Omega)^n\bigg\}\subset H^{1 }_{ \alpha,0}(\Omega)^n, \quad J\geq1,
\end{equation}
where the notation $\sum\limits_{j=1}^{J}\ip{u}{\phi_{\alpha_2,j}}_{L^2(0,1)}\phi_{\alpha_2,j}$ means the following
\begin{equation*}
	(x,y)\mapsto \sum\limits_{j=1}^{J}\begin{pmatrix}\ip{u_1(x,\cdot)}{\phi_{\alpha_2,j}}_{L^2(0,1)}\phi_{\alpha_2,j}(y)\\
		\ip{u_2(x,\cdot)}{\phi_{\alpha_2,j}}_{L^2(0,1)}\phi_{\alpha_2,j}(y)\\
		\vdots\\
		\ip{u_n(x,\cdot)}{\phi_{\alpha_2,j}}_{L^2(0,1)}\phi_{\alpha_2,j}(y)
		\end{pmatrix}.
\end{equation*}
\begin{lemma}\label{exp of u}
	Any function $u\in H^{1 }_{ \alpha,0}(\Omega)^n$ has the following representation:
	\begin{equation*}
		u=\sum\limits_{j=1}^{\infty}\ip{u}{\phi_{\alpha_2,j}}_{L^2(0,1)}\phi_{\alpha_2,j}.
	\end{equation*}
\end{lemma}
\begin{proof}
	Let us show that the sequence $\{S_J u\}_{J \geq 1}$ defined by $S_J u = \sum_{j=1}^J \langle u, \phi_{\alpha_2, j} \rangle_{L^2(0,1)} \phi_{\al_2, j}$ is a Cauchy sequence in $H^1_{\alpha,0}(\Omega)^n$. For any $J > K \geq 1$, we have
		\begin{equation*}
	\| S_J u - S_K u \|^2_{H^1_{\alpha,0}(\Omega)^n} =
	\left\| \sum_{j=K+1}^J \langle u, \phi_{\alpha_2, j} \rangle_{L^2(0,1)} \phi_{\alpha_2, j} \right\|^2_{H^1_{\alpha,0}(\Omega)^n}.
		\end{equation*}
	Thanks to \eqref{norm}, we have
		\begin{equation*}
		\| S_J u - S_K u \|^2_{H^1_{\alpha,0}(\Omega)^n} =\sum_{j=K+1}^J \left\| \langle u, \phi_{\alpha_2, j} \rangle_{L^2(0,1)} \right\|^2_{H^1_{\alpha_1,0}(0,1)^n}
	+ \sum_{j=K+1}^J \lambda_{\alpha_2, j} \left\| \langle u, \phi_{\alpha_2, j} \rangle_{L^2(0,1)} \right\|^2_{L^2(0,1)^n}.
		\end{equation*}
	Using Lebesgue’s dominated convergence theorem, it is easy to follow that the above terms go to zero as $J, K \to +\infty$. Thus we obtain
	$
	S_J u \xrightarrow[H^1_{\alpha,0}(\Omega)]{J \to +\infty} v$ for some  $v \in H^1_{\alpha,0}(\Omega)^n.
	$
	In particular,	$\langle v, \phi_{\alpha_1, k} \phi_{\alpha_2, j} \rangle_{L^2(\Omega)} = \langle u, \phi_{\alpha_1, k} \phi_{\al_2, j} \rangle_{L^2(\Omega)}$ for every $j, k \geq 1,$
	and it follows that $u = v$.
	\end{proof}
\subsection{Partial Observability}
 Using the Lebeau-Robbiano spectral inequality \eqref{lr} we have, for any $J\in \N,$
\begin{equation}\label{lr1}
	\sum_{j=1}^{J}|a_j|^2\leq C e^{C\sqrt{\lambda_{\alpha_2,J}}}\int_{\omega}\left|\sum_{j=1}^{J}a_j \phi_{\alpha_2,j}(y)\right|^2 \rd y.
\end{equation}
where $\lambda_{\alpha_2,j}$ and $\phi_{\alpha_2,j}$ are the eigenvalues and eigenfunctions of \eqref{eigen eqn2}, respectively. 
\begin{proposition}\label{par obs}
Let $T>0$ be given and we assume \Cref{null control 1} holds. Let $C>0$ be the constant provided by \Cref{null control 1} and define $C_T:=Ce^{C/T}$. Then for all $\sigma_T\in H^1_{\alpha,0}(\Omega)^n,$ we have the following partial observability inequality 
	\begin{equation}\label{eq:obs_inequalities}
		\begin{cases}
		\displaystyle \norm{\Pi_{E_{\alpha,J}}\sigma(0)}_{H^{1 }_{ \alpha,0}(\Omega)^n}^2\leq  (C_T)^2 e^{2C\sqrt{\lambda_{\alpha_2,J}}}\int_{0}^{T}\left\|\mathbf{1}_{\gamma} B^*(  \mathbf{P} \sigma(t))\right\|_{L^2(\pa \Omega)^m}^2 \rd t & \text{ if } 0\leq \alpha_1<1,\\
		\displaystyle\norm{\Pi_{E_{\alpha,J}}\sigma(0)}_{H^{1 }_{ \alpha,0}(\Omega)^n}^2\leq  (C_T)^2 e^{2C\sqrt{\lambda_{\alpha_2,J}}}\int_{0}^{T}\left\|\mathbf{1}_{\gamma} B^*(\sigma)(t)\right\|_{L^2(\pa \Omega)^m}^2 \rd t& \text{ if } 1<\alpha_1<2,\\
		\end{cases}
	\end{equation}
where $\sigma$ is the solution of the two-dimensional adjoint system \eqref{adjintro}--\eqref{bd2} with $\sigma(T)=\sigma_T$.
\end{proposition}
\begin{proof}
	Let us first assume that $\sigma_T=\sum\limits_{j=1}^{J}\sigma^{j}_T(x)\phi_{\alpha_2,j}(y),$ for some $\sigma^j_T\in H^{1 }_{ \alpha_1,0}(0,1)^n.$ 
Thus if we write the solution $\sigma$ of the adjoint system \eqref{adjintro}--\eqref{bd2} as
\begin{equation}\label{exp_sigma}
\sigma(t,x,y)=\sum\limits_{j=1}^{J}\sigma^j(t,x)\phi_{\alpha_2,j}(y),
\end{equation}
 then $\sigma^j$ satisfy the following
\begin{equation}\label{adj2}
	\begin{cases}
		\partial_t \sigma^j+\pa_{x}(x^{\alpha_1}\pa_{x}\sigma^j)-\lambda_{\alpha_2,j}\sigma^j+A^*\sigma^j=0 &  t\in (0,T), x\in (0,1),\\
		\begin{cases}\sigma^j(t,0)=0 & \text{ if } 0\leq \alpha_1<1\\
			(x^{\alpha_1}\pa_x \sigma^j)(t,0)=0 & \text{ if } 1< \alpha_1<2
		\end{cases}, &\text{ in }   (0,T),\\
		 \sigma^j(t,1)=0 & t\in (0,T),\\
		\sigma^j(T,x)=\sigma^j_T(x) & x\in (0,1).
	\end{cases}
\end{equation}
Note that $\sigma^j$ can be written as $\sigma^j(t,\cdot)=e^{-\lambda_{\alpha_2,j}(T-t)}v(t,\cdot),$ where $v$ is the solution of the adjoint system \eqref{adj system} associated with $\sigma_T^j.$ By \Cref{null control 1} and a functional analysis argument (see e.g. \cite[Theorem 2.44]{Cor08}), system \eqref{adj2} satisfies the following observability inequality for each $1\leq j\leq J$ 
\begin{equation}\label{obs1}
	\begin{cases}
	\norm{\sigma^j(0)}^2_{H^{1 }_{ \alpha_1,0}(\Omega)^n}\leq (C_T)^2\displaystyle \int_{0}^{T}|B^*(x^{\alpha_1}\pa_x\sigma^j)(t,0)|_{\cplx^m}^2 \rd t & \text{ if } 0\leq \alpha_1<1,\\
		\norm{\sigma^j(0)}^2_{H^{1 }_{ \alpha_1,0}(\Omega)^n}\leq (C_T)^2\displaystyle \int_{0}^{T}|B^*(\sigma^j)(t,0)|_{\cplx^m}^2 \rd t & \text{ if } 1<\alpha_1<2.
	\end{cases}
\end{equation}
Note that $\Pi_{E_{\alpha,J}}\sigma(0)=\sigma(0).$ From expression \eqref{exp_sigma} and using the computation of the $H^1_{\alpha,0}(\Omega)^n$-norm (see \Cref{sec:norm_comp}), we can write the following
\begin{align*}
\norm{\sigma(0)}^2_{H^{1 }_{ \alpha,0}(\Omega)^n}=\sum_{j=1}^{J}\norm{\sigma^j(0)}^2_{H^1_{\alpha_1,0}(0,1)^n}+\sum_{j=1}^{J} \lambda_{\alpha_2,j}\norm{\sigma^j(0)}^2_{L^2(0,1)^n}.
\end{align*}
Thanks to weighted Hardy-Poincare inequality \eqref{HPIi} and \eqref{weightespoincare1},
we obtain
\begin{align}\label{poin}
	\norm{\sigma(0)}^2_{H^{1 }_{ \alpha,0}(\Omega)^n}\leq C\lambda_{\alpha_2,J}\sum_{j=1}^{J}\norm{\sigma^j(0)}^2_{H^1_{\alpha_1,0}(0,1)^n}.
\end{align}
Combining \eqref{obs1} and \eqref{poin}, we further have
\begin{align}\label{ob}
	\begin{cases}
	\displaystyle \norm{\sigma(0)}^2_{H^{1 }_{ \alpha,0}(\Omega)^n}\leq C\lambda_{\alpha_2,J}(C_T)^2 \int_{0}^{T} \sum_{j=1}^{J}|B^*(x^{\alpha_1}\pa_x\sigma^j)(t,0)|_{\cplx^m}^2 \rd t  & \text{ if } 0\leq \alpha_1<1,\\
	\displaystyle \norm{\sigma(0)}^2_{H^{1 }_{ \alpha,0}(\Omega)^n}\leq C\lambda_{\alpha_2,J}(C_T)^2 \int_{0}^{T} \sum_{j=1}^{J}|B^*(\sigma^j)(t,0)|_{\cplx^m}^2 \rd t  & \text{ if } 1<\alpha_1<2.
	\end{cases}
\end{align}
Let us denote $b_k$ as the $k$-th column of the matrix $B\in \mathcal{L}(\mathbb C^m,\mathbb C^n)$. Next, we put in the Lebeau-Robbiano inequality \eqref{lr1}, $a^k_j=b_k^{*}(x^{\alpha_1}\pa_x\sigma^j)(t,0)$ for  weakly degenerate case (resp. $a^k_j=b_k^{*}(\sigma^j)(t,0)$ for strongly degenerate one), thus obtaining for all $1\leq k\leq m$
\begin{equation}\label{lr11}
		\begin{cases}
	\displaystyle	\sum_{j=1}^{J}|b_k^{*}(x^{\alpha_1}\pa_x\sigma^j)(t,0)|^2\leq C e^{C\sqrt{\lambda_{\alpha_2,J}}}\int_{\omega}\left|\sum_{j=1}^{J}b_k^{*}(x^{\alpha_1}\pa_x\sigma^j)(t,0) \phi_{\alpha_2,j}(y)\right|^2 \rd y& \text{ if } 0\leq \alpha_1<1,\\
	\displaystyle	\sum_{j=1}^{J}|b_k^{*}\pa\sigma^j)(t,0)|^2\leq C e^{C\sqrt{\lambda_{\alpha_2,J}}}\int_{\omega}\left|\sum_{j=1}^{J}b_k^{*}(\sigma^j)(t,0) \phi_{\alpha_2,j}(y)\right|^2 \rd y & \text{ if } 1< \alpha_1<2,
		\end{cases}
\end{equation}
Adding each inequality \eqref{lr11} for $1\leq k\leq m$, we deduce
\begin{equation}\label{ob1}
	\begin{cases}
	\displaystyle\sum_{j=1}^{J}\left|B^* (x^{\alpha_1}\pa_x\sigma^j)(t,0)\right|_{\cplx^m}^2\leq C e^{C\sqrt{\lambda_{\alpha_2,J}}}\int_{\omega}\left|\sum_{j=1}^{J}B^*(x^{\alpha_1}\pa_x\sigma^j)(t,0) \phi_{\alpha_2,j}(y)\right|_{\cplx^m}^2 \rd y  & \text{ if } 0\leq \alpha_1<1,\\
	\displaystyle\sum_{j=1}^{J}\left|B^* (\sigma^j)(t,0)\right|_{\cplx^m}^2\leq C e^{C\sqrt{\lambda_{\alpha_2,J}}}\int_{\omega}\left|\sum_{j=1}^{J}B^*(\sigma^j)(t,0) \phi_{\alpha_2,j}(y)\right|_{\cplx^m}^2 \rd y  & \text{ if } 1<\alpha_1<2.
	\end{cases}
\end{equation}
Integrating both sides of \eqref{ob1} with respect to $t$ over $(0,T)$ and then using \eqref{ob}, we have
the desired result.
\end{proof}
Let us recall the definition \eqref{E_aJ} of the space $E_{\alpha,J}$. We define $E_{\alpha,J}^{-1}:=-\mathbf{D} E_{\alpha,J}\subset H^{-1}_{\alpha}(\Omega)^n$, where we recall that $\mathbf{D}u=\left(\begin{array}{ccc}\text{div}\left(D\nabla u_1\right) & \cdots & \text{div}\left(D\nabla u_n\right)\end{array}\right)^\top,$ when $u= \left(\begin{array}{cccc}u_1 & u_2 & \cdots & u_n\end{array}\right)^\top.$
By classical duality arguments, we can prove the following. 
\begin{proposition}\label{thm_dual}
	Let $T>0.$ There exists $C_T>0$ such that the following two properties are equivalent
	\begin{itemize}
		\item For every $u_0\in E^{-1}_{\alpha,J},$ there exists a control $q\in L^2(0,T;L^2(\pa \Omega)^m)$ such that
		\begin{equation*}
			\begin{cases}
				\Pi_{E^{-1}_{\alpha,J}}u(T)=0\\
				\norm{q}_{L^2(0,T;L^2(\pa \Omega)^m)}\leq  C_Te^{C\sqrt{\lambda_{\alpha_2,J}}}\norm{u_0}_{H^{-1 }_{ \alpha}(\Omega)^n},
			\end{cases}
		\end{equation*} 
where $u$ is the solution of the system \eqref{DCP}--\eqref{bd}.
\item For all $\sigma_T\in {E_{\alpha,J}}$, the solution $\sigma$ of the adjoint system	 \eqref{adjintro}--\eqref{bd2}
satisfies \eqref{eq:obs_inequalities}.
\end{itemize}
\end{proposition}
Observe that, \Cref{par obs} combined with \Cref{thm_dual} implies the following
\begin{corollary}\label{cor_cost}
	For every $J\geq 1,$ and $u_0\in E^{-1}_{\alpha,J},$ there exists a control $q=q(u_0)\in L^2(0,T; L^2(\pa \Omega)^m)$ with
	\begin{equation*}
		\norm{q}_{L^2(0,T;L^2(\pa\Omega)^m)}\leq  (C_T) e^{C\sqrt{\lambda_{\alpha_2,J}}}\norm{u_0}_{H^{-1}_{\alpha}(\Omega)^n},
	\end{equation*}
such that the solution of equation \eqref{DCP}--\eqref{bd} satisfies $\Pi_{E^{-1}_{\alpha,J}}u(T)=0.$
\end{corollary}
\subsection{Dissipation along the $y$-direction}
\begin{proposition}\label{dis}
	Let us consider system \eqref{DCP}--\eqref{bd} and assume that in some time interval $(t_0,t_1)$ control $q=0$ and also assume that for all $J\geq 1,$ $\Pi_{E_{\alpha,J}^{-1}}u(t_0)=0.$ Then we have the following estimate
	\begin{equation*}
		\norm{u(t)}_{H^{-1}_{\alpha}(\Omega)^n}\leq Ce^{-\lambda_{\alpha_2,{J+1}}(t-t_0)}\norm{u(t_0)}_{H^{-1}_{\alpha}(\Omega)^n}, \forall t\in (t_0,t_1).
	\end{equation*}
\end{proposition}
\begin{proof}
	Let $ u(t_0)  
	=-\mathbf{D} \tilde{u}_0,$
\,$\tilde{u}_0 \in H^1_{\alpha,0} (\Omega)^n \).
		The assumption $ \Pi_{E_{\alpha,J}^{-1}} u(t_0) = 0 $ implies that $ \Pi_{E_{\alpha,J}} \tilde{u}_0 = 0 $.
		Let $ \tilde{u} $ be the solution in $ H^1_{\alpha,0} (\Omega)^n $ to
		\begin{equation}\label{DCP1}
		\begin{cases}
			\partial_t \tilde{u} =\mathcal{I}_n \mathbf{D} \tilde u+A \tilde u 
			& \text{in } (t_0, t_1) \times \Omega, \\
			\tilde{u}(t_0) = \tilde{u}_0 & \text{in } \Omega,
		\end{cases}
		\end{equation}
		with the homogeneous boundary conditions as \eqref{bd} with $q=0$ in the time interval $(t_0,t_1).$
		Since the matrix $ A $ is constant, one can check that
		\begin{equation*}
		u = -
		\mathbf{D} \tilde{u} 
		\quad \text{in } (t_0, t_1) \times \Omega,
		\end{equation*}
		and thus using the isomorphism between $H_{\alpha,0}^1(\Omega)^n$ and $H^{-1}_{\alpha}(\Omega)^n$  (see \eqref{iso} in \Cref{iso app})
		\begin{align}\label{norm equivalent1}
		\| u(t) \|_{H^{-1}_{\alpha}(\Omega)^n} = \| \tilde{u}(t) \|_{H^1_{\alpha,0}(\Omega)^n}, \quad \| u(t_0) \|_{H^{-1}_{\alpha}(\Omega)^n} = \| \tilde{u}_0 \|_{H^1_{\alpha,0}(\Omega)^n}.
		\end{align}
		As a consequence, it is enough to prove the dissipation estimate for regular data, namely,
		\begin{equation*}
		\| \tilde{u}(t) \|_{H^1_{\alpha,0}(\Omega)^n} \leq C e^{\lambda_{\alpha_2,{J+1}}{(t - t_0)}} \| \tilde{u}_0 \|_{H^1_{\alpha,0}(\Omega)^n} \quad \text{ for all } t \in (t_0, t_1),
		\end{equation*}
		where $ \tilde{u}_0 $ such that $ \Pi_{E_{\alpha,J}} \tilde{u}_0 = 0$, i.e., $\tilde{u}_0$ is of the form (see \cref{exp of u})
		\begin{equation*}
		\tilde{u}_0 = \sum_{j=J+1}^{+\infty} \tilde{u}_{0,j} \phi_{\alpha_{2,j}}, \quad \tilde{u}_{0,j} = \left\langle \tilde{u}_0, \phi_{\alpha_{2,j}} \right\rangle_{L^2(0,1)^n} \in H^1_{\alpha_1,0}(0,1)^n.
		\end{equation*}

\noindent		
Let us write the solution of the system \eqref{DCP1} in $(t_0,t_1)$ as
\begin{equation*}
	\tl u(t,x,y)=\sum\limits_{j=1}^{\infty}\tl u^j(t,x)\phi_{\alpha_2,j}(y),
\end{equation*}
where $\tl u^j$ satisfies the following equation in $(0,1)$
\begin{equation}\label{dcp2}
	\begin{cases}
		\partial_t \tl u^j =\pa_{x}(x^{\alpha_1}\pa_{x}\tl u^j)-\lambda_{\alpha_2,j}\tl u^j+A \tl u^j &  t\in (t_0,t_1), x\in (0,1),\\
		\begin{cases}\tl u(t,0)=0 & \text{ if } 0\leq \alpha_1<1\\
			(x^{\alpha_1}\pa_x \tl u)(t,0)=0 & \text{ if } 1 < \alpha_1<2
		\end{cases} &\text{ in }   (t_0,t_1),\\
		 \tl u^j(t,1)=0& t\in (t_0,t_1)\\
		\tl u^j(0,x)=\tl u_{0,j}(x) & x\in (0,1).
	\end{cases}
\end{equation}
In order to find the dissipation estimate, let us multiply both sides of \eqref{dcp2} by $\tl u^{j}$. Next, performing the integration by parts and noting the fact that we assume matrix $A$ is stable, we have
\begin{align*}
	\dfrac{\rd}{\rd t}\int_{0}^{1}|\tl u^j(t,x)|^2 \rd x\leq -2\int_{0}^{1}x^{\alpha_1}|\pa_x \tl u^j(t,x)|^2 \rd x-2\lambda_{\alpha_2,j}\int_{0}^{1}|\tl u^j(t,x)|^2 \rd x.
\end{align*} 
Integrating on $(t_0,t)$, we have
\begin{align}\label{diss 1}
	\norm{\tl u^j(t)}^2_{L^2(0,1)^n}\leq e^{-2\lambda_{\alpha_{2,j}}(t-t_0)}\norm{\tl u_{0,j}}^2_{L^2(0,1)^n},\quad t\in (t,t_0).
\end{align}
Similarly the both sides of \eqref{dcp2} by $\pa_x(x^{\alpha_1}\pa_x \tl u^{j})$ and performing the integration by parts using the fact that the matrix $A$ is stable we, have
\begin{align*}
\dfrac{\rd}{\rd t}\int_{0}^{1}| x^{\alpha_1} \pa_x \tl u^j(t,x)|^2 \rd x\leq -2\int_{0}^{1}|\pa_x(x^{\alpha_1}\pa_x \tl u^j(t,x))|^2 \rd x-2\lambda_{\alpha_2,j}\int_{0}^{1}|x^{\alpha_1}\pa_x \tl u^j(t,x)|^2 \rd x, \quad t\in (t,t_0).
\end{align*}
Performing integration over $(t_0,t)$ we have
\begin{align}\label{diss 2}
	\norm{\tl u^j(t)}^2_{H^1_{\alpha_1,0}(0,1)^n}\leq e^{-2\lambda_{\alpha_{2,j}}(t-t_0)}\norm{\tl u_{0,j}}^2_{H^1_{\alpha_1,0}(0,1)^n},\quad t\in (t,t_0).
\end{align}
Next, as \( \Pi_{E_{\alpha,J}} \tilde{u}_0 = 0 \) and \( A \) is constant, we have \(\Pi_{E_{\alpha,J}} \tilde{u}(t) = 0 \) for every \( t \in (t_0, t_1) \).
This gives that 	\begin{equation*}
	\tl u(t,x,y)=\sum\limits_{j=J+1}^{\infty}\tl u^j(t,x)\phi_{\alpha_2,j}(y).
\end{equation*}
Thus using the computation of the $H^1_{\alpha,0}(\Omega)^n$ norm (see \eqref{norm}), and thanks to the dissipation estimates \eqref{diss 1} and \eqref{diss 2} corresponding to the one-dimensional system \eqref{dcp2}, one can derive the following
\begin{align*}
	\norm{\tl u(t)}^2_{H^{1 }_{ \alpha,0}(\Omega)^n}&=\sum_{j=J+1}^{\infty}\norm{\tl u^j(t)}^2_{H^1_{\alpha_1,0}(0,1)^n}+\sum_{j=J+1}^{\infty} \lambda_{\alpha_2,j}\norm{\tl u^j(t)}^2_{L^2(0,1)^n}\\
	&\leq e^{-2\lambda_{\alpha_2,{J+1}}(t-t_0)}\left(\sum_{j=J+1}^{\infty}\norm{\tl u_{0,j}}^2_{H^1_{\alpha_1,0}(0,1)^n}+\sum_{j=J+1}^{\infty} \lambda_{\alpha_2,j}\norm{\tl u_{0,j}}^2_{L^2(0,1)^n}\right)\\
	&= e^{-2\lambda_{\alpha_2,{J+1}}(t-t_0)} \norm{\tl u_0}^2_{H^{1 }_{ \alpha,0}(\Omega)^n}.
\end{align*}
This estimate along with \eqref{norm equivalent1} provides the desired one.
\end{proof}
\subsection{Lebeau-Robbiano approach}\label{sec:lr}
In this section, we first prove the following boundary null controllability result for equation \eqref{DCP}--\eqref{bd}. 
\begin{theorem}\label{control_2d_suf}
	Let $T>0$ and $u_0\in H^{-1}_{\alpha}(\Omega)^n$. For any given $A\in \mathcal{L}(\cplx^n)$ and $B\in\mathcal{L}(\cplx^m;\cplx^n)$ verifying \eqref{kalman}, there exists a control $q\in L^2(0,T;L^2(\pa \Omega)^m)$ such that system \eqref{DCP}--\eqref{bd} satisfies $u(T)=0$.
\end{theorem}
\begin{proof}
	For ease of reading, we split the proof into several steps. 
	
	\smallskip
	\textbf{Step 1. A time-splitting procedure.}

We divide the time interval $(0,T)$ in a infinite sequence of smaller intervals with appropriate length which will be chosen later along with suitable cut off frequencies to steer the solution to $0$ in time $T.$

Let us choose $u_0\in H^{-1}_{\alpha}(\Omega)^n.$ We decompose the time interval $[0,T)$ as follows:
\begin{align*}
	[0,T)=\cup_{k=0}^{\infty}[a_k, a_{k+1}]
\end{align*}
with $a_0=0, a_{k+1}=a_k+2T_k, \, T_k=\frac{\hat \alpha}{\beta} 2^{-k\rho}$, where $\rho\in (0,1), \hat \alpha=\frac{\beta T}{2}(1-2^{-\rho}), \beta\in \N$ so that we have $2\sum\limits_{k=0}^{\infty}T_k=T.$ We choose the frequencies as $\gamma_k=\beta 2^k$, where $\beta$ is a natural number to be chosen later.

Next, for all $k\geq 0$, we construct a control $q$ and the solution $u$ of the system \eqref{DCP}--\eqref{bd} by induction as follows
\begin{equation*}
	q(t)=\begin{cases}
		q\left({\Pi_{E^{-1}_{\alpha,\gamma_k}}u(a_k)}\right)(t),& \text{ if } t \in (a_k, a_k+T_k),\\
		0,& \text{ if } t\in (a_k+T_k, a_{k+1}).
	\end{cases}
\end{equation*}
Our main goal is to show that the control $q\in L^2(0,T; L^2(\pa \Omega)^m)$ steers the solution $u$ to rest in time $T$.

\smallskip
\textbf{Step 2. Estimate on the interval $[a_k, a_k+T_k]$}. For each $k\geq 0$,
thanks to the continuity estimate \eqref{con est}, we have
\begin{equation}\label{uak}
	\norm{u(a_k+T_k)}_{H^{-1}_{\alpha}(\Omega)^n}\leq C\left(\norm{u(a_k)}_{H^{-1}_{\alpha}(\Omega)^n}+\norm{q}_{L^2(a_k,a_k+T_k;L^2(\pa \Omega)^m)}\right),
\end{equation}
since $T_k\leq T$. By \Cref{cor_cost}, we can build a control $q=q\left(\Pi_{E^{-1}_{\alpha,\gamma_k}}u(a_k)\right)$ such that $\Pi_{E^{-1}_{\alpha,\gamma_k}}u(T_k+a_k)=0$, verifying the estimate
\begin{align*}
	\norm{q}_{L^2(a_k,a_k+T_k;L^2(\pa \Omega)^m)}\leq  (C_{T_k}) e^{C\sqrt{\lambda_{\alpha_2,{\gamma_k}}}}\norm{\Pi_{E^{-1}_{\alpha,\gamma_k}}u(a_k)}_{H^{-1}_{\alpha}( \Omega)^n}.
\end{align*}
As $\norm{\Pi_{E^{-1}_{\alpha,\gamma_k}}}_{\mathcal{L}_{H^{-1}_{\alpha}(\Omega)^n}}\leq 1$, recalling that $C_T=Ce^{C/T}$ (see \Cref{par obs}), the above inequality can be written in the form, 
\begin{align*}
	\norm{q}_{L^2(a_k,a_k+T_k;L^2(\pa \Omega)^m)}\leq Ce^{C\left(\frac{1}{T_k}+\sqrt{\lambda_{\alpha_2,{\gamma_k}}}\right)} \norm{u(a_k)}_{H^{-1}_{\alpha}(\Omega)^n}.
\end{align*}
We also have \begin{align*}\frac{1}{T_k}=\frac{\beta}{\hat \alpha}2^{k\rho}\leq C\beta 2^{k} \text{ and } \sqrt{\lambda_{\alpha_2,{\gamma_k}}}\leq C \beta 2^{k},
\end{align*}
which yields
\begin{align}\label{control est lr}
	\norm{q}_{L^2(a_k,a_k+T_k;L^2(\pa \Omega)^m)}\leq Ce^{C\beta 2^k} \norm{u(a_k)}_{H^{-1}_{\alpha}(\Omega)^n}.
\end{align} 
Thanks to \eqref{uak}, we obtain
\begin{equation}\label{uak1}
	\norm{u(a_k+T_k)}_{H^{-1}_{\alpha}(\Omega)^n}\leq C\left(1+e^{C\beta 2^k}\right)\norm{u(a_k)}_{H^{-1}_{\alpha}(\Omega)^n}\leq  C e^{C\beta 2^k}\norm{u(a_k)}_{H^{-1}_{\alpha}(\Omega)^n}.
\end{equation}

\smallskip
\textbf{Step 3. Estimate on the interval $[a_k+T_k, a_{k+1}].$} Since $\Pi_{E^{-1}_{\alpha,\gamma_k}}u(a_k+T_k)=0,$ using the dissipation result in \Cref{dis} we have
\begin{equation}\label{est_step_3}
	\norm{u(a_{k+1})}_{H^{-1}_{\alpha}(\Omega)^n}\leq C e^{-\lambda_{\alpha_2,{\gamma_k+1}}T_k}\norm{u(a_k+T_k)}_{H^{-1}_{\alpha}(\Omega)^n}.
\end{equation}

\smallskip
\textbf{Step 4. Final Estimate.} Putting together \eqref{uak1} and \eqref{est_step_3}, we write 
\begin{equation*}
	\norm{u(a_{k+1})}_{H^{-1}_{\alpha}(\Omega)^n}\leq C e^{-\lambda_{\alpha_2,{\gamma_k+1}}T_k} e^{C\beta 2^k}\norm{u(a_k)}_{H^{-1}_{\alpha}(\Omega)^n}.
\end{equation*}
By induction we have
\begin{equation*}
	\norm{u(a_{k+1})}_{H^{-1}_{\alpha}(\Omega)^n}\leq C e^{\sum_{p=0}^{k}\left(-\lambda_{\alpha_2,{\gamma_p+1}}T_p+C\beta 2^p\right)} \norm{u_0}_{H^{-1}_{\alpha}(\Omega)^n}.
\end{equation*}
Thanks to \eqref{eigenvalue} and \eqref{first in}, we have \begin{align*}\lambda_{\alpha_2,{\gamma_p+1}}=\kappa_{\alpha_2}^2j_{\nu(\alpha_2), {\gamma_p+1}}^2\geq \kappa_{\alpha_2}^2\left((\beta 2^p+1)+\frac{\nu}{2}-\frac{1}{4}\right)^2 {\pi^2}>C_2\beta^2 2^{2p}. 
\end{align*}
Using the fact that  \begin{align*}
	-\lambda_{\alpha_2,{\gamma_p+1}}T_p=-\frac{\hat \alpha}{\beta}2^{-p\rho}\lambda_{\alpha_2,{\gamma_p+1}}\leq -C_3\beta 2^{p(2-\rho)},
\end{align*}
we deduce 
\begin{equation*}
\norm{u(a_{k+1})}_{H^{-1}_{\alpha}(\Omega)^n}\leq C e^{\sum_{p=0}^{k}\left(-C_3 \beta 2^{p(2-\rho)}+C\beta 2^p\right)} \norm{u_0}_{H^{-1}_{\alpha}(\Omega)^n}.
\end{equation*}
There exists a $l_0\in \N$ such that $\left(-C_3 \beta 2^{p(2-\rho)}+C\beta 2^p\right)\leq -C_4 \beta 2^{p(2-\rho)}, \,\, \forall p\geq l_0.$
Therefore, for all $k>l_0$ we have
\begin{equation}\label{lo}\sum_{p=0}^{k}\left(-C_3 \beta 2^{p(2-\rho)}+C\beta 2^p\right)\leq C'\beta-C_4\beta\sum_{p=l_0}^{k}2^{p(2-\rho)}\leq C'\beta-C''\beta 2^{k(2-\rho)}.
 \end{equation}
Finally, we obtain \begin{equation}\label{u at time ak}
	\norm{u(a_{k+1})}_{H^{-1}_{\alpha}(\Omega)^n}\leq C e^{-C\beta  2^{k(2-\rho)}} \norm{u_0}_{H^{-1}_{\alpha}(\Omega)^n}.
\end{equation}

\smallskip
\textbf{Step 5. Control function.} Estimates \eqref{control est lr}, \eqref{lo} and \eqref{u at time ak} show that  $q\in L^2(0,T; L^2(\pa \Omega)^m)$, as
\begin{align*}
	\norm{q}_{L^2(0,T; L^2(\pa \Omega)^m)}&=\sum_{k=0}^{\infty}\norm{q}_{L^2(a_k,a_k+T_k;L^2(\pa \Omega)^m)}\\
	&\leq Ce^{C\beta}\norm{u_0}_{H^{-1}_{\alpha}(\Omega)^n}+C\sum_{k=0}^{\infty}e^{C\left(\beta-{C_1\beta  2^{k(2-\rho)}}+\beta2^{k+1}\right)}\norm{u_0}_{H^{-1}_{\alpha}(\Omega)^n}\\
	&\leq Ce^{C\beta}\norm{u_0}_{H^{-1}_{\alpha}(\Omega)^n}+Ce^{C\beta}\left(\sum_{k=p_0}^{\infty}e^{-C''' 2^{k(2-\rho)}}\right)\norm{u_0}_{H^{-1}_{\alpha}(\Omega)^n}\\
	&\leq Ce^{C\beta}\norm{u_0}_{H^{-1}_{\alpha}(\Omega)^n}<\infty.
\end{align*}
We also have the controllability
\begin{equation*}
	\norm{u(T)}_{H^{-1}_{\alpha}(\Omega)^n}=\lim_{k\mapsto \infty}\norm{u(a_{k+1})}_{H^{-1}_{\alpha}(\Omega)^n}=0.
\end{equation*}
This ends the proof. 
\end{proof}
\begin{remark}\label{rem:con cost}
	To find the optimal control cost, let us first choose $T<1$ and without loss of generality, we take $\frac{1}{T}=N,$ where $N\in \N.$ Then choose $\beta=\frac{\rho_0}{T},$ where $\rho_0$ is some large natural number independent of $T.$ Note that by this choice of $\beta$, $\hat \alpha$ defined in above theorem is independent of $T$. Therefore by previous analysis we can say that the control for the system \eqref{DCP}--\eqref{bd} satisfies the optimal control cost $Ce^{C/T}.$ The case $T\geq 1$ is also reduced to the previous one. Indeed, any continuation by zero of a control on $(0,\frac{1}{2})$ is a control on $(0,T)$ and the estimate follows from the decrease of the cost with respect to time.
	\end{remark}
\subsection{Proof of the main controllability result in 2-$d$}

\begin{proof}[Proof of \Cref{main theorem}]
	The proof follows from combining the results shown in the previous sections. By \Cref{control_2d_suf}, we have proved the null controllability of the 2-$d$ degenerate parabolic system \eqref{DCP}-\eqref{bd} assuming the Kalman rank condition \eqref{kalman}. The necessary part is implied by the result in one dimension (see \Cref{null control 1}). Indeed, by using Fourier decomposition in the $y$-direction, one can prove that the controllability of the 2-$d$ degenerate system \eqref{DCP}--\eqref{bd} implies the controllability of the 1-$d$ degenerate system \eqref{oned}. Thus, as \eqref{kalman} is necessary for the controllability of 1-$d$ systems \eqref{oned}, it is also necessary for the null controllability of the 2-$d$ degenerate system. Hence, the proof of \Cref{main theorem} is now complete.
\end{proof}

\section{ Degenerate parabolic equation in $N$-dimensional cube}
\label{higher_d}

As mentioned in \Cref{remark_extensions}, we can extend \Cref{main theorem_scalar} to the $N$-dimensional setting. For simplicity, we present the weakly degenerate case.

Let us consider $\Omega=(0,1)^N\subset \mathbb R^{N}$ $(N\geq 3)$ and $\omega=(a_1,b_1)\times \ldots \times (a_{N-1},b_{N-1})\subset \mathbb R^{N-1}$ where each interval $(a_i,b_i)\subset (0,1)$. We introduce the following control system
\begin{equation}\label{eq:sys_dimN}
	\begin{cases}
		\partial_t u = \dv(D {\nabla u}) &\text{in } (0,T)\times\Omega,\\
		u(t)=1_{\gamma}q(t) &\text{on } (0,T)\times \pa \Omega, \\
		u(0)=u_0 &\text{in } \Omega,
	\end{cases}
\end{equation} 
where $\gamma=\{0\}\times \omega,$ $D=\text{diag}(x_1^{\alpha_1},x_2^{\alpha_2},...,x_N^{\alpha^{N}}) \in M_{N\times N}(\mathbb{R}) $, $0\leq\alpha_i<1$, $ i=1,2,...,N.$ We have the following controllability result for \eqref{eq:sys_dimN}.
\begin{theorem}\label{th_highd}
	Let $T>0$. For any $u_0\in H^{-1}_{\alpha}(\Omega)$, there exists a control $q\in L^2(0,T;L^2(\pa \Omega))$ such that system \eqref{eq:sys_dimN} satisfies $u(T)=0.$
\end{theorem}
\begin{proof}We prove this result using induction hypothesis for the dimension of the spatial domain.
\Cref{main theorem} and \Cref{rem:con cost} readily imply the controllability of \eqref{eq:sys_dimN} with an explicit control cost of $Ce^{\frac{C}{T}}$ for $N=2$. Next, we turn our attention to system \eqref{eq:sys_dimN} for $N=3$. Our task is to decompose the solution of this 3-$d$ system using the Fourier basis in the $x_3$-direction. Let $\sigma$ be the solution of adjoint system of \eqref{eq:sys_dimN}, that is,
\begin{equation}\label{adj_nd}
\begin{cases}
\partial_t \sigma + \dv(D\nabla \sigma)=0 &\textnormal{in } (0,T)\times \Omega,\\
\sigma(t) = 0 &\textnormal{on } (0,T)\times \partial \Omega, \\
\sigma(T)=\sigma_T &\textnormal{in } \Omega.
\end{cases}
\end{equation}

Let us write the solution to \eqref{adj_nd} in the following manner 
\begin{equation*}
	\sigma(t,x_1,x_2,x_3)=\sum\limits_{j=1}^{J}\sigma^j(t,x_1,x_2)\phi_{\alpha_3,j}(x_3),
\end{equation*}
where $\sigma^j$ satisfies the $2$-$d$ system
\begin{equation}\label{adj2_3d}
	\begin{cases}
		\partial_t \sigma^j + \dv(D {\nabla \sigma^j})-\lambda_{\alpha_3,j}\sigma^j=0 &\text{in } (0,T)\times\Omega',\\
		\sigma^j(t)=0 &\text{on } (0,T)\times \pa \Omega', \\
		\sigma^j(T)=\sigma^j_T &\text{in } \Omega',
		\end{cases}
	\end{equation} 
where $\Omega'=(0,1)\times(0,1)$. Introducing the change of variables $\sigma^j(t,\cdot)=e^{-\lambda_{\alpha_3,j}(T-t)}v$ where $v$ is the solution to 
\begin{equation*}
\begin{cases}
\partial_t v + \dv(D\nabla v)=0 &\textnormal{in } (0,T)\times \Omega,\\
v(t) = 0 &\textnormal{on } (0,T)\times \partial \Omega, \\
v(T)=\sigma_T^j &\textnormal{in } \Omega,
\end{cases}
\end{equation*}
and due to \Cref{main theorem_scalar}, the solution $\sigma^j$ of the associated adjoint \eqref{adj2_3d} satisfies the observability inequality
\begin{equation}\label{obs_2d-j}
	\displaystyle \norm{\sigma^j(0)}_{H^{1 }_{ \alpha,0}(\Omega')}^2\leq C (C_T)^2 \int_{0}^{T}\left\|\mathbf{1}_{\gamma'}   (D\nabla \sigma^j(t)) \right\|_{L^2(\pa \Omega')}^2 \rd t,
\end{equation}
for each $j\in\mathbb N$, $\gamma'=\{0\}\times(a_1,b_1)$, and where the positive constants $C$ and $C_T$ are independent of $j$.

Thanks to \eqref{obs_2d-j} and the spectral inequality \eqref{lr} applied to the $x_3$-direction with $\omega=(a_2,b_2)$, we prove the partial observability of the form \eqref{eq:obs_inequalities} in \Cref{par obs} for the 3-$d$ adjoint system \eqref{adj_nd}. Thus proceeding with the Lebeau-Robbiano approach as done in \Cref{control_2d_suf},  we can establish the controllability of the 3-$d$ system with the desired control cost. Finally, the null controllability of the $N$-dimensional system \eqref{eq:sys_dimN} can be achieved after a finite number of steps by induction.
\end{proof}

\begin{remark} We make the following remarks.

\begin{itemize}
\item Note that in \cite{AB2014}, for a system of heat equations, the authors considered the $N$-dimensional domain in the form $\Omega = (0,1) \times \Omega'$, where $\Omega'$ is any smooth bounded domain in $\mathbb{R}^{N-1}$. Since the Lebeau-Robbiano spectral inequality \eqref{eq:spec_ineq} holds for the usual Dirichlet Laplace operator $\Delta_{x'}$ in $\Omega'$, it suffices to decompose the solution once using the Fourier basis in the $x'$-direction. In contrast, in our case, the spectral inequality \eqref{lr} (see \Cref{thm:spec_ineq_degen}) holds when $\Omega' = (0,1)$. Consequently, we need to employ the strategy from \cite{AB2014} repeatedly, over a finite number of steps, to conclude the desired result in $\mathbb{R}^N$.
\item It would be very interesting to establish a direct spectral inequality for degenerate elliptic operators \(\dv(D\nabla u)\) in a general smooth bounded domain \(\Omega \subset \mathbb{R}^N\). For the case \(N = 2\), we believe this can be achieved by adapting the approach of \cite{cannarsa2016global} for parabolic Carleman estimates and following the methodology in \cite[Section 5]{LRL12}. This will be addressed in a forthcoming paper.
\end{itemize}

\end{remark}

We can also prove the controllability of \eqref{eq:sys_dimN} when some degenerate parameters $\alpha_i\in (1,2)$. Note, however, that this introduces some notational difficulties for considering the different combinations of boundary conditions according to the degenerate parameters $\alpha_i$.

For simplicity, let us consider the following example in 3-$d$: let $\Omega=(0,1)^3$, where $\omega=(a_1,b_1)\times (a_2,b_2)\subset \mathbb R^{2}$  is a bounded domain, $\alpha_1,\alpha_2\in [0,1),$ $\alpha_3\in (1,2)$, see \Cref{fig_cube_domain}.
We define the boundary conditions of the concerned equation in the same spirit of \eqref{bd}. Let us denote $\Gamma_1=\{0\}\times[0,1]\times[0,1],  \Gamma_2=[0,1]\times\{1\}\times[0,1], \Gamma_3=\{1\}\times[0,1]\times[0,1],   \Gamma_4=[0,1]\times\{0\}\times[0,1]$, $\Gamma_5=[0,1]\times[0,1]\times\{1\},$ $\Gamma_6=[0,1]\times[0,1]\times\{0\}$  so that $\partial \Omega=\cup_{i=1}^{6}\Gamma_i.$ We denote ${\Sigma}_i=(0,T)\times\Gamma_i$, $i=1,2,3,4,5,6$ $\Sigma=\cup_{i=1}^{6}{\Sigma}_i$, $\Sigma_{ij}={\Sigma}_i\cup {\Sigma}_j$, ${\Sigma}_{ijk}={\Sigma}_i\cup{\Sigma}_j\cup {\Sigma}_k$ and so.
We introduce the following control system
\begin{equation}\label{eq:sys_dim3}
	\begin{cases}
		\partial_t u = \dv(D {\nabla u}) \text{ in } (0,T)\times\Omega,\\
		u(t)=1_{\gamma}q(t) \text{ on } \Sigma_1, \quad u(t)=0  \text{ on } \Sigma_{2345}, \quad (D\nabla u(t))\nu=0 \text{ on } \Sigma_{6}, \\
		u(0)=u_0 \text{ in } \Omega,
	\end{cases}
\end{equation} 
where $\gamma=\{0\}\times \omega,$ $D=\text{diag}(x_1^{\alpha_1},x_2^{\alpha_2},x_3^{\alpha_{3}}) \in M_{3\times 3}(\mathbb{R}).$ Arguing as in \Cref{th_highd}, we can get the controllability for \eqref{eq:sys_dim3} and for any other possible combination of degenerate parameters. We skip the details. 

\begin{figure}[h!]
	\centering
	\begin{tikzpicture}[scale=3] 
		\draw[black, thick] (0,0,0) -- (1,0,0) -- (1,1,0) -- (0,1,0) -- cycle; 
		\draw[black, thick] (0,0,0) -- (0,0,1);
		\draw[black, thick] (1,0,0) -- (1,0,1);
		\draw[black, thick] (1,1,0) -- (1,1,1);
		\draw[black, thick] (0,1,0) -- (0,1,1);
		\draw[black, thick] (0,0,1) -- (1,0,1) -- (1,1,1) -- (0,1,1) -- cycle; 
		
		\draw[->, black, thick] (-0.1,0,0) -- (1.2,0,0) node[anchor=north west] {$x_1$};
		\draw[->, black, thick] (0,-0.1,0) -- (0,1.2,0) node[anchor=north west] {$x_3$};
		\draw[->, black, thick] (0,0,-0.1) -- (0,0,1.3) node[anchor=south east] {$x_2$};
		
		\node[blue] at (0.7,0.7,0.6) {\large $\Omega$};
		\node[black] at (0.15,0,0.2) {$0$};
		\node[black] at (1.10,0,0.2) {$1$};
		\node[black] at (0.1,0,1.2) {$1$};
		\node[black] at (-0.05,1.05,0) {$1$};
		\node[black] at (0.5,0,0.5) {$\Gamma_6$};
		\node[black] at (0,.7,0.5) {$\Gamma_1$};
		\node[black] at (0.8,.2,0) {$\Gamma_4$};
		
		\fill[red, opacity=0.6] (0,0.2,0.15) -- (0,0.2,0.65) -- (0,0.5,0.65) -- (0,0.5,0.15) -- cycle;
		\node[red] at (0,0.4,0.8) {\large $\omega$};
		
	\end{tikzpicture}
	\caption{Example of the domain $\Omega$ for equation \eqref{eq:sys_dim3} with $N=3$. The operator degenerates only on the faces $\Gamma_1$, $\Gamma_4$, and $\Gamma_6$. The red region, denoted by $\omega$, represents the control set, which is active at the degenerate boundary}
	\label{fig_cube_domain}
\end{figure}
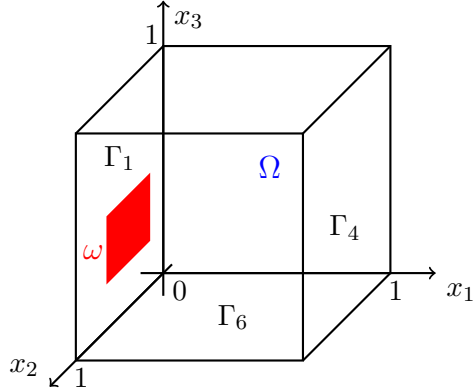


\section{Further results, remarks and open problems}\label{sec:open pr}
\subsection{Controlling from non-degenerate points}\label{non_deg}
In this paper, we studied the controllability of the 2-$d$ degenerate parabolic equation \eqref{DCP_sc} in a square domain $\Omega=(0,1)\times(0,1)$ by means of a boundary control acting through the points of degeneracy ($\gamma=\{0\}\times \omega\subset \pa \Omega$, $\omega\subset (0,1)$). One can establish a similar null controllability result for \eqref{DCP_sc}, when control is acting in the non-degenerate end
($\gamma=\{1\}\times \omega\subset \pa \Omega$, $\omega\subset (0,1)$). 

For simplicity of the definition of boundary data, we consider only the weakly degenerate case $(\alpha_1,\alpha_2)\in [0,1)\times [0,1)$. Let us introduce the following control system
\begin{equation}\label{DCP2}
	\begin{cases}
		\partial_t u=\dv(D {\nabla u})
			 &  \text{ in } (0,T)\times \Omega,\\
		u(t)=\mathbf{1}_{\gamma} q(t)& \text{ in } (0,T)\times \partial \Omega, \\
		u(0)=u_0 & \text{ in } \Omega,
	\end{cases}
\end{equation}
where $\gamma=\{1\}\times \omega\subset \pa \Omega,$ $\omega\subset (0,1)$ (see \Cref{fig:region1}). 
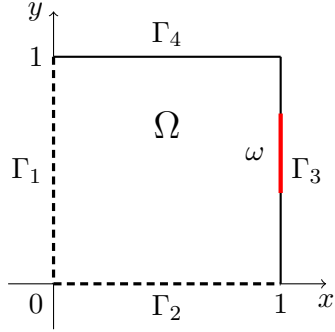
\begin{figure}[h!]
	\centering
	\begin{tikzpicture}[scale=3] 
		\draw[thin, -] (-0.2,0) -- (0,0); 
		\draw[thin, ->] (1,0) -- (1.2,0) node[below] {$x$}; 
		
		\draw[thin, -] (0,-0.2) -- (0,0); 
		\draw[thin, ->] (0,1) -- (0,1.2) node[left] {$y$}; 
		
		\draw[densely dashed, very thick] (0,0) -- (1,0) node[midway, below] {$\Gamma_2$};
		\draw[thick] (1,0) -- (1,1) node[midway, right] {$\Gamma_3$};
		\draw[thick] (1,1) -- (0,1) node[midway, above] {$\Gamma_4$};
		\draw[densely dashed, very thick] (0,1) -- (0,0) node[midway, left] {$\Gamma_1$};
		
		\draw[ultra thick, red] (1,0.4) -- (1,0.75); 
		\node[right] at (0.8,0.575) {$\omega$}; 
		
		\node[below left] at (0,0) {0};
		\node[below] at (1,0) {1};
		\node[left] at (0,1) {1};
		\node at (0.5, 0.7) {\Large$\Omega$}; 
	\end{tikzpicture}
	\caption{The domain $\Omega$ for equation \ref{DCP2}, with the operator degenerating along the dashed lines $\Gamma_1$ and $\Gamma_2$. The red region, denoted by $\omega$, represents the control set, which is active at the non degenerate boundary.}
	\label{fig:region1}
\end{figure}

Unfortunately, the technique of this work cannot be applied directly to equation \eqref{DCP2}. The advantage of taking control at the boundary $\Gamma_1$ is that, here, the unitary outward normal vector to each point on the boundary is $\nu=(1,0)$. Thus, the observation term $\mathbf{1}_{\{0\}\times \omega} ( D\nabla \sigma(t))=\mathbf{1}_{\{0\}\times \omega}( x^{\alpha_1}\pa_x\sigma)(t)$ has only the derivative along the $x$-direction. This particular expression of the observation term helps us to glue the observability inequality for the one-dimensional ($x$-direction) degenerate parabolic equation with the Lebeau-Robbiano spectral inequality \eqref{lr1} for the other one (see the proof of \Cref{par obs}). 

On the other hand, if we put the control as in \Cref{fig:region1}, the normal derivative at any point of the form $(1,y)\in \Gamma_3$ is $\nu=\left(\frac{-y}{\sqrt{1+y^2}},\frac{1}{\sqrt{1+y^2}}\right).$ Thus, our approach does not seem effective to apply directly for \eqref{DCP2}. 

To overcome this difficulty, we introduce a change of variable $x\mapsto x+1$ so that we can move the control at the end $x=0$ and exploit the expression of the  outward normal $\nu=(1,0).$ Thus, the system \eqref{DCP2} transforms to
\begin{equation}\label{DCP3}
	\begin{cases}
		\partial_t u=\dv(D {\nabla u}) &  \text{ in } (0,T)\times (-1,0)\times(0,1),\\
		u(t)=\mathbf{1}_{\gamma} q(t)& \text{ in } (0,T)\times \partial \Omega, \\
		u(0)=u_0 & \text{ in } \Omega,
	\end{cases}
\end{equation}
where $\gamma=\{0\}\times \omega\subset \pa \Omega,\, \omega\subset (0,1),$
\begin{equation*}
	 D(x,y)=\begin{pmatrix}
		(1+x)^{\alpha_1}& 0\\
		0& y^{\alpha_2}
	\end{pmatrix},\quad x\in (-1,0),\quad y\in (0,1).
\end{equation*}
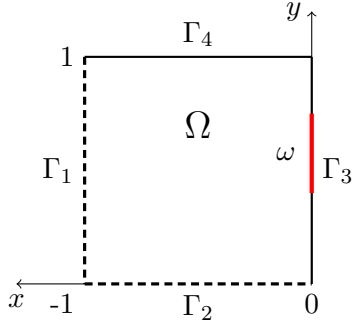
\begin{figure}[h!]
	\centering
	\begin{tikzpicture}[scale=3] 
		\draw[thin, ->]   (0,0)--(-0.3,0)
		node[below] {$x$}; 
		\draw[thin, ->] (1,0) -- (1,1.2) node[left] {$y$} ; 
		
		\draw[thin, -] (0,0) -- (0,0); 
		
		\draw[densely dashed, very thick] (0,0) -- (1,0) node[midway, below] {$\Gamma_2$};
		\draw[thick] (1,0) -- (1,1) node[midway, right] {$\Gamma_3$};
		\draw[thick] (1,1) -- (0,1) node[midway, above] {$\Gamma_4$};
		\draw[densely dashed, very thick] (0,1) -- (0,0) node[midway, left] {$\Gamma_1$};
		
		\draw[ultra thick, red] (1,0.4) -- (1,0.75); 
		\node[right] at (0.8,0.575) {$\omega$}; 
		
		\node[below left] at (0,0) {-1};
		\node[below] at (1,0) {0};
		\node[left] at (0,1) {1};
		\node at (0.5, 0.7) {\Large$\Omega$}; 
	\end{tikzpicture}
	\caption{The domain $\Omega=(-1,0)\times (0,1)$ for equation \ref{DCP3}, with the operator degenerating along the dashed lines $\Gamma_1$ and $\Gamma_2$. The red region, denoted by $\omega$, represents the control set, which is active at the non degenerate boundary.}
	\label{fig:region2}
	\end{figure}
	
Thus, from the techniques developed in the present paper, the controllability of \eqref{DCP3} reduces to establishing the null controllability of the corresponding one-dimensional equation
\begin{equation}\label{oned 2}
	\begin{cases}
		\pa_t w =\pa_{x}((1+x)^{\alpha_1}\pa_{x}w)& \text{ in }   (0,T) \times (-1,0),\\
		w(t,0)=h(t), \quad
		w(t,-1)=0 &\text{ in }   (0,T),\\
		w(0,x)=w_0(x)  & \text{ in }   (-1,0).
	\end{cases}
\end{equation}
Again performing a change of variable $x+1\mapsto x$, the control problem \eqref{oned 2} is transformed into the following
 \begin{equation}\label{oned 3}
	\begin{cases}
		\pa_t w =\pa_{x}(x^{\alpha_1}\pa_{x}w) & \text{ in }   (0,T) \times (0,1),\\
		w(t,1)=h(t), \quad
		w(t,0)=0 &\text{ in }   (0,T),\\
		w(0,x)=w_0(-x)  & \text{ in }   (0,1).
	\end{cases}
\end{equation}

Using the technique of the proof of \Cref{null control 1} for the one-dimensional degenerate parabolic system, it can be proved that system \eqref{oned 3} is null controllable with control cost $Ce^{C/T}.$ This is now enough to lead the null controllability for the equation \eqref{oned 2} and hence the same for \eqref{DCP3} and finally for \eqref{DCP2}. 

We note that a strategy similar to the one described above (but this time changing the role of the $x$- and $y$-coordinates) can be employed to control from the sides $\Gamma_2=[0,1]\times\{0\}$ or $\Gamma_4=[0,1]\times\{1\}$. For brevity, we left the details to the reader. 

\subsection{A 2-$d$ degenerate parabolic equation controlled from a hyperplane}
Let us consider the control system when $(\alpha_1,\alpha_2)\in [0,1)\times [0,1)$
\begin{equation}\label{DCP5}
	\begin{cases}
		\partial_t u=\text{div}\left(D\nabla u\right)+\delta_{x_0}\mathbf{1}_{\omega}(y) f(t,x,y) &  \text{ in } (0,T)\times \Omega,\\
		u(t)=0& \text{ on } (0,T)\times \partial \Omega, \\
		u(0)=u_0 & \text{ in } \Omega,
	\end{cases}
\end{equation}
where $\omega=(a,b)\subseteq (0,1)$ and $x_0\in(0,1)$ (see \Cref{fig:region4}).  
 
\begin{figure}[h!]
	\centering
	\begin{tikzpicture}[scale=3]
		
		\draw[thin, -] (-0.2,0) -- (0,0); 
		\draw[thin, ->] (1,0) -- (1.2,0) node[below] {$x$}; 
		
		\draw[thin, -] (0,-0.2) -- (0,0); 
		\draw[thin, ->] (0,1) -- (0,1.2) node[left] {$y$}; 
		
		\node at (0.7,0.7) {$\Omega$};
		\draw[densely dashed, very thick] (0,0) -- (1,0) node[midway, below] {$\Gamma_2$};
		\draw[thick] (1,0) -- (1,1) node[midway, right] {$\Gamma_3$};
		\draw[thick] (1,1) -- (0,1) node[midway, above] {$\Gamma_4$};;
		\draw[densely dashed, very thick] (0,1) -- (0,0) node[midway, left] {$\Gamma_1$};
		\fill[red] (0,0.3) circle (1pt) node[left] {$a$};
		\fill[red] (0,0.7) circle (1pt) node[left] {$b$};
		\fill[black] (0.35,0) circle (1pt) node[above] {$x_0$};
		
		\draw[red, dashed] (0,0.3) -- (0.35,0.3);
		\draw[red, dashed] (0,0.7) -- (0.35,0.7);
		\draw[ultra thick, red] (0.35,0.3) -- (0.35,0.7);
		
		\draw (0,1) node[left] {$1$};
		\draw (1,0) node[below] {$1$};
		\draw (0,0) node[below left] {$0$};
		\draw (0.54,0.575) node[below left] {$\omega$};
	\end{tikzpicture}
	\caption{The domain $\Omega$ for equation \ref{DCP5}, with the operator degenerating along the dashed lines. The red region, denoted by $\omega$, represents the control set, which is active at the level of the point $x_0$.}
	\label{fig:region4}
\end{figure}
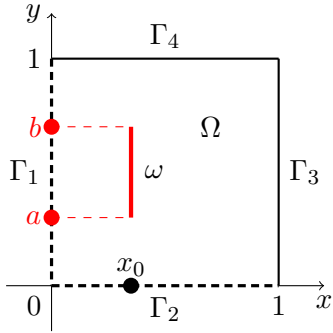
It is worth mentioning that in \cite{AS20}, the authors studied the pointwise controllability of the one-dimensional weakly degenerate $(0\leq\alpha_1<1)$ heat equation given by
\begin{equation}\label{pointwise}
\begin{cases}
	u_t =\pa_x (x^{\alpha_1} \pa_x u) + \delta_{x_0} v(t) & \text{ in } (0,T)\times (0.1), \\
	u(t, 0) = u(t, 1) = 0 & t \in (0, T), \\
	u(0, x) = u_0(x) & x \in (0, 1).
\end{cases}
\end{equation}
 They proved that \eqref{pointwise} is null controllable in time $T>T_0(x_0)$, where
 \begin{equation}\label{time}
 		[0,+\infty]\ni T_0(x_0) := \limsup\limits_{k \to +\infty} \frac{-\log\left( |\phi_{{\alpha_1}, k}(x_0)| \right)}{\lambda_{\alpha_1, k}},
 \end{equation}
  and not controllable if $T<  T_0(x_0).$ 
 
Employing the pointwise controllability result for one-dimensional heat equation (see \cite{D73}) with control cost of the form $Ce^{C/T}$ combined with the Lebeau-Robbiano technique, the author in \cite{S15} studied the controllability of the $N$-dimensional heat equation in a cylindrical domain $\Omega=(0,\pi)\times \Omega^\prime$ where $\Omega^\prime\subset \mathbb R^{N-1}$ by means of control acting through a hyperplane, more precisely, 
 \begin{equation}\label{sam}
 \begin{cases}
 	\partial_t u = \Delta u + \delta_{{x_0}} \mathbf{1}_{\omega} v(t,x,x^\prime) & \text{in } \Omega \times (0, T), \\
 	u(t) = 0 & \text{on } \partial\Omega \times (0, T), \\
 	u(0) = u_0 & \text{in } \Omega.
 \end{cases}
 \end{equation}
  In this spirit, using the techniques of \cite{S15} and the one-dimensional result for degenerate heat equations in \cite{AS20}, one can obtain directly the following controllability result for the 2-$d$ system \eqref{DCP5} when $0\leq\alpha_1<1.$ Let us consider the set
\begin{equation*}\mathcal{S}=\bigg\{\left(\frac{j_{\nu(\alpha_1),k}}{j_{\nu(\alpha_1),n}}\right)^{\frac{1}{\kappa_{\alpha_1}}}: n>k\geq1 \bigg\},
	\end{equation*}
where $j_{\nu(\alpha_1),k}$ is the sequence of the zeros of Bessel functions of the first kind of order $\nu({\alpha_1})$.
\begin{theorem}\label{th:pw1}
	Assume that $x_0 \notin \mathcal S$ and recall 
	$ T_0(x_0) \in [0, +\infty]$ defined in \eqref{time}.
		 Under the assumptions: $x_0 \notin \mathcal S \text{ and } \omega\equiv (0,1),$  we have the following 
		\begin{enumerate}
			\item If $T > T_0(x_0)$, system \eqref{DCP5} is null-controllable at time $T$.
			\item For any $T < T_0(x_0)$, system \eqref{DCP5} is not null-controllable at time $T$.
		\end{enumerate}
	\end{theorem}

\begin{remark}
In this direction, it seems that the same methodology can be adapted without major modifications for the case of strongly degenerate case $(1<\alpha_1<2)$. Nonetheless, a rigorous proof is still needed.
\end{remark}
Recently the authors in \cite{khodja2024new} dealt with the controllability of \eqref{sam} in a two dimensional setting as in \Cref{fig:region4} by choosing $\omega\subsetneq (0,1)$ and some $x_0$ such that $T_0(x_0)>0$. Exploring a suitable biorthogonal construction and then employing the Lebeau-Robbiano strategy, the authors proved the null controllability of \eqref{sam} in time $T>T_0(x_0)$ and they also established the lack of null controllability when $T<T_0(x_0)$. We we mention a possible extension for further research problem regarding the degenerate parabolic equation \eqref{DCP5}.
 \begin{ques}
 	Let us consider some $x_0\in (0,1)$ such that $T_0(x_0)>0,$ where $T_0(x_0)$ is defined in \eqref{time}. Can we construct some control function $f$ such that the system \eqref{DCP5} satisfies $u(T)=0$ for $T>T_0(x_0),$ when $\omega\subsetneq (0,1)?$
 \end{ques}
An immediate interest would be the combination of the Kalman rank condition \eqref{kalman} for our problem and the minimal time as in \Cref{th:pw1}. In more detail, we can state the following problem. Consider the coupled degenerate parabolic system 
\begin{equation}\label{DCP6}
	\begin{cases}
		\partial_t u=\mathcal{I}_n \mathbf{D} u+A u 
		+\delta_{{x_0}} \mathbf{1}_{\omega} Bv(t,x,y) &  \text{ in } (0,T)\times \Omega,\\
		u(t)=0& \text{ on } (0,T)\times \partial \Omega, \\
		u(0)=u_0 & \text{ in } \Omega,
	\end{cases}
\end{equation}
where $ \omega\subset (0,1)$ and $x_0\in(0,1)$.
\begin{ques}
Under the assumptions of \Cref{th:pw1} and the Kalman rank condition \eqref{kalman}, is system \eqref{DCP6} null controllable for some $T>T_0$ only depending on $x_0$? Does null controllability fails if $T<T_0?$
\end{ques}

\subsection{The half-degenerate case in dimension $N\geq 3$}  
So far, we have only studied the null-controllability of degenerate equations in the domain which has the particular form $\Omega=(0,1)\times\cdots\times (0,1)$.
This restriction arises because the available spectral inequality (see \Cref{thm:spec_ineq_degen}) has been established for degenerate operators only in one dimension. However, by focusing on the specific case of a half-degenerate operator, the analysis can be extended to less restrictive higher-dimensional domains. It is worth noting that the case of interior controllability with such half degenerate operators remains an open problem, as highlighted in \cite[Section 5.2]{FA19}.

Let us consider the domain $\Omega=(0,1)\times\Omega^\prime\subset \mathbb R^{N}$ where $\Omega^\prime\subset \mathbb R^{N-1}$ $(N\geq 3)$ is a smooth bounded domain, and let us consider the control system
\begin{equation}\label{eq:sys_dimN2}
\begin{cases}
\partial_t u = (x^{\alpha_1}u_x)+\Delta_{x^\prime}u &\text{in } (0,T)\times\Omega,\\
u(t)=0 &\text{on } (0,T)\times(0,1)\times \partial \Omega^\prime,\\
u(t)=1_{\gamma}q(t) &\text{on } (0,T)\times\{0,1\}\times \Omega^\prime, \\
u(0)=u_0 &\text{in } \Omega,
\end{cases}
\end{equation} 
where $\gamma=\{0\}\times\omega$ with $\omega\subset \Omega^\prime$ and $\Delta_{x^\prime}$ is the $(N-1)$-dimensional Laplace operator in $\Omega^\prime$ (see \Cref{fig:cyl_domain}) and $0\leq \alpha_1<1.$

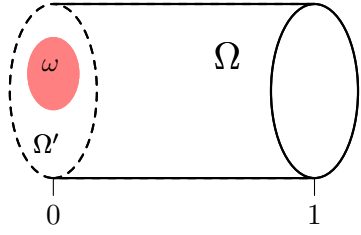
\begin{figure}[htbp!]
\centering
\begin{tikzpicture}[line cap=round,line join=round,scale=1.15]
  \def\a{1} 
  \def\h{0} 
  \def\z{3} 
  
  \foreach\i in {0,\h} 
  {%
    \draw[thick,dashed] (\z,\i+\a) -- (0,\i+\a) arc (90:270:0.5*\a cm and \a cm) -- (\z,\i-\a) ;
    \draw[thick, dashed] (0,\i-\a) arc (-90:90:0.5*\a cm and \a cm);
    \draw[thick] (\z,\i) ellipse (0.5*\a cm and \a cm);
    \draw[thick] (\z,\i+\a) -- (0,\i+\a);
    
    \draw[fill=red!50,draw=none] (0,0.2) ellipse (0.30*\a cm and 0.42\a cm);
    
    \draw[thick] (0,-1) -- (3,-1);
    \node at (2.0, 0.4) {\Large$\Omega$};
    \node[right] at (-0.25,0.3) {\small $\omega$};
    \node[right] at (-0.35,-0.6) {\small $\Omega^\prime$}; 
  }

  \draw[] (0,-\a) -- (0,-1.2) node [below] {$0$};
  \draw[] (\z,-\a) -- (\z,-1.2) node [below] {$1$};
\end{tikzpicture}  
\caption{Example of the domain $\Omega$ for equation \eqref{eq:sys_dimN2} for $N=3$. The operator degenerates only on the dashed face of the cylinder. The red region, denoted by $\omega$, represents the control set, which is active at the degenerate boundary.}
\label{fig:cyl_domain}
\end{figure}
Note that the underlying operator in \eqref{eq:sys_dimN2} only degenerates in the set $\{0\}\times\Omega^\prime$. We also note that similar ideas as developed in \Cref{sec:wp} yield the well-posedness of \eqref{eq:sys_dimN2} and its corresponding adjoint system. 

From the control point of view, since in the $x^\prime$-direction there is no degeneracy, we can follow the approach developed in \Cref{lr sc} and replace the 1-$d$ spectral inequality \eqref{lr1} by the usual one (see \cite{LR})
\begin{equation}\label{eq:spec_ineq}
\sum_{\lambda_j \leq \mu}|a_j|^2 \leq Ce^{C\sqrt{\mu}} \int_{\omega}\left|\sum_{\lambda_j \leq \mu}a_j\Phi_j\right|^2 \rd x^\prime,
\end{equation}
for any any positive constant $\mu$, $\{a_j\}\in \mathbb R$, and where $(\lambda_j,\Phi_j)$ are the corresponding eigenvalues and eigenfunctions of the elliptic problem
\begin{equation*}
-\Delta_{x^\prime} \Phi =\lambda \Phi \quad \text{in $\Omega^\prime$}, \quad \Phi=0 \quad\text{on $\partial \Omega^\prime$}.
\end{equation*}
By direct adaptations, we can establish the null controllability of the system \eqref{eq:sys_dimN2}. In order to state the controllability result, let us mention the associated functional spaces.
Recall the space $H^1_{\alpha}(\Omega)$ for $\alpha=(\alpha_1,0,...,0)$ as defined in \Cref{sec:wp} with $D(x,x')=\text{diag}({x^{\alpha_1},1,...,1}) \in M_{N\times N}(\mathbb{R})$. The spaces
$H^1_{\alpha,0}(\Omega)$ and $H^{-1}_{\alpha}(\Omega)$ are defined in a similar fashion as \Cref{sec:wp}.
\begin{theorem}\label{th_half dg}
	Let $T>0$. For any $u_0\in H^{-1}_{\alpha}(\Omega)$, there exists a control $q\in L^2(0,T;L^2(\pa \Omega))$ such that the system \eqref{eq:sys_dimN2} satisfies $u(T)=0.$
	\end{theorem}

\subsection{The case when at least one ${\alpha_i=1}$ for $i=1,2$}\label{alphaeq1}

In this section, we consider the case where at least one $\alpha_i = 1$, $i = 1, 2$. As noted in \Cref{remark_extensions}, \Cref{main theorem_scalar} remains valid in this situation. However, some modifications are required in the definition of the space $H^1_{\alpha,0}(\Omega)$, which lead to further changes at the levels of well-posedness and control. These adjustments do not fit comfortably within the frameworks presented in Sections \ref{sec:wp} to \ref{lr sc}. As we have emphasized previously, the main 2-$d$ problem reduces to analyzing the boundary control system in 1-$d$. Therefore, we now focus on the necessary changes at this level, while the remainder of the analysis can be carried out in a similar manner. 

Without loss of generality assume $\alpha_1=1.$ Consider the 1-$d$ control system 
\begin{equation}\label{eq:case1}
	\begin{cases}
		\pa_t w =\pa_{x}(x\pa_{x}w)& \text{ in }   (0,T) \times (0,1),\\
	(x\pa_x w)(t,0)=h(t), \, w(t,1)=0 &\text{ in }   (0,T),\\
	w(0,x)=w_0(x)  & \text{ in }   (0,1).
\end{cases}
\end{equation}
Following \cite[Section 2.2]{CMV20}, let us take the space $H_{1,0}^1(0,1)=\{u\in H_{1}^1(0,1)\mid u(1)=0\}$ with $H_1^1(0,1)$ as in \eqref{h1_a} with $\alpha_1=1$. Note that in this case, the Hardy--Poincar\'e inequality (see eq. \eqref{weightespoincare1}) does not hold. Thus we will consider the space $H^1_{1,0}(0,1)$ with the norm 
\begin{equation}\label{norm_alpha1}
	\|u\|_{H^1_{1}(0,1)}^2 := \|u\|_{L^2(0,1)}^2 + \|x^{\frac{1}{2}} \pa_x u\|_{L^2(0,1)}^2,
\end{equation}
and its dual space $H^{-1}_{1}(0,1)$ with respect to the pivot space $L^2(0,1)$.
Thanks to \cite[Proposition 2.13]{CMV20}, the self-adjoint operator $\mathcal A_{1} : \mathcal D(\mathcal A_{1}) \subset L^2(0, 1) \to L^2(0, 1) $ defined by
\[
\begin{cases}
	\mathcal A_{1} u := \pa_x(x \pa_x u), \text{ for all } u \in \mathcal D(\mathcal A_{1}), \\
	\mathcal D(\mathcal A_{1}) := \left\{ u \in H^1_{1,0}(0, 1) \mid x \pa_x u \in H^1(0, 1), (x u)\in H^1_0(0,1) \right\} = H^2_{1}(0, 1) \cap H^1_{1,0}(0, 1),
\end{cases}
\]  is
the infinitesimal generator of a strongly continuous semigroup of contractions on
$L^2(0, 1)$ and satisfies the following.
\begin{proposition}
	The admissible eigenvalues $\lambda$ for problem $-\mathcal A_{1} \phi=\lambda \phi$
	are given by 
	\begin{equation*}
		\lambda_{1,k}=\frac{1}{4}j_{0, k}^2, \quad k\geq 1,
	\end{equation*}
	and the corresponding normalized eigenfunctions are 
	\begin{equation*}
		\phi_{1,k}(x)=\frac{1}{|J'_{0}(j_{0,k})|} J_{0}(j_{0,k}x^{\frac{1}{2}}), \quad x\in (0,1), \quad k\geq 1,
	\end{equation*}
	where we recall the Bessel function $J_0$ defined in \eqref{exp_of_bessel} and where $j_{0,k}$ are the positive zeros of $J_0$. Moreover the family $\{\phi_{1,  k}\}_{k\geq 1}$ forms an orthonormal basis in $L^2(0,1).$
\end{proposition}
It is worth noting that the boundary admissibility condition \eqref{ad strng1d} in \Cref{lm adm} can be adapted for $\alpha_1=1$ from \cite[Proposition 3.10]{galo2024boundary} with taking into account the space $H_{1,0}^1(0,1)$ equipped with the norm \eqref{norm_alpha1}. Following the results in \Cref{sub_wp}, we can establish the well-posedness of the solutions (by transposition) of \eqref{eq:case1}.

Let us provide a brief sketch of the null controllability result for \eqref{eq:case1}
\begin{theorem}\label{null control 1alpha}
	Let $T>0$. Then for every $w_0\in H^{-1 }_{1}(0,1)$,  there exists a control $h\in L^2(0,T)$ such that the system \eqref{eq:case1} satisfies $w(T)=0$. Moreover, the control satisfies the following estimate
	\begin{equation*}
		\norm{h}_{L^2(0,T)}\leq Ce^{\frac{C}{T}}\norm{w_0}_{H^{-1 }_{1}(0,1)},
	\end{equation*} 
	for some positive constant $C$ which is independent of $T.$
\end{theorem}
\begin{proof}
	Thanks to \Cref{obs strong} for $\alpha_1=1$, we have \begin{equation}\label{asymp_obs1}\lim_{x\to 0^+}\phi_{ 1, k}(x)=\frac{1}{\left|J'_{0}(j_{0,k})\right|}\sim_{k\to +\infty} ({j_{0,k}})^{\frac{1}{2}}.
	\end{equation}
	Following the proof of \Cref{null control 1} and using the above expression, we can say that the controllability of \eqref{eq:case1} is equivalent to find a function $p\in L^2(0,T)$ satisfying
 	\begin{equation}\label{moment_case1}
 		\int_{-\frac{T}{2}}^{\frac{T}{2}} e^{-\lambda_{1, k} t} \, p\left(t+\frac{T}{2}\right) \rd t={|J'_{0}(j_{0,k})|} {e^{-\lambda_{1, k}\frac{T}{2}}} \left|\ip{w_0}{\phi_{1, k}}\right|.
\end{equation}
We define $h(t)=p(T-t)$ as the control for \eqref{eq:case1}.
It is easy to check from \Cref{prop:verify} with $\alpha_1=1,$ that the collection $\Lambda=\{\lambda_{1,k}\}_{k\in \N}$ satisfies the hypothesis of \Cref{d1} and hence we can find a family of functions $\{\Psi_{k}\}_{k\in \N}\in L^2\left(-\frac{T}{2},\frac{T}{2}\right)$ that are biorthogonal to  
$e_{k}=e^{-\lambda_{1,k} t}.$ With this, if we define the function $p$ as
\begin{equation}\label{pform}
	p(t)=\sum_{k=1}^{\infty}d_k \Psi_k\left(t-\frac{T}{2}\right),
\end{equation}
where $d_k={|J'_{0}(j_{0,k})|} {e^{-\lambda_{1, k}\frac{T}{2}}} \left|\ip{w_0}{\phi_{1, k}}\right|,$ then it satisfies the moment problem \eqref{moment_case1}. Using the fact that $\norm{\phi_{ 1, k}}_{H^1_{1,0}(0,1)}\leq C (1+{j_{0,k}})$ for some $C>0$  and estimate \eqref{asymp_obs1}, we have 
\begin{equation}\label{bounddk}
|d_k|\leq  C ({j_{0,k}})^{\frac{1}{2}}{e^{-\lambda_{1, k} T}} \norm{w_0}_{H^{-1}_{ 1}(0,1)}.
\end{equation} 
Putting together estimate \eqref{cost} in \Cref{d1} for the biorthogonal family $\{\Psi_k\}_{k\in \N}$ with \eqref{pform} and \eqref{bounddk}, one can get the required control cost.
\end{proof}
Thanks to \Cref{null control 1alpha} and the spectral inequality \eqref{lr} which holds for all $\alpha_2\in [0,2)$, following the same reasoning as in \Cref{lr sc}, we can establish the controllability of the corresponding two-dimensional model in the space  $H^{-1}_{\alpha}(\Omega)$. Here $H^{-1}_{\alpha}(\Omega)$ denotes the dual of $H^1_{\alpha,0}(\Omega)$ with respect to the pivot space $L^2(\Omega).$ Since we do not have Hardy--Poincar\'e inequality (see eq. \eqref{hardy} in \Cref{lm hardy}), we consider the space $H^1_{\alpha,0}(\Omega)$ endowed the following norm of $H^1_{\alpha}(\Omega)$, as defined in \Cref{sec:wp}
\begin{equation*}\norm{u}^2_{H^1_{\alpha}(\Omega)}=\int_{\Omega}|u|^2+\int_{\Omega}|D^{1/2}\nabla u|^2=\int_{\Omega}|u|^2+\int_{\Omega}\left(x^{\alpha_1}|\pa_xu|^2+y^{\alpha_2}|\pa_yu|^2\right).\end{equation*} 
The entire argument of this paper applies to this particular case, we only need to consider the above norm on $H^1_{\alpha,0}(\Omega)$ at each step accordingly. 
%

\subsection{More general degenerate coefficients}

A natural question that arises in this context is whether we can control the following system 
\begin{equation}\label{DCP_general}
	\begin{cases}
		\partial_t u=\text{div}\left(D\nabla u\right) &  \text{ in } (0,T)\times \Omega,\\
	u(t)=\mathbf{1}_{\Gamma} q(t) & \text{ in } (0,T)\times \partial \Omega, \\ 
	u(0)=u_0, & \text{ in } \Omega,
	\end{cases}
\end{equation}
where the matrix function $D: \overline\Omega\mapsto \mc M_{2\times 2}(\mathbb{R})$ is given by the more general expression
\begin{equation*}
	D(x,y)=\begin{pmatrix}
		a_1(x)& 0\\
		0& a_2(y)
	\end{pmatrix},
\end{equation*}
where $a_1,a_2\in C^0[0,1]$ such that $a_i>0$ and $a_i(0)=a_i(1)=0$, $i=1,2$ (see e.g. \cite{CFR07}). In this case, the controllability problem of \eqref{DCP_general} boils down to the controllability of the 1-d problem 
\begin{equation}
\begin{cases}
	u_t =\pa_x (a_1(x) \pa_x u) & \text{ in } (0,T)\times (0.1), \\
	u(t, 0) = q(t), \quad u(t, 1) = 0 & t \in (0, T), \\
	u(0, x) = u_0(x) & x \in (0, 1),
\end{cases}
\end{equation}
and a spectral inequality of the type
\begin{equation}\label{lr_gen}
		\sum_{\lambda_j \leq \mu} |a_j|^2 \leq C e^{C\sqrt{\mu}} 
		\int_{\omega} \left| \sum_{\lambda_j \leq \mu} a_j \phi_j \right|^2 \, \rd{x}
	\end{equation}
	for any $\{a_j\} \in \mathbb{R}$ and any $\mu > 0$, where $\{\phi_j,\lambda_j\}$ satisfy the eigenvalue problem
\begin{equation}\label{eigen_gen}
	\begin{cases}
		-(a_2(\cdot)\phi')'(y)=\lambda \phi(y) \qquad  y\in (0,1),\\
			\phi(0)=\phi(1)=0.
	\end{cases}
\end{equation}
However, as far as we know, inequality \eqref{lr_gen} is not available in the literature and therefore our method is not directly applicable. We believe that this is an interesting open problem that deserves further attention. 

In a similar direction, another interesting question is to study the controllability of the degenerate system in non-divergence form
\begin{equation}\label{DCP_nondiv}
	\begin{cases}
		\partial_t u = A \Delta u + B \cdot \nabla u & \text{ in } (0,T)\times \Omega,\\
		u(t) = \mathbf{1}_{\Gamma} q(t) & \text{ on } (0,T)\times \partial \Omega,\\ 
		u(0) = u_0 & \text{ in } \Omega,
	\end{cases}
\end{equation}
where \( A = A(x,y) \) is a suitable function that degenerates at the boundary of \( \Omega \), and \( B = (B_1(x,y), B_2(x,y)) \) is a nonzero vector field. Systems of this type arise in the study of population genetics (see \cite[Section 1.4]{cannarsa2016global}). To the best of our knowledge, there are very few works studying the controllability of systems like \eqref{DCP_nondiv}, and all existing results are restricted to the one-dimensional case, see, for instance, \cite{CFR08,Fra16,AFI25}. As in the previous open problem, the main difficulty to extend the controllability results to \eqref{DCP_nondiv} with our method is the lack of an appropriate spectral inequality for the underlying differential operator. These problems offer promising avenues for further research.

\appendix
\section{Appendix}\label{app}
\subsection{$H^1_{\alpha,0}(\Omega)$ norm computation}\label{sec:norm_comp}
Thanks to \Cref{exp of u}, we have that	any function $u\in H^{1 }_{ \alpha,0}(\Omega)$ has the following representation:
\begin{equation*}
	u=\sum\limits_{j=1}^{\infty}\ip{u}{\phi_{\alpha_2,j}}_{L^2(0,1)}\phi_{\alpha_2,j}.
\end{equation*} Further we have the following
\begin{align}\label{norm}
	\norm{u}^2_{H^{1 }_{ \alpha,0}(\Omega)}=\sum_{j=1}^{\infty}\norm{\ip{u}{\phi_{\alpha_2,j}}_{L^2(0,1)}}^2_{H^1_{\alpha_1,0}(0,1)}+\sum_{j=1}^{\infty} \lambda_{\alpha_2,j} \norm{\ip{u}{\phi_{\alpha_2,j}}_{L^2(0,1)}}^2_{L^2(0,1)}.
\end{align}
Indeed, for simplicity let us denote $a_j(x):=\ip{u(x,\cdot)}{\phi_{\alpha_2,j}(\cdot)}_{L^2(0,1)}.$ Then 
$u(x, y) = \sum_{j=1}^\infty a_j(x) \phi_{\alpha_2, j}(y).$ 
As $\| u \|^2_{H^1_{\alpha, 0}(\Omega)} = \int_\Omega x^{\alpha_1} |u_x|^2 + y^{\alpha_2} |u_y|^2$, let us explore the following representation
\begin{align*}
\| u \|^2_{H^1_{\alpha, 0}(\Omega)} =& \int_\Omega x^{\alpha_1} |u_x|^2 + y^{\alpha_2} |u_y|^2=\langle x^{\alpha_1/2} u_x, x^{\alpha_1/2} u_x \rangle_{L^2(\Omega)}
	+ \langle y^{\alpha_2/2} u_y, y^{\alpha_2/2} u_y \rangle_{L^2(\Omega)}\\
	=& \left\langle \sum_{j=1}^\infty x^{\alpha_1/2} a_j'(x) \phi_{\alpha_2, j}(y), 
		\sum_{j=1}^\infty x^{\alpha_1/2} a_j'(x) \phi_{\alpha_2, j}(y) \right\rangle_{L^2(\Omega)}\\
		&\hspace{1cm}+\left\langle \sum_{j=1}^\infty  a_j(x)y^{\alpha_2/2}  \phi'_{\alpha_2, j}(y), 
			\sum_{j=1}^\infty a_j(x) y^{\alpha_2/2}  \phi'_{\alpha_2, j}(y) \right\rangle_{L^2(\Omega)}\\
			&= \sum_{j=1}^\infty \| x^{\alpha_1/2} a_j'(x) \|^2_{L^2(0,1)} \| \phi_{\alpha_2, j} \|^2_{L^2(0,1)} 
			+ \sum_{j=1}^\infty \| a_j\|_{L^2(0,1)} \| y^{\alpha_2/2}\phi'_{\alpha_2, j} \|^2_{L^2(0,1)}\\
			&= \sum_{j=1}^\infty \| a_j\|^2_{H^1_{\alpha_1,0}(0,1)} 
			+ \sum_{j=1}^\infty  \lambda_{\alpha_2,j}\| a_j\|^2_{L^2(0,1)}, 
	\end{align*}
here we have used that $\{\phi_{\alpha_2,j}\}_{j\in \N}$ is orthonormal and  the fact $\| y^{\alpha_2/2}\phi'_{\alpha_2, j}(y) \|^2_{L^2(0,1)}=\lambda_{\alpha_2,j}$, thanks to the eigenvalue problem \eqref{eigen eqn2}.

\subsection{Isomorphism between $H^1_{\alpha,0}(\Omega)$ and $H^{-1}_{\alpha}(\Omega)$}\label{iso app} 
Let us first consider the following degenerate elliptic problem in $\Omega=(0,1)\times (0,1),$
\begin{align*}
	-\text{div}(D\nabla u)=f\in H^{-1}_{\alpha}(\Omega) \text{ with homogeneous boundary conditions.}
\end{align*}
	Consider the following  weak formulation
		\begin{equation*}
	\int_{\Omega}   \big( x^{\alpha_1}\partial_x u \, \partial_x v +y^{\alpha_2} \partial_y u \, \partial_y v \big) \, \rd x \, \rd y = \langle f, v \rangle_{H^{-1}_{\alpha}(\Omega), H^1_{\alpha, 0}(\Omega)} \quad \forall \, u, v \in H^1_{\alpha, 0}(\Omega),
		\end{equation*}
	where $f \in H^{-1}_{\alpha}(\Omega)$. The associated bilinear form is
		\begin{equation*}
	a_{\alpha}(u, v) = \int_{\Omega}  \big(x^{\alpha_1} \partial_x u \, \partial_x v +  y^{\alpha_2}\partial_y u \, \partial_y v \big) \, \rd x \, \rd y.
		\end{equation*}
For any $u \in H^1_{\alpha, 0}(\Omega)$, define a functional $\Lambda_u \in H^{-1}_{\alpha}(\Omega)$ by
	\begin{equation}\label{app_fn}
	\langle \Lambda_u, v \rangle = a_{\alpha}(u, v) = \int_{\Omega}\big( x^{\alpha_1}\partial_x u \, \partial_x v +y^{\alpha_2} \partial_y u \, \partial_y v \big) \, \rd x \, \rd y,\quad \forall \,  v \in H^1_{\alpha, 0}(\Omega).
	\end{equation}
	If \(\Lambda_u = 0\), then
	\begin{equation*}a_{\alpha}(u, v) = 0 \quad \forall \, v \in H^1_{\alpha, 0}(\Omega).
		\end{equation*}
	Choosing \(v = u\), we find
	\begin{equation*}
	a_{\alpha}(u, u) = \int_{\Omega}   \big(x^{\alpha_1} |\partial_x u|^2 +y^{\alpha_2} |\partial_y u|^2 \big) \, \rd x \, \rd y = 0.
		\end{equation*}
	This implies $u = 0$ by Hardy--Poincar\'e inequality \eqref{hardy}. Thus, the map $u\mapsto \Lambda_u$ is injective.

	For any $g \in H^{-1}_{\alpha}(\Omega)$, the Lax-Milgram theorem (as $a_{\alpha}$ is bilinear continuous and coercive) guarantees the existence of a unique $u \in H^1_{\alpha, 0}(\Omega)$ such that
	\begin{equation*}
	a_{\alpha}(u, v) = \langle g, v \rangle \quad \forall \, v \in H^1_{\alpha, 0}(\Omega).
	\end{equation*}
	This proves the map $u\mapsto \Lambda_u$ surjectivity.
To verify norm preservation, compute the dual norm of $\Lambda_u.$
	\begin{equation*}
	\|\Lambda_u\|_{H^{-1}_{\alpha}(\Omega)} = \sup_{v \in H^1_{\alpha, 0}(\Omega) \setminus \{0\}} \frac{\langle \Lambda_u, v \rangle}{\|v\|_{H^1_{\alpha,0}(\Omega)}} = \sup_{v \in H^1_{\alpha, 0}(\Omega) \setminus \{0\}} \frac{a_{\alpha}(u, v)}{\|v\|_{H^1_{\alpha,0}(\Omega)}}.
	\end{equation*}
	Using the Cauchy-Schwarz inequality for \(a_{\alpha}(u, v)\)
	\begin{equation*}
	a_{\al}(u, v) \leq \left( \int_{\Omega} \big( x^{\alpha_1}  |\partial_x u|^2 + y^{\alpha_2}|\partial_y u|^2 \big) \, \rd x \, dy \right)^{1/2} \left( \int_{\Omega} \big(  x^{\alpha_1} |\partial_x v|^2 +y^{\alpha_2} |\partial_y v|^2 \big) \, \rd x \, \rd y \right)^{1/2}.
	\end{equation*}
	Thus we have,
	\begin{equation*}
	\|\Lambda_u\|_{H^{-1}_{\alpha}(\Omega)} \leq \|u\|_{H^1_{\alpha,0}(\Omega)}.
	\end{equation*}
	Equality holds for \(v = u\), so
	\begin{equation*}
	\|\Lambda_u\|_{H^{-1}_{\alpha}(\Omega)} = \|u\|_{H^1_{\alpha,0}(\Omega)}.
	\end{equation*}
Using the definition \eqref{app_fn} of the functional $\Lambda_u$, we obtain
\begin{equation}\label{iso}
	\norm{-\text{div}(D\nabla u)}_{{H^{-1}_{\alpha}}(\Omega)}=\norm{u}_{H^1_{\alpha,0}(\Omega)}.
\end{equation}
\subsection{Verification of boundary admissibility conditions}\label{sec:ad}
We start this section by recalling some well-known fact regarding the observation acting on the eigenfunctions of the corresponding weak and strong degenerate 1-$d$ operators.
\begin{lemma}[Lemma 6.1, \cite{cannarsa2017cost}]\label{obs weak}
	For $0\leq \alpha_1<1,$ the eigenfunctions $\{\phi_{\alpha_1,k}\}$ satisfy the following
	\begin{align*}
		(x^{ \alpha_1}\pa_x \phi_{ \alpha_1, k})|_{x=0}=\frac{(1-\alpha_1)\sqrt{(2- \alpha_1)}(j_{\nu({ \alpha_1 }),k})^{\nu(\alpha_1)} }{2^{\nu(\alpha_1)} \Gamma(\nu(\alpha_1)+1)\left|J'_{\nu({ \alpha_1})}(j_{\nu({ \alpha_1 }),k})\right|},
	\end{align*}
and also we have 
\begin{align}\label{est obs weak}
(x^{ \alpha_1}\pa_x \phi_{ \alpha_1, k})|_{x=0}\sim_{k\to +\infty} \frac{(1-\alpha_1)\sqrt{(2- \alpha_1)}}{2^{\nu(\alpha_1)} \Gamma(\nu(\alpha_1)+1)}(j_{\nu({ \alpha_1 }),k})^{\nu(\alpha_1)+1/2}.
\end{align}
\end{lemma}
\begin{lemma}\label{obs strong}
	For $1\leq \alpha_1<2,$ the eigenfunctions $\{\phi_{\alpha_1,k}\}$ satisfy the following expression
	\begin{align*}
		( \phi_{ \alpha_1, k})|_{x=0}=\frac{\sqrt{2\kappa_{\alpha_1}}(j_{\nu({ \alpha_1 }),k})^{\nu(\alpha_1)} }{2^{\nu(\alpha_1)} \Gamma(\nu(\alpha_1)+1)\left|J'_{\nu({ \alpha_1})}(j_{\nu({ \alpha_1 }),k})\right|},
	\end{align*}
	and also we have 
	\begin{align}\label{est obs str}
		( \phi_{ \alpha_1, k})|_{x=0}\sim_{k\to +\infty} \frac{\sqrt{2\kappa_{\alpha_1}}}{2^{\nu(\alpha_1)} \Gamma(\nu(\alpha_1)+1)}(j_{\nu({ \alpha_1 }),k})^{\nu(\alpha_1)+1/2}.
	\end{align}
\end{lemma}
\begin{proof}
	Thanks to the expression of the eigenfunctions  (see eq. \eqref{eigenfn} in \Cref{eigenelement}) and the following asymptotic behavior of the Bessel function (see \cite[9.1.7, p. 360]{AM1992}), \begin{equation*}
	\text{for } \nu \geq 0 \text{ and } 0 < x \leq \sqrt{\nu + 1},\quad 	J_\nu(x) \sim \frac{1}{\Gamma(\nu + 1)} \left(\frac{x}{2}\right)^\nu \quad \text{as } x \to 0^+,
	\end{equation*}  it follows that,  \begin{equation*}\lim_{x\to 0^+}\phi_{ \alpha_1, k}(x)=\frac{\sqrt{2\kappa_{\alpha_1}}(j_{\nu({ \alpha_1 }),k})^{\nu(\alpha_1)} }{2^{\nu(\alpha_1)} \Gamma(\nu(\alpha_1)+1)\left|J'_{\nu({ \alpha_1})}(j_{\nu({ \alpha_1 }),k})\right|}.
\end{equation*}
	Furthermore, noting the fact $|J'_{\nu({ \alpha_1})}(j_{\nu({ \alpha_1 }),k})|\sim_{k\to +\infty} \frac{1}{\sqrt{j_{\nu({ \alpha_1 }),k}}}$	(see Lemma A.4 of \cite{GL16}), one can readily get \eqref{est obs str}.
\end{proof}
 Thanks to \cite[Proposition 2.9]{galo2023boundary} and \cite[Proposition 3.10]{galo2024boundary}, one can verify the following boundary admissibility conditions associated to the 1-$d$ degenerate control problem which essentially lead to the admissibility conditions in 2-dimensional control problem.
\begin{lemma}\label{lm adm}
		For any $T>0$, the following boundary admissibility conditions hold
	\begin{align}
		\label{ad weak1d}\int_{0}^{T}\left| ( x^{\alpha_1} \pa_x\sigma^j(t))(0)\right|_{\cplx^n}^2 \rd t\leq C\left\| \sigma^j_T\right\|^2_{{H^{1 }_{ \alpha_1,0}(0,1)^n}} \quad \text{ if } 0\leq \alpha_1<1,\\
		\label{ad strng1d}	\int_{0}^{T}\left| (  \sigma^j(t))(0)\right|_{\cplx^n}^2 \rd t\leq C\left\| \sigma^j_T\right\|^2_{{H^{1 }_{ \alpha_1,0}(0,1)^n}}\quad  \text{ if } 1<\alpha_1<2,
	\end{align}
	where $\sigma^j$ is the solution of the adjoint system \eqref{adj2} with $\sigma^j(T)=\sigma^j_T$.
\end{lemma}
\begin{proposition}\label{ad}
	For any $T>0$, the following boundary admissibility conditions hold
	\begin{align}
	\label{ad weak}	\int_{0}^{T}  \left\|\mathbf{1}_{\gamma} B^*( \mathbf{P} \sigma(t))\right\|_{L^2(\pa \Omega)^m}^2 \rd t \leq  C \norm{\sigma_T}_{{H^{1}_{\alpha,0}}(\Omega)^n}^2 \quad \text{ if } 0\leq \alpha_1<1,\\
	\label{ad strng}	\int_{0}^{T}\| \mathbf{1}_{\gamma}B^*  \sigma(t)\|_{L^2(\pa \Omega)^m}^2 \rd t \leq  C \norm{\sigma_T}_{{H^{1}_{\alpha,0}}(\Omega)^n}^2 \quad \text{ if } 1<\alpha_1<2,
	\end{align}
where $\sigma$ is the solution of the adjoint system \eqref{adjintro}--\eqref{bd2} with $\sigma(T)=\sigma_T \in {H^{1}_{\alpha,0}}(\Omega)^n$.
\end{proposition}
\begin{proof}
		Let us first assume that $\sigma_T=\sum\limits_{j=1}^{\infty}\sigma^{j}_T(x)\phi_{\alpha_2,j}(y)\in {H^{1}_{\alpha,0}}(\Omega)^n$ with $\sigma^j_T\in H^{1 }_{ \alpha_1,0}(0,1)^n.$ 
	Thus if we write 
	\begin{equation*}
		\sigma(t,x,y)=\sum\limits_{j=1}^{\infty}\sigma^j(t,x)\phi_{\alpha_2,j}(y),
	\end{equation*}
	then $\sigma^j$ satisfy the following
	\begin{equation*}
		\begin{cases}
			\partial_t \sigma^j+\pa_{x}(x^{\alpha_1}\pa_{x}\sigma^j)-\lambda_{\alpha_2,j}\sigma^j+A^*\sigma^j=0 &  t\in (0,T), \; x\in (0,1),\\
			\begin{cases}\sigma^j(t,0)=0 & \text{ if } 0\leq \alpha<1\\
				(x^{\alpha_1}\pa_x \sigma^j)(t,0)=0 & \text{ if } 1 < \alpha<2
			\end{cases} &\text{ in }   (0,T),\\
			\sigma^j(t,1)=0 & t\in (0,T),\\
			\sigma^j(T,x)=\sigma^j_T(x) & x\in (0,1).
		\end{cases}
	\end{equation*}
 With the boundary admissibility \eqref{ad weak1d} for the above 1-$d$ degenerate system,
let us calculate the following norm for the square domain $\Omega$
\begin{align*}
\int_{0}^{T}	\left\|\mathbf{1}_{\gamma} B^*( \mathbf{P} \sigma(t))\right\|_{L^2(\pa \Omega)^m}^2 \rd t& \leq C \int_{0}^{T}\left\|\mathbf{1}_{\gamma} ( \mathbf{P} \sigma(t))\right\|_{L^2(\pa \Omega)^n}^2 \rd t,\, \text{ as } B^*\in\mathcal{L}(\cplx^n;\cplx^m)\\
	&=\int_{0}^{T}\left\|\mathbf{1}_{\gamma} ( x^{\alpha_1} \pa_x\sigma(t))\right\|_{L^2(\pa \Omega)^n}^2 \rd t\\
	&\leq C \int_{0}^{T}\left\| ( x^{\alpha_1} \pa_x\sigma(t))(0,\cdot)\right\|_{L^2(0,1)^n}^2 \rd t, \, \text{ using } \omega\subset (0,1)\\
	&= C \int_{0}^{T}\sum_{j=1}^{\infty}\left| ( x^{\alpha_1} \pa_x\sigma^j(t))(0)\right|_{\cplx^n}^2 \rd t, \text{ using orthonormality of } \{\phi_{\alpha_2,j}\}\\
	&\leq C \sum_{j=1}^{\infty} \int_{0}^{T}\left| ( x^{\alpha_1} \pa_x\sigma^j(t))(0)\right|_{\cplx^n}^2 \rd t.
	 \end{align*}
Next, using the 1-$d$ admissibility condition \eqref{ad weak1d}, we further obtain
 \begin{align*}
 	\int_{0}^{T}	\left\|\mathbf{1}_{\gamma} B^*( \mathbf{P} \sigma(t))\right\|_{L^2(\pa \Omega)^m}^2 \rd t& \leq C \sum_{j=1}^{\infty} \left\| \sigma^j_T\right\|^2_{{H^{1 }_{ \alpha_1,0}(0,1)^n}} \\
 	& \leq C \left(\sum_{j=1}^{\infty} \left\| \sigma^j_T\right\|^2_{H^{1 }_{ \alpha_1,0}(0,1)^n}+\sum_{j=1}^{\infty} \lambda_{\alpha_2,j} \left\| \sigma^j_T\right\|^2_{L^2(0,1)^n}\right)\\
 	&=C \|\sigma_T\|^2_{H^1_{\alpha,0}(\Omega)^n}.
 \end{align*}
Similarly, using \eqref{ad strng1d} in \Cref{lm adm}, one can establish \eqref{ad strng}.
	\end{proof}

\bibliographystyle{alpha}
\bibliography{biblio}

\medskip
\medskip
\medskip
\medskip

\begin{flushleft}

\textbf{Víctor Hernández-Santamaría, Subrata Majumdar and Luz de Teresa}\\
Instituto de Matemáticas\\
Universidad Nacional Autónoma de México \\
Circuito Exterior, Ciudad Universitaria\\
04510 Coyoacán, Ciudad de México, México\\
\texttt{victor.santamaria@im.unam.mx}\\
\texttt{subrata.majumdar@im.unam.mx}\\
\texttt{ldeteresa@im.unam.mx}
\end{flushleft}

\end{document}